\documentclass[11pt,reqno,english]{article}

\usepackage[utf8]{inputenc}  
\usepackage{dsfont}
\usepackage{xcolor}

\usepackage{enumerate}          
\usepackage{amssymb,amsmath,amsfonts,amsthm,fullpage}            
\usepackage{pdfsync}
\usepackage{url}
\usepackage{hyperref}
\usepackage{anysize}
\usepackage{mathrsfs}
\usepackage{graphicx,float,color}
\usepackage{epsfig}
\usepackage{tikz}
\usepackage{comment}
\usepackage{authblk}

\numberwithin{equation}{section}

\begin{document}
\theoremstyle{plain}
\newtheorem{theorem}{Theorem}[section]
\newtheorem{definition}[theorem]{Definition}
\newtheorem{proposition}[theorem]{Proposition}
\newtheorem{lemma}[theorem]{Lemma}
\newtheorem{corollary}[theorem]{Corollary}
\newtheorem{conjecture}[theorem]{Conjecture}
\newtheorem{remark}[theorem]{Remark}
\theoremstyle{remark}
\newtheorem{example}[theorem]{Example}
\errorcontextlines=0

\renewcommand{\P}{\mathcal{P}}
\newcommand{\Id}{\text{Id}}
\newcommand{\R}{\mathbb{R}}
\newcommand{\M}{\mathbf{H}^m}
\newcommand{\Op}{{\rm Op}}
\newcommand{\J}{\mathcal{J}}
\newcommand{\Z}{\mathbb{Z}}
\newcommand{\N}{\mathbb{N}}
\newcommand{\Bad}{{\rm Bad}}
\newcommand{\supp}{\text{supp}}
\newcommand{\ord}{\text{ord}}
\newcommand{\e}{\varepsilon}
\newcommand{\bqn}{\begin{equation}}
\newcommand{\eqn}{\end{equation}}
\newcommand{\al}{\alpha}
\newcommand{\tor}{(2\mathbb{T})^{2}}
\newcommand{\er}{{\overrightarrow{e_{r}}}}
\newcommand{\et}{{\overrightarrow{e_{\theta}}}}
\newcommand{\red}{\color{red}}
\newcommand{\black}{\color{black}}
\renewcommand{\i}{\iota}
\newcommand{\todo}[1]{$\clubsuit$ {\tt #1}}
\newcommand{\B}{\mathcal{B}}
\newcommand{\F}{\mathbb{F}}
\newcommand{\E}{\mathbb{E}}
\newcommand{\G}{G}
\newcommand{\Cay}{{\rm Cay}}
\newcommand{\Sch}{{\rm Sch}}
\newcommand{\dd}{{\rm d}}
\newcommand{\comp}{{\rm c}}
\newcommand{\Tr}{{\rm Tr}}
\newcommand{\IND}{{\mathbf{1}}}
\newcommand{\cGr}{{\mathcal{G}^\bullet}}
\newcommand{\cT}{\mathcal{T}}
\newcommand{\tx}{\tilde{x}}
\newcommand{\ty}{\tilde{y}}
\newcommand{\tR}{\widetilde{R}}
\newcommand{\tr}{{\rm Tr}}
\newcommand{\Var}{{\rm Var}}    
\newcommand{\Ind}{{\rm Ind}}
\newcommand{\cX}{\mathcal{X}}

\newcommand{\FIX}{\mathrm{Fix}}
\newcommand{\INT}[1]{\left[ #1 \right]}

\setcounter{tocdepth}{1}

\title{Quantum mixing on large Schreier graphs}

\author{Charles Bordenave\thanks{Aix-Marseille Université, CNRS, I2M, Marseille, France. Email: charles.bordenave@univ-amu.fr}\,, \; Cyril Letrouit\thanks{CNRS \& Université Paris-Saclay, LMO, Orsay, France. Email: cyril.letrouit@universite-paris-saclay.fr}\quad and \; Mostafa Sabri\thanks{Science Division, New York University Abu Dhabi, Saadiyat Island, UAE. Email: mostafa.sabri@nyu.edu}}

\maketitle

\begin{abstract}

We prove quantum ergodicity and quantum mixing for sequences of finite Schreier graphs converging to an infinite Cayley graph whose adjacency operator has absolutely continuous spectrum. Under Benjamini–Schramm convergence (or strong convergence in distribution), we show that correlations between eigenvectors at distinct energies  vanish asymptotically when tested against a broad class of local observables.

Our results apply to all orthonormal eigenbases and do not require tree-like structure or periodicity of the limiting graph, unlike previous approaches based on non-backtracking operators or Floquet theory. The proof introduces a new framework for quantum ergodicity, based on trace identities, resolvent approximations and representation-theoretic techniques and extends to certain families of non-regular graphs.

We illustrate the assumptions and consequences of our theorems on Schreier graphs arising from free products of groups, right-angled Coxeter groups and lifts of a fixed base graph. \color{black}
\end{abstract}
%Quantum ergodicity describes the delocalization of most eigenfunctions of Laplace-type operators on graphs or manifolds exhibiting chaotic classical dynamics. Quantum mixing is a stronger notion, additionally controlling correlations between eigenfunctions at different energy levels. In this work, we study families of finite Schreier graphs that converge to an infinite Cayley graph and establish quantum mixing under the assumption that the limiting Cayley graph has absolutely continuous spectrum. The convergence of Schreier graphs is understood in the Benjamini–Schramm sense or in the sense of strong convergence in distribution. Our proofs rely on a new approach to quantum ergodicity, based on trace computations,  resolvent approximations and representation theory. We illustrate our assumptions on several examples and provide applications to Schreier graphs associated with free products of groups and right-angled Coxeter groups.

\tableofcontents

\section{Introduction and main results}
Eigenvalue statistics of quantum systems are expected to reflect the dynamical properties of their classical counterparts. Berry and Tabor conjectured that for integrable systems, eigenvalues tend to behave universally like independent random variables (Poisson statistics). In contrast, for chaotic systems, Bohigas, Giannoni and Schmit  conjectured that eigenvalues tend to follow random matrix theory statistics (GOE, GUE, etc.). 

Universal phenomena are also expected at the level of \emph{eigenstates}. For instance, it is known that on a compact Riemannian manifold with classically ergodic (or chaotic) geodesic flow, most high-frequency eigenfunctions of the Laplace-Beltrami operator become delocalized -- that is, they spread uniformly over the manifold as the corresponding eigenvalue tends to $+\infty$. This is the \emph{quantum ergodicity} phenomenon. Its first precise mathematical statement is due to Shnirelman \cite{zbMATH03507706}, and rigorous proofs were later established in \cite{zbMATH04048716,zbMATH03951660, shnirelman1993addendum}. Since then, quantum ergodicity has been the subject of many works, see for instance \cite{zbMATH07569934, zbMATH07495614} which review the motivations and some important achievements. 

It was discovered in the late 1990's that graphs were a good model to study quantum chaos \cite{kottos1997quantum,jakobson1999eigenvalue}. For discrete ``combinatorial" graphs, this means studying the eigenvalues and eigenvectors of the adjacency matrices of a sequence of graphs with number of vertices tending to $+\infty$ (we will sometimes simply say ``in large graphs"). The first studies of quantum chaos on discrete graphs were devoted to \emph{random regular graphs} \cite{jakobson1999eigenvalue,zbMATH06111055,zbMATH06534466}. Beyond the random case, it was proved in \cite{zbMATH06176899} that eigenvectors of large regular graphs satisfying some geometric conditions cannot be localized. An analog of Shnirelman's theorem was then established in large (deterministic) regular graphs looking locally like a tree \cite{zbMATH06434640}, and quantum ergodicity in large regular graphs was 
studied further in \cite{zbMATH06992369,anantharaman2017quantum,zbMATH07514683}. Strong delocalization results have also been established for eigenvectors of random regular graphs, see for instance \cite{zbMATH07068259,zbMATH07067280,he2025gaussian}. 

Beyond regular graphs, the paper \cite{zbMATH07097490} proved that quantum ergodicity holds along sequences of \emph{not necessarily regular} graphs. The result is established under some assumptions on the spectrum of a limiting object. This limiting object is the Benjamini-Schramm limit of the sequence of graphs. One of the assumptions is a  bound on the resolvent of its underlying Schr\"odinger operator which implies absolute continuity of the spectrum, and it has been verified in several concrete examples in \cite{zbMATH06789004,zbMATH07162035}. In a nutshell, \cite{zbMATH07097490} establishes in a rather general setting that ``spectral delocalization implies spatial delocalization of eigenvectors". As in the original works on quantum ergodicity on manifolds \cite{zbMATH03507706,zbMATH04048716,zbMATH03951660, shnirelman1993addendum}, the proofs are based on propagation arguments using a classical flow, here the non-backtracking flow on the graph. The results of \cite{zbMATH07097490,zbMATH06789004,zbMATH07162035}, however, are restricted to large discrete graphs which look like \emph{trees} near most points (i.e., the Benjamini-Schramm limit is supported on infinite trees). Beyond this tree-like setting, quantum ergodicity was proved in \cite{zbMATH07751295} for the family of $\Z^d$-periodic graphs, away from flat bands (point spectrum). This relied in an essential way on Bloch-Floquet theory.

In the present paper, we introduce a completely different roadmap to quantum ergodicity in large graphs, based on trace computations, resolvent approximations and representation theory. This new scheme of proof allows us to establish \emph{quantum mixing} -- a strengthened version of quantum ergodicity introduced in \cite{zbMATH00933297,ZELDITCH2006183} -- in large \emph{Schreier graphs} whose Benjamini-Schramm limit  is a \emph{Cayley graph} having absolutely continuous spectrum. This setting is rather general, and we show concrete applications of our results on important examples. We also give a matricial extension of our main theorems, which allows us to consider some families of non-regular graphs. A key contribution of our proof is that it bypasses the limitations of previous methods which heavily relied on the structure of the limiting graph (being a tree or $\Z^d$-periodic).   Our Schreier graphs can have many short cycles, corresponding to relations between generators in the group defining the Cayley graph -- they do not necessarily look locally like trees or $\mathbb{Z}^d$, precluding the use of non-backtracking operators or Bloch-Floquet theory.

Regarding our setting, 
for a given group $\Gamma$ and a set of generators $S$ of this group, we consider a sequence of actions $\rho_N$ and 
establish quantum mixing for the sequence of Schreier graphs associated to these actions, under various convergence assumptions on the sequence $(\rho_N)$. One of them is equivalent to the Benjamini-Schramm convergence of the Schreier graphs toward the Cayley graph $\Cay(\Gamma,S)$; a second one is a stronger notion of convergence called strong convergence in distribution, and it yields a stronger delocalization result. 
The core arguments, presented in the first half of the paper, are rather short: see the sketch of proof in Section~\ref{s:proofideas}. We believe these ideas have more chance to be adapted to manifolds than earlier methods, which would be a very interesting avenue to explore. The second half of the paper is devoted to illustrating our assumptions and providing several examples of applications.

\subsection{Setup}\label{s:setup}

\paragraph{Cayley and Schreier graphs. } Let $\Gamma$ be a finitely generated group and let $S=S^{-1}$ be a symmetric set of generators. We denote by $\Cay(\Gamma,S)$ the associated (left) {\em Cayley graph}. The vertex set is $\Gamma$ and the edge set is $\Gamma \times S$: the directed edge $e = (g,s)$ connects $g$ to $sg$.

Schreier graphs are natural graphs to approximate the Cayley graph associated to the same group. For $N\in\mathbb{N}$, we consider a permutation representation $\rho_N\in {\rm Hom}(\Gamma,S_N)$, where $S_N$ denotes the group of permutations over $\INT{N} = \{1,\ldots,N\}$. In other words, we consider an action of $\Gamma$ on $ \INT{N}$ defined for $(g,x) \in \Gamma \times \INT{N}$ by $g.x := \rho_N(g)(x)$. The {\em Schreier graph} associated to this action is denoted by $\Sch(\Gamma,S,\rho_N)$: its vertex set is $\INT{N}$, its edge set is $\INT{N} \times S$, and the edge $(x,s)$ connects $x$ to $s.x$. Relations between elements of $\Gamma$ translate into relations between the permutations $\rho_N(g)$ and create cycles in the corresponding Cayley and Schreier graphs. 

This setting is rather general since, for instance, Veblen's Theorem implies that any $d$-regular graph with even $d$ can be realized as a Schreier graph of the free group with $d/2$ free generators. Finite Cayley graphs, for which quantum ergodicity results have been established in \cite{MR4651019,MR4630491,Cyril}, are very special in comparison.  Historically,  Schreier graphs appeared as coset graphs of Cayley graphs (then $N$ is the index of a subgroup $H$ of $\Gamma$). See \cite{Woess_2000,Loh2017,CeccheriniSilbersteinAdderio2021} for general references. Also, random Schreier graphs are extensively studied, see for example \cite{bordenave2025sparsegraphsbenjaminischrammlimits} and references therein.

For $g \in \Gamma$, let $\lambda(g) \in \mathcal{B}(\ell^2(\Gamma))$ be the element of the left regular representation defined, for all $g,v \in \Gamma$, by $\lambda(g)\delta_v =\delta_{gv}$ where $\delta_v \in \ell^2(\Gamma)$ is the Dirac delta function.
The adjacency operators of ${\rm Cay}(\Gamma,S)$ and $\Sch(\Gamma,S,\rho_N)$ are given respectively by: \begin{equation}\label{e:lambdaSrhoNS} \lambda(\IND_S ) = \sum_{g\in S} \lambda(g) \quad \hbox{ and } \quad \rho_N(\IND_S ) = \sum_{g\in S} \rho_N(g),  \end{equation} where we have identified the permutation $\rho_N(g)$ with its permutation matrix, that is, for $x,y \in \INT{N}$,  $\rho_N(g)(x,y) = \IND (x = g.y)$.

More generally, a local operator on a Cayley or Schreier graph is conveniently written in terms of the group algebra $\mathbb C [\Gamma]$. Recall that $\mathbb{C}[\Gamma]$ denotes the complex group ring 

$$ 
\mathbb{C}[\Gamma]=\Bigl\{ p = \sum_g p_g g \mid p_g\in\mathbb{C}, \text{ only finitely many } p_g\neq 0\Bigr\}. 
$$
It is equipped with the involution $p^* = \sum_g \bar p_g g^{-1}$ and the multiplication is $pq=\sum_{g,h}p_gq_h\, gh$. If $p\in \mathbb C[\Gamma]$ and $\rho$ is a unitary representation of $\Gamma$ on a Hilbert space $H$, we define the operator in $\mathcal{B}(H)$  \begin{equation}\label{e:defrhoa} 
\rho(p) = \sum_g p_g \rho(g). 
\end{equation}
The element $p \in \mathbb C[\Gamma]$ can be 
thought of as the symbol of the operator $\rho(p)$, in analogy to differential operators on a manifold.
Importantly the operator $\rho(p)$ is self-adjoint if $p = p^*$, that is, if $p_{g^{-1}} = \bar p_g$ for all $g \in \Gamma$.

For $p = \sum_g p_g g \in \mathbb C$ with support in $S$, we study the operators 
\begin{equation} \label{e:lambdarhoN}
\lambda(p)= \sum_g p_g \lambda(g) \quad \hbox{ and } \quad \rho_N(p) = \sum_g p_g \rho_N(g)
\end{equation}
which can be seen as weighted adjacency operators on $\Cay(\Gamma,S)$ and $\Sch(\Gamma,S,\rho_N)$ respectively. For example, if $p_g \geq 0$ and $\sum_g p_g  = 1$, then $\lambda(p)$ is the transition kernel of the random walk on $\Cay(\Gamma,S)$ where at time $t$, the walker in position $X_t \in \Gamma$ jumps to $g.X_t$ with probability $p_g$ (and similarly for $\rho_N(p)$ on $\Sch(\Gamma,S,\rho_N)$).

In this work, we are interested in situations where we have a sequence $N \to \infty$ of actions $\rho_N$ of $\Gamma$ which converges in some sense toward the regular representation. The simplest notion of convergence considered in this work is the convergence in distribution. We say that $\rho_N$ converges {\em in distribution} toward $\lambda$ if
\begin{equation}\label{e:convN}
\forall g\in \Gamma \setminus \{e\}, \quad \frac1N {\rm Tr} (\rho_N(g))\underset{N\rightarrow +\infty}{\longrightarrow} 0,
\end{equation}
where $e$ is the unit of $\Gamma$. We will see in Section~\ref{sec:Schreier} that this notion of convergence coincides with the Benjamini-Schramm convergence of $\Sch(\Gamma,S,\rho_N)$ to $\Cay(\Gamma,S)$ for any finite generating set $S$ of $\Gamma$. We will also consider below for some of our results a strengthening of the convergence in distribution called strong convergence in distribution.

\paragraph{Quantum ergodicity and quantum mixing on Schreier graphs. } 

In the above setting, we fix $p \in \mathbb C[\Gamma]$ such that $p = p^*$ and consider the self-adjoint operator $P = \lambda(p)$. We assume that  $P$ has purely absolutely continuous spectrum in some interval. More precisely, we assume the  stronger condition that in a compact interval $I_0\subset \mathbb R$, for some $C_0 \geq 1$, 
\begin{equation}\label{eq:AC}
C_0^{-1} \leq \Im  R^z  (e,e) \leq C_0 , \hbox{ for all $z = E + i \eta$, $E \in I_0$, $\eta > 0$,}
\end{equation}
where $e$ is the unit of $\Gamma$, $R^z = (P - z \Id)^{-1}$ is the resolvent of $P$ and $\Im R^z = (R^z - (R^z)^\ast)/ (2i)$.  Equivalently, the spectral measure associated to $\lambda(p)$ (see \eqref{eq:defmup}) has a density bounded above and below on $I_0$.  As discussed in Section~\ref{s:discussion} and Remark  \ref{rk:mostafa}, we can often relax the lower bound in \eqref{eq:AC}. In the sequel, we fix an interval $I \subset I_0$ whose closure is in the interior of $I_0$.

We next consider, along a subsequence of integers $N$, permutation representations $\rho_N : \Gamma \to S_N$ which converge in distribution toward $\lambda$ in the sense \eqref{e:convN}. Consider the self-adjoint matrix $P_N = \rho_N(p) \in M_N(\mathbb C)$, and $(\varphi_\alpha)_{\alpha \in \INT{N}}$ an orthonormal basis of eigenvectors of $P_N$ with  eigenvalues $(\lambda_\alpha)_{\alpha \in \INT{N}}$. We would like to argue that eigenvectors $\varphi_\alpha$ with $\lambda_\alpha \in I$ are delocalized. More precisely, for integer $N \geq 1$ and $a_N \in \mathbb C^N$ with $\| a_N \|_\infty = \max_{x} |a_N(x)| \leq 1$, we expect that under \eqref{eq:AC} and some mild extra assumptions, 
\begin{equation}\label{eq:QEaN}
 \lim_{N \to \infty}   \frac{1}{|\Lambda_I|} \sum_{\alpha \in \Lambda_I}  \left| \langle \varphi_\alpha, a_N\varphi_\alpha\rangle  - \langle a_N\rangle \right|^2   = 0,
\end{equation}
where $\langle f,g\rangle=\sum_{x\in[N]} \bar{f}(x)g(x)$ denotes the scalar product in $\ell^2([N])$, $\Lambda_I$ denotes the set of $\alpha\in[N]$ such that $\lambda_\alpha\in I$, and 
$$
\langle a_N\rangle =\frac 1 N \sum_{x=1}^N a_N(x). 
$$
Since the works \cite{zbMATH06176899,zbMATH06434640}, convergence results of the form \eqref{eq:QEaN} have been identified as discrete versions of the \emph{quantum ergodicity} phenomenon. Their validity depends a priori on the choice of orthonormal basis $(\varphi_\alpha)_{\alpha\in[N]}$ and of observables $(a_N)$ -- although one tries in general to establish them for \emph{any} orthonormal basis of eigenvectors, and for the largest possible class of observables.

There are several possible extensions of \eqref{eq:QEaN}. 
One extension 
is to fix a finite subset $T \subset \Gamma$  (for example, the ball of radius $r$ in a Cayley graph of $\Gamma$)  and consider a sequence of matrices $K_N \in M_N(\mathbb C)$ satisfying $\|K_{N}\|_{1,\infty}  = \max_{x,y} | K_{N}(x,y)| \leq 1$ and such that for all $x,y \in \INT{N}$, 
 \begin{equation}\label{eq:defKN}
 K_N(x,y) = 0 \hbox{ unless there exists $g \in T$ such that $x = g.y$}.
 \end{equation}
These observables are a discrete counterpart to $0$-th order pseudodifferential operators used as observables in quantum ergodicity in a continuous setting.

 For $T  = \{ e\}$ and $K_N(x,x) = a_N(x)$, we retrieve the previous multiplication operator (which is called a \emph{diagonal} observable). We call such matrix $K_N$ {\em $T$-local}. We expect in this setting that 
\begin{equation}\label{eq:QEKN}
\lim_{N\to \infty} \frac{1}{|\Lambda_I|} \sum_{\alpha \in \Lambda_I} \left| \langle \varphi_\alpha, K_N \varphi_\alpha\rangle  -  \langle \varphi_\alpha, \langle K_N  \rangle  \varphi_\alpha \rangle \right|^2 =  0,
\end{equation}
where
\begin{equation}\label{eq:KNav}
\langle K_N \rangle  = \rho_N (k_N),  \quad k_N = \sum_g k_{N,g}  g \in  \mathbb C [\Gamma] \quad \hbox{ and } \quad k_{N,g}= \frac 1 N \sum_{x  = 1}^N  K_N(g.x,x).  
\end{equation}
Explicitly, $\langle \varphi, \langle K_N  \rangle  \varphi \rangle =  \sum_x \bar \varphi(g.x) \varphi(x) k_{N,g}$.
When restricted to $T= \{e\}$, \eqref{eq:QEKN} implies \eqref{eq:QEaN}.  Statements of the form \eqref{eq:QEaN}-\eqref{eq:QEKN} are called discrete \emph{quantum ergodicity}.

Our main theorems will imply such results, under various assumptions on $\Gamma$, $\rho_N$ or $K_N$.

\medskip

Let us now explain the notion of {\em quantum (weak) mixing}, in analogy to the definitions provided in \cite{zbMATH00933297,ZELDITCH2006183} for quantum systems on Riemannian manifolds (see also \cite{feingold1986distribution,zbMATH04188839,sunada1997quantum} for related earlier papers, and \cite{arXiv:2512.15504} for a recent preprint). On Riemannian manifolds, it is proved in \cite{zbMATH00933297,ZELDITCH2006183} that weak mixing of the geodesic flow implies some decay of correlations between eigenfunctions at different energy levels -- while ergodicity of the geodesic flow implies decay of correlations only between eigenfunctions at close enough energy levels. To our knowledge, in discrete settings, the correlation of eigenfunctions at different energies has only been studied in Wigner-type \cite{zbMATH07425783,zbMATH07946620,CEH} and random band matrices \cite{yau2025delocalizationonedimensionalrandomband}.

We first define {\em small scale quantum weak mixing} and postpone to Section~\ref{s:discussion} a tentative definition of quantum weak mixing in this framework, as well as some discussions on related works in random matrix theory. We start  with the simplest case $a_N \in \mathbb C^N$ with $\|a_N\|_\infty \leq 1$. We shall say that small scale quantum weak mixing occurs if for all $E_1,E_2$ in the interior of $I$, 
\begin{equation}\label{eq:QMaN}
\lim_{\eta \to 0} \limsup_{N \to \infty} \frac{1}{|\Lambda_{J_{E_1}^\eta}|} \sum_{\alpha \in \Lambda_{J_{E_1}^\eta}} \sum_{\beta  \in \Lambda_{J_{E_2}^\eta}} |\langle \varphi_\beta , a_N \varphi_\alpha \rangle - \IND_{\alpha = \beta} \langle a_N \rangle |^2 = 0,
\end{equation}
where $J_E^\eta = [E-\eta,E+\eta]$  (notice that for all $\beta \ne \alpha$, $\langle \varphi_\beta, a_N \varphi_\alpha \rangle = \langle \varphi_\beta, a_N \varphi_\alpha \rangle - \langle a_N \rangle  \langle \varphi_\beta, \varphi_\alpha \rangle $). Compared to \eqref{eq:QEKN}, \eqref{eq:QMaN} also captures some information on the correlations between distinct eigenvectors. For simplicity, we will often speak of quantum mixing in place of small scale quantum weak mixing from now on.

Let us explain why taking the limit $\eta\rightarrow 0$ in \eqref{eq:QMaN} is necessary to obtain $0$ in the right-hand side. If $J = J_E^\eta$ with $\eta > 0$ fixed, then \eqref{eq:AC} and the convergence of $\rho_N$ toward $\lambda$ imply that $|\Lambda_J|$ is of order $\eta N$. In such situation, we cannot expect  for a typical unit $\psi$ in $\mathbb C^N$ that $\sum_{\beta \in \Lambda_{J}} |\langle \varphi_\beta ,  \psi \rangle|^2$ will be small. Indeed, if $\psi$ is a uniformly distributed vector on the unit sphere $\mathbb S^{N-1}$, we have in expectation: $\mathbb E \sum_{\beta \in \Lambda_J} | \langle  \varphi_\beta , \psi\rangle|^2 = |\Lambda_J|/N$. 

As before, there is a generalization of quantum mixing to non-necessarily diagonal observables. We fix a finite subset $T \subset \Gamma$ and consider a sequence $K_N \in M_N(\mathbb C)$ of $T$-local matrices as in \eqref{eq:defKN} and such that $\| K_N \|_{1,\infty} \leq 1$ for any $N$.  We say that small scale quantum weak mixing (or simply quantum mixing) occurs if uniformly over $E_1,E_2$ in the interior of $I$, 
\begin{equation}\label{eq:QMKN}
\lim_{\eta \to 0} \limsup_{N \to \infty} \frac{1}{|\Lambda_{J_{E_1}^\eta}|} \sum_{\alpha \in \Lambda_{J_{E_1}^\eta}} \sum_{\beta  \in \Lambda_{J_{E_2}^\eta}} \left|\langle \varphi_\beta , K_N \varphi_\alpha \rangle  - \langle \varphi_\beta,  \langle K_N \rangle  \varphi_\alpha \rangle \right|^2 = 0,
\end{equation}
where $\langle K_N \rangle$ was defined in \eqref{eq:KNav}. Note that \eqref{eq:QMKN} implies \eqref{eq:QMaN} for $T = \{e\}$ and $K_N(x,x) = a_N(x)$. 
Importantly, quantum weak mixing implies quantum ergodicity, see Lemma~\ref{le:mixtoergo} for a precise statement. 

Let us roughly explain the meaning of \eqref{eq:QMKN}. If we take $K_N(x,y)=k_{N} (x)\mathbf{1}(x=g.y)$, it says that for most eigenfunctions $\varphi_\alpha$, for all eigenfunctions $\varphi_\beta$ in the corresponding spectral intervals, we have  
$$
\langle \varphi_\beta,K_N\varphi_\alpha\rangle  = \frac 1 N \sum_{x =1}^N N \overline{ \varphi_\beta (g.x)} \varphi_\alpha (x) k_N(x) \approx \langle k_{N} \rangle \langle \varphi_\beta(g.),\varphi_\alpha\rangle.
$$
That is, seeing $x \in \INT{N}$ as a uniformly sampled random variable on $\INT{N}$, in a weak sense, $ N \overline{\varphi_\beta(g.x)}\varphi_\alpha(x)\approx   \langle \varphi_\beta(g.),\varphi_\alpha\rangle$. 
This says that the eigenfunction correlation a priori depends on the orthonormal basis. At our level of generality this is unavoidable, as can be seen in the case of $\Z^d$ for $\alpha=\beta$ and $g\neq e$ (see \cite[Section 5.1.2]{zbMATH07751295}). For diagonal observables however, the centering reduces to $\langle a_N\rangle \delta_{\alpha=\beta}$ (see \eqref{eq:QMaN}). We discuss the centering term further in Section~\ref{s:discussion} below.

\subsection{Quantum mixing for asymptotically uncorrelated observables}

In this subsection, we state our first main result. It says that under the only assumption that $\rho_N$ converges in distribution toward $\lambda$, quantum mixing holds for all $T$-local matrices $K_N \in M_N(\mathbb C)$  which have no long-range correlations as $N\rightarrow +\infty$, meaning that for $t \in T$, the values of $K_N(t.x,x)$ and $K_N(tg.x,g.x)$ are not correlated as $|g|\to \infty$ (where $|g|$ is the length of $g$ in the word metric for a fixed choice of a finite generating set of $\Gamma$). More precisely:
\begin{definition}\label{def:corr}
We say that a sequence of observables $a_N:\INT{N} \rightarrow \mathbb{C}$ is asymptotically uncorrelated if 
\begin{equation}\label{e:epsgconv}
\lim_{|g|\rightarrow +\infty} \sigma(g)=0, \qquad \sigma(g)=\limsup_{N\rightarrow +\infty}\Bigl| \frac1N \sum_{x\in \INT{N}} (a_N(x) - \langle a_N\rangle) (a_N(g.x) - \langle a_N\rangle) \Bigr|.
\end{equation}
For a finite $T \subset \Gamma$ and a sequence of $T$-local matrices $K_N \in M_N(\mathbb C)$, we say that $(K_N)$ is asymptotically uncorrelated if for any $t \in T$, $a_N (x) = K_N (t.x,x)$ is asymptotically uncorrelated.
\end{definition}

It is possible to prove that for any sequence $\rho_N$ converging in distribution to $\lambda$, there exists a bounded sequence of observables $(a_N(x))_{x \in \INT{N}} \in \mathbb C^N$ which is \emph{not} asymptotically uncorrelated (see Remark \ref{r:notalluncorrel}). The following lemma shows nevertheless that random observables are asymptotically uncorrelated with probability one, if $\rho_N$ converges in distribution to $\lambda$.

\begin{lemma}\label{p:asympuncorrproba1}
Assume that $\rho_N$ converges in distribution to $\lambda$. For each $N \geq 1$, let $(a_N(x))_{x \in \INT{N}} \in \mathbb C^N$ be independent random variables such that  $\|a_N\|_{\infty}\leq 1$ almost surely and the expectation of $a_N(x)$ is independent of $x$. Then almost surely,  for all $g \ne e$, $\sigma(g) = 0$ where $\sigma$ is as in Definition \ref{def:corr}. In particular, the asymptotic uncorrelation condition holds almost surely.
\end{lemma}

We now state our first result of quantum mixing:
\begin{theorem}\label{t:asympdecorr}
Assume that $\rho_N$ converges in distribution toward $\lambda$. Let $ p \in \mathbb C[\Gamma]$ such that $p = p^*$ and an interval $I_0$ such that \eqref{eq:AC} holds for $P = \lambda(p)$. Let $I$ be a closed interval in the interior of $I_0$ and  $(\varphi_\alpha)$ be an orthonormal eigenbasis for $P_N = \rho_N(p)$.
Then, for any finite $T \subset \Gamma$ and any asymptotically uncorrelated sequence of $T$-local matrices $(K_N)$ in $M_N(\mathbb C)$ such that $\|K_N\|_{1,\infty} \leq 1$, quantum ergodicity \eqref{eq:QEKN} and quantum mixing \eqref{eq:QMKN} hold in $I$.
\end{theorem}

Compared to previous results in the literature \cite{zbMATH06434640,zbMATH06992369,zbMATH06789004,anantharaman2017quantum,zbMATH07097490,zbMATH07162035,zbMATH07751295,MR4651019,MR4630491,Cyril}, Theorem~\ref{t:asympdecorr} gives the first quantum ergodicity theorem controlling all orthonormal bases, without requiring the limiting graph to be a tree or $\Z^d$-periodic. For this, we merely require \eqref{eq:AC}, a mild strengthening of the assumption that the spectrum of the limiting Cayley graph is purely absolutely continuous in $I_0$, and convergence in distribution, which is equivalent to Benjamini-Schramm convergence of the sequence of Schreier graphs, see Proposition~\ref{p:conv}. We have no spectral gap assumption on $P_N$. On the other hand, the conclusion holds for a subclass of observables. This is necessary at this level of generality. In fact, we provide in Section~\ref{s:asympdecorr} an example which shows that the conclusion of Theorem~\ref{t:asympdecorr} does not hold in general, if one only assumes that $\|a_N\|_{\infty}\leq 1$ (but not the asymptotic uncorrelation condition).

Together, Lemma~\ref{p:asympuncorrproba1} and Theorem~\ref{t:asympdecorr} show that quantum ergodicity and quantum mixing hold for deterministic sequences of graphs, with high probability with respect to the choice of random observables. This contrasts with the results in the random graphs and random matrices literature, e.g. \cite{zbMATH07068259,zbMATH07425783}, where quantum unique ergodicity is proved for deterministic observables, with high probability with respect to the choice of a random graph or a random matrix. 

\subsection{Quantum mixing under rapid decay property}

In our next result, Theorem~\ref{t:HNperp}, quantum ergodicity and quantum mixing are established for $T$-local observables $K_N$ as in \eqref{eq:defKN} such that for all $t \in T$, the $\mathbb C^N$-vector $(K_N(t.x,x))_{x \in \INT{N}}$ belongs to $H_N^\perp$, where $H_N$ is a space of dimension $o(N)$ and $\perp$ is the orthogonal complement in $\mathbb C^N$. The theorem relies on the assumption that the group $\Gamma$ satisfies the rapid decay (RD) property, which we now introduce.  For any $p \in \mathbb{C}[\Gamma]$, we have 
\begin{equation*}\label{e:L2oplower}
\|p \|_2 = \| \lambda (p)  \delta_e \|_2 \leq \| \lambda (p) \|_{\rm op},
\end{equation*}
where $\|\cdot\|_{\rm op}$ denotes the operator norm. 
The RD property, introduced by Haagerup \cite{zbMATH03634902} on the free group and defined by Jolissaint \cite{zbMATH04169445} (see also the survey \cite{zbMATH06859874}), is a quantified reversed bound.
\begin{definition}\label{d:RD}
A group $\Gamma$ has the rapid decay (RD) property if for some finite set $S$ of generators of $\Gamma$ there exist $C,C_1>0$ such that for any  $p \in \mathbb{C}[\Gamma]$,  
\begin{equation}\label{e:rddef}
\| \lambda (p) \|_{\rm op} \leq C ({\rm diam}_S(p))^{C_1} \| p \|_2,
\end{equation}
where ${\rm diam}_S(p)$ is the maximum between one and the diameter of the support of $p$  in ${\rm Cay}(\Gamma,S)$. 
\end{definition}

One may check that this property of $\Gamma$ is independent of the finite set $S$ of generators (namely, if it holds for some $S$ it holds for all finite generating sets), and that $C_1$ in \eqref{e:rddef} may be chosen independently of $S$. 
Examples of groups having the RD property are free groups, groups with polynomial growth (such as $\mathbb{Z}^d$),  $\mathrm{SL}_2(\mathbb{Z})$, hyperbolic groups, Cartesian or free products of such groups, etc (see \cite{zbMATH06859874}). Non-examples include  $\mathrm{SL}_n(\mathbb{Z})$ for $n \geq 3$ \cite[Corollary 3.1.9]{zbMATH04169445} and amenable groups with non-polynomial growth \cite[Corollary 3.1.8]{zbMATH04169445}. Note that it is trivial to check that groups with polynomial growth have property RD  since
$$
\| \lambda(p) \|_{\rm op} \leq \| p \|_1 \leq \sqrt {|\mathrm{\supp}(p)|} \| p \|_2 \leq \sqrt {|B_S(\mathrm{diam}_S(p))|} \| p \|_2,$$
where $B_S(r)$ denotes the ball of radius $r$ and center $e$ in ${\rm Cay} (\Gamma,S)$.

Besides the RD property, we will need an off-diagonal decay of the resolvent stated below as \eqref{e:RDrenforceeweak}, in addition to the previous diagonal control we have in \eqref{eq:AC}. The Ward identity (see \eqref{e:ward}) asserts that, for all $z = \lambda_0 + i \eta$ with $\lambda_0 \in I_0$, $\eta >0$,
$$
\sum_{g \in \Gamma} \eta |R^z (e,g)|^2 = \Im R^z (e,e) \leq C_0
$$
where for $Q$ a bounded operator on $\ell^2(\Gamma)$, we set $Q(x,y) := \langle \delta_x , Q \delta_y \rangle$.
We will also check in Section~\ref{s:resolvent} that \eqref{eq:AC} together with the Cauchy-Schwarz inequality implies the rough bound
\begin{equation}
    \label{eq:AC0}
    \sum_{g \in \Gamma} \eta |(\Im R^z)(e,g) |^4  \leq |\Im R^z (e,e) | ^3\leq C_0^3.
\end{equation}
Notably, $ \sum_{g \in \Gamma} \eta^2 |(\Im R^z)(e,g)|^4$ goes to $0$ as $\eta \to 0$ uniformly in $\lambda_0 \in I_0$. We  will require an improvement of this last claim. Namely, with $C_1$ as in \eqref{e:rddef}, we will assume that for some $C'_1 > 2C_1 + 1$, 
\begin{equation}\label{e:RDrenforceeweak}
\lim_{\eta \to 0} \sup_{\lambda_0 \in I_0} \sum_{g\in \Gamma}\eta^2 |(\Im R^{\lambda_0+i\eta})(e,g)|^4 |g|^{C'_1} = 0,
\end{equation} 
where $|g|$ is the length of $g$ in the word metric for a fixed choice of a finite generating set of $\Gamma$.

\begin{theorem}\label{t:HNperp}
Assume that $\rho_N$ converges in distribution toward $\lambda$ and that $\Gamma$ has the RD property. Let $ p \in \mathbb C[\Gamma]$ such that $p = p^*$ and an interval $I_0$ such that \eqref{eq:AC} and \eqref{e:RDrenforceeweak} hold for $P = \lambda(p)$.  Let $I$ be a closed interval in the interior of $I_0$ and $(\varphi_\alpha)$ be an orthonormal eigenbasis for  $P_N = \rho_N(p)$. Then, for any finite $T \subset \Gamma$, there exists a sequence of subspaces $H_N\subset \mathbb C^N$ with ${\rm dim}(H_N)=o(N)$ such that for any sequence of $T$-local matrices $(K_N)$ in $M_N(\mathbb C)$ satisfying  $\|K_N\|_{1,\infty} \leq 1$ and $(K_N(t.x,x))_{x \in \INT{N}} \in H_N^\perp$ for all $t \in T$, quantum ergodicity \eqref{eq:QEKN} and quantum mixing \eqref{eq:QMKN} hold in $I$.
\end{theorem}

Roughly speaking, $H_N$ is a spectral subspace given by the discrepancy between the spectra of $Q(P_N)$ and $Q(P)$, for some family of polynomials $Q$. Essentially, it is spanned by eigenvectors of $Q(P_N)$ with eigenvalues outside the spectrum of $2Q(P)$.

Assumption \eqref{e:RDrenforceeweak} is studied in detail in Section~\ref{s:applications}, where several illustrating examples are provided.  Let us only say for now that we expect this condition is implied by \eqref{eq:AC} for a very large class of non-amenable groups (if not all), but we are able to prove it only in specific cases (see notably forthcoming Proposition~\ref{p:QEcoxeter}). It is not satisfied when $\Gamma=\mathbb{Z}^d$, see Section~\ref{s:Zd}. Also, our proof of Theorem~\ref{t:HNperp} provides a rate of convergence in \eqref{eq:QEKN}-\eqref{eq:QMKN} and on the dimension of $H_N$, see Appendix \ref{a:convrate}.

\subsection{Quantum mixing under strong convergence}\label{s:theoRD}

So far we presented quantum mixing results which hold for subclasses of observables, namely asymptotically uncorrelated ones, and ones orthogonal to some subspace of negligible dimension. Allowing only subclasses of observables could miss some fine properties of the eigenvectors. 
To have more evidence that $\overline{\varphi_\beta}(g.x)\varphi_\alpha(x)\approx \frac{1}{N}\langle \varphi_\beta(g.),\varphi_\alpha\rangle$, we add a final assumption, which allows us to control \emph{all} observables.

A sequence of unitary representations $(\rho_N)$ of $\Gamma$ is said to converge \emph{strongly in distribution} to the regular representation $\lambda$ of $\Gamma$ if it converges in distribution to $\lambda$, and additionally, for any $p\in \mathbb{C}[\Gamma]$, 
\begin{equation}\label{e:strconvinitial}
\lim_{N\rightarrow +\infty} \|\rho_N(p)\|_{\rm op}=\|\lambda(p)\|_{\rm op}
\end{equation}
We refer to \cite{zbMATH06324768,arXiv:2503.21619,van2025strong}.

In the context of representations of $\Gamma$ in the symmetric group $S_N$ acting on $\R^N$, since the constant vector $1\in\R^N$ is invariant under any permutation, we restrict the operator norm in the left-hand side of \eqref{e:strconvinitial} to $1^\perp$. In other words, we restrict ourselves to the subrepresentation $(\rho_N)$ restricted to $1
^\perp$. We thus arrive at the following definition.
\begin{definition}\label{d:strongconv}
If $\rho_N\in{\rm Hom}(\Gamma,S_N)$ is a sequence of representations of $\Gamma$, we say that $\rho_N$ strongly converges in distribution to the regular representation $\lambda$ if it converges in distribution to $\lambda$, and additionally, for any $p \in \mathbb{C}[\Gamma]$, 
\begin{equation}\label{e:AN1perpconv}
\lim_{N \to \infty} \|\rho_N(p)_{|1^\perp}\|_{\rm{op}} =\|\lambda(p)\|_{\rm op }
\end{equation}
where $\lambda(p)$ and $\rho_N(p)$ are defined respectively in \eqref{e:defrhoa} and \eqref{e:lambdarhoN}. In the sequel, we may simply say that the sequence ``converges strongly".
\end{definition}

The norm convergence \eqref{e:AN1perpconv} is equivalent to the convergence of the spectrum in  Hausdorff distance: for every $p\in\mathbb{C}(\Gamma)$,
$$
\underset{N\rightarrow +\infty}{\lim} {\rm d}_H({\rm sp}\big(\rho_N(p)_{1^\perp}),{\rm sp}(\lambda(p))\big)=0,
$$
(see \cite[Lemma 2.11]{van2025strong}). In other words, the condition \eqref{e:AN1perpconv} indicates the absence of outlier eigenvalues.

Our main result in this framework of strong convergence is the following:

\begin{theorem}\label{t:QEstrongCV}
Assume that $\rho_N$ converges strongly in distribution to $\lambda$ and that $\Gamma$ has the RD property. Let $ p \in \mathbb C[\Gamma]$ such that $p = p^*$ and an interval $I_0$ such that \eqref{eq:AC} and \eqref{e:RDrenforceeweak} hold for $P = \lambda(p)$.  Let $I$ be a closed interval in the interior of $I_0$ and $(\varphi_\alpha)$ be an orthonormal eigenbasis for  $P_N = \rho_N(p)$. Then, for any finite $T \subset \Gamma$, and any sequence of $T$-local matrices $(K_N)$ in $M_N(\mathbb C)$ with  $\|K_N\|_{1,\infty} \leq 1$, quantum ergodicity \eqref{eq:QEKN} and quantum mixing \eqref{eq:QMKN} hold in $I$.
\end{theorem}

\begin{remark}[Relative strong convergence]\label{r:QEstrongCV}
Recall that an invariant subspace $H_N$ of a representation $\rho_N$ is a subspace which is left invariant by all $\rho_N(g)$, $g \in \Gamma$ (in other words, ${\rho_N}_{|H_N}$ is a subrepresentation of $\rho_N$).  We will prove that if we replace in \eqref{e:AN1perpconv} the vector space $1^\perp$ more generally with $(H_N + \mathrm{span}(1))^\perp$ for some invariant vector subspace $H_N$ then the conclusion of Theorem~\ref{t:HNperp} holds with this choice of $H_N$. This comment is relevant for some applications where such relative strong convergence has been established but where strong convergence does not hold: for Cartesian products of free groups, see \cite[Theorem 9.5]{arXiv:2304.05714} and \cite[Proposition 2.7]{arXiv:2503.21619}.
\end{remark}

\subsection{Extension to finite coverings}

\label{subsec:matricial}

In Section~\ref{s:matricial}, we extend our main theorems to the setting of finite coverings. From a graph-theoretic perspective, this amounts to replacing Cayley graphs of $\Gamma$ with graphs that are quasi-transitive under the action of $\Gamma$. The Schreier graphs are then replaced by finite graphs covered by these quasi-transitive graphs. Algebraically, this corresponds to replacing the group algebra $\mathbb{C}[\Gamma]$ with the free $M_r(\mathbb{C})$–module $M_r(\mathbb{C})[\Gamma]$. This framework has been considered, for instance, in \cite{zbMATH01309770, arXiv:2304.05714, zbMATH07693380, zbMATH07304098, zbMATH07582187}, though for purposes different from quantum ergodicity; see also the survey~\cite{bordenave2025sparsegraphsbenjaminischrammlimits}.

This extension covers a broad range of examples of non-regular graphs, including graph products and graph coverings, (see Examples \ref{ex:cartesian}–\ref{ex:covering}). For instance, our results apply to the adjacency operators of a sequence of $N$–coverings of a fixed base graph converging in the Benjamini–Schramm sense to their universal covering tree. This situation has notably been studied in \cite{zbMATH06789004, zbMATH07162035} in the more delicate context of the Anderson model.

To avoid notational overload, we defer the precise statement of the main theorem in this setting, Theorem~\ref{t:QEmat}, to Section~\ref{s:matricial}.

\subsection{Applications of the main theorems}
In the second half of the paper, we detail several applications of our main results.
\begin{itemize}
\item In Section~\ref{s:freeproducts} we show that Theorems \ref{t:asympdecorr} and \ref{t:HNperp} apply when $\Gamma$ is the free product of at least three non-trivial finite groups (or the free product of two complete graphs). This includes ``anisotropic walks'' on such Cayley graphs, which are a lot more general than the regular tree considered in \cite{zbMATH06434640,anantharaman2017quantum}.
\item Then, we provide examples of groups with a more complicated structure than free products which nevertheless satisfy our assumptions: we prove in Section~\ref{s:RACG} that Theorems \ref{t:asympdecorr}, \ref{t:HNperp} and \ref{t:QEstrongCV} apply to some class of right-angled Coxeter groups (RACG). 
\item In Section~\ref{s:liftsofgraphs} we show that Theorem~\ref{t:QEmat} applies to $N$-lifts of a finite base graph. This improves over the results of \cite{zbMATH07097490,zbMATH07162035} which imply quantum ergodicity, but not quantum mixing, in the case of generic covers.
\end{itemize}

\subsection{Further discussion and open questions} \label{s:discussion}

We discuss here some variants, extensions, applications and open questions related to our main results.

{\em Quantum weak mixing.} We have defined small scale quantum weak mixing in \eqref{eq:QMaN}-\eqref{eq:QMKN}. Such small scale estimates are very natural from the perspective of local laws established in random matrix theory. In analogy with \cite{zbMATH00933297,ZELDITCH2006183}, one may also be interested in \emph{quantum weak mixing} in a macroscopic interval $I$. A reasonable definition for this would be that, for the sequence of $T$-local observables $(K_N)$, and for any real $\tau$, we have
\begin{equation}\label{eq:QMZ}
\lim_{\eta \to 0} \lim_{N\to \infty} \frac{1}{|\Lambda_I|} \sum_{\substack{\alpha , \beta  \in \Lambda_I : \\ |\lambda_\beta - \lambda_\alpha - \tau| \leq \eta}} \left| \langle \varphi_\beta, K_N \varphi_\alpha\rangle  -  \langle \varphi_\beta, \langle K_N  \rangle  \varphi_\alpha \rangle \right|^2 =  0.
\end{equation}
This is essentially the definition used in \cite{arXiv:2512.15504} in the setting of hyperbolic surfaces, with diagonal observables.
Taking $\tau = 0$, we readily see that \eqref{eq:QMZ} implies quantum ergodicity \eqref{eq:QEKN}. We show in Lemma~\ref{le:mixtoergo} that \eqref{eq:QMKN} implies \eqref{eq:QMZ}, so our theorems imply this statement of quantum weak mixing as well.
On compact manifolds,  Zelditch \cite{zbMATH00933297,ZELDITCH2006183} has a characterization of quantum weak mixing in terms of the weak mixing of the geodesic flow. This does not seem to have a clear analog in our discrete setting. The same comment applies to the converse of the quantum ergodicity theorem, proved by Sunada \cite{sunada1997quantum}.  
It would be very interesting to explore further this analogy.

\smallskip 

{\em On the lower bound on the resolvent in \eqref{eq:AC}.}  The lower bound $\Im R^z (e,e) \geq C_0^{-1}$ in \eqref{eq:AC} is only needed in the proofs of small scale quantum weak mixing to guarantee that for some $C'_0 >0$, for any $\eta >0$, all $N$ large enough and all $E \in I$, we have $|\Lambda_{J^\eta_E}| \geq C'_0 N \eta$. If one is interested in quantum ergodicity \eqref{eq:QEKN} or quantum weak mixing \eqref{eq:QMZ} on a fixed interval $I \subset I_0$, then a lower bound of the form $| \Lambda_I | \geq C'_0 N$ is sufficient (where $C'_0$ is allowed to depend on $I$). The latter condition is satisfied under the sole hypothesis that the spectrum of the operator $P = \lambda(p)$ intersects the interior of $I$. See Remark \ref{rk:mostafa} for more details.

\smallskip 

{\em Rates of convergence.}  It is possible to extract from the proofs of our main theorems quantitative rates of convergence: in typical situations, our proofs give a rate of convergence for quantum ergodicity in \eqref{eq:QEKN} of order
$$
\frac{\ln\ln N}{\ln N}
$$
as we explain in Appendix \ref{a:convrate}, see notably Theorem~\ref{t:QEN} for a precise statement.  We also discuss rates of convergence for quantum mixing in the appendix.  This is not better than standard rates in quantum ergodicity, see for instance \cite[Remark 1.2]{anantharaman2017quantum}. In random matrix settings, nearly optimal bounds are known for deterministic matrices $K_N$ and random matrices $P_N$, see \cite{zbMATH07946620,yau2025delocalizationonedimensionalrandomband} and the recent survey \cite{erdos2025lecture}.  Our Theorem~\ref{t:asympdecorr} on asymptotically uncorrelated observables  is in some sense closer to this random matrix setting where the observable is not correlated with the underlying structure of $P_N$.

\smallskip 

{\em Identifying the term $\langle K_N\rangle$.} For diagonal observables, the centering by $\langle a_N\rangle$ in \eqref{eq:QEaN} and \eqref{eq:QMaN} is natural and independent of the choice of orthonormal basis. For non-diagonal observables, it is natural to wonder whether the centering term $\langle \varphi_\beta, \langle K_N\rangle \varphi_\alpha \rangle$ appearing in \eqref{eq:QEKN}-\eqref{eq:QMKN} could be computed asymptotically.  This should be related to the covariance structure of Gaussian random waves on the limiting Cayley graph. For a general Cayley graph, there exist Gaussian random waves for eigenvalues inside the absolutely continuous spectrum (see below Lemma 2.1 in \cite{arXiv:0907.5065}). If $\Gamma$ is a free group and if $P = \lambda(p)$ is the adjacency operator of the Cayley graph with its free generators, then for any $t \in \Gamma$, $\langle \varphi_\alpha,\varphi_\alpha(t\cdot)\rangle$ should be close, for most eigenvectors $\varphi_\alpha$ with eigenvalue close to $\lambda$, to the covariance at $t$ of the invariant eigenvector process with eigenvalue $\lambda$ on the $d$-regular tree, namely the spherical function $\Phi_\lambda$ evaluated at $|t|$ (see \cite[Theorem 2]{anantharaman2017quantum} and \cite[Theorem 2.2]{zbMATH07067280}).

\smallskip 

{\em Existence of absolutely continuous spectrum.}  
There seems to be relatively few investigations of Cayley graphs exhibiting absolutely continuous spectrum. In Section~\ref{s:applications}, we discuss the case of free products of groups and some right-angled Coxeter groups. Surface groups $\Gamma = \langle a_1\ldots,a_g, b_1, \ldots,b_g | [a_1,b_1][a_2,b_2] \cdots [a_g,b_g] =e \rangle$, $g \geq 2$, are an interesting case: we do not know if the adjacency operator of the Cayley graph with generators $(a_i,b_i)$ and their inverses has some absolutely continuous spectrum. The nilpotent groups would also be interesting  and natural to study.

\subsection{Organization of the paper}
The paper is organized as follows. In Section~\ref{sec:general}, we show a general bound which is common to the proofs of all our main results, and summarize the other main ideas of the proofs. Section~\ref{sec:Schreier} reviews useful facts on Benjamini-Schramm convergence and its consequences on convergence of spectral measures and trace estimates. Section~\ref{s:pfthasymp} is the central one: it is devoted to the proofs of Theorems \ref{t:asympdecorr}, \ref{t:HNperp} and \ref{t:QEstrongCV} and Lemma~\ref{p:asympuncorrproba1}. In Section~\ref{s:matricial}, we state and prove a generalization of our results to a matricial framework, which requires additional notation but appreciably widens the scope of applications of our results by notably including irregular graphs. Section~\ref{s:applications} details applications of our main results to free products, RACG and lifts.
These applications are complemented in Section~\ref{s:necessityassump} by several examples (or non-examples) showing that quantum ergodicity/mixing can fail as soon as any of the assumptions in our main results is removed. In Appendix~\ref{s:strongcvindistrib} we present some facts on the strong convergence in distribution and its relation to convergence in distribution. Appendix~\ref{s:apprd} details classical consequences of the RD property. Finally, in Appendix~\ref{a:convrate} we derive convergence rates for Theorems~\ref{t:HNperp} and \ref{t:QEstrongCV} by providing quantitative versions of certain qualitative arguments from the main body of the paper.

\medskip

\textbf{Acknowledgments and fundings.} We thank Camille Horbez for pointing out the paper \cite{white}, Anne Thomas for a discussion regarding this paper, Tuomas Sahlsten and Kai Hippi for a discussion of quantum mixing and for pointing out the paper \cite{ZELDITCH2006183}.  
We would like to express our sincere gratitude to the CIRM for giving us the opportunity to work together during two weeks of ``Research in Pairs”. CB is grateful to the Institute for Advanced Study for its support and inspiring environment during the 2025–2026 academic year; his research there was supported by the James D. Wolfensohn Fund. CL acknowledges the support of the Agence nationale de la
recherche, through the ANR JCJC project RANDOP (ANR-24-CE40-3377).

\section{Quantum mixing and resolvent estimates}
\label{sec:general}

\subsection{Small scale quantum weak mixing implies quantum ergodicity and quantum weak mixing} \label{s:generalbound}

 In this preliminary section, we introduce important notation and show that small scale quantum weak mixing \eqref{eq:QMKN} implies quantum ergodicity \eqref{eq:QEKN} and quantum weak mixing \eqref{eq:QMZ}. This explains why we focus on establishing \eqref{eq:QMKN} in the rest of this paper.

Let $A \in M_N(\mathbb C)$ be a self-adjoint matrix, $(\varphi_\alpha)_{\alpha \in [N]}$ be an orthonormal basis of eigenvectors with associated eigenvalues $(\lambda_\alpha)_{\alpha \in [N]}$. If $I$ is a subset of $\mathbb R$, we define $\Lambda_I \subset [N]$ as the subset of $\alpha$ such that $\lambda_\alpha \in I$.

If $K \in M_N(\mathbb C)$, we define for $I,J \subset \mathbb R$ and $\tau, \eta \geq 0$, the quantum moments: 
\begin{equation}\label{eq:defLL0}
L^{\tau,\eta}_I(A,K) = \frac{1}{|\Lambda_I|} \sum_{\alpha \in \Lambda_I} \hspace{-3pt}\sum_{\substack{\beta  \in \Lambda_I : \\ |\lambda_\beta -\lambda_\alpha - \tau| \leq \eta}} \hspace{-10pt}| \langle \varphi_\beta, K \varphi_\alpha \rangle|^2  \; \hbox{ and } \; L_{IJ} (A,K) = \frac{1}{|\Lambda_I|} \sum_{\alpha \in \Lambda_I} \sum_{\beta \in \Lambda_J} | \langle \varphi_\beta, K \varphi_\alpha \rangle|^2 ,
\end{equation}
with the convention that $L_{I,J} =  L_I^{\tau,\eta} = 0$ if $\Lambda_I$ is empty. The notation is slightly abusive since $L^{\tau,\eta} _I(A,K)$ actually depends on the choice of the orthonormal basis. The quantum moment $L_{IJ} (A,K)$ is however independent of the choice of the basis as it can be written in terms of spectral projections (see the proof of Lemma~\ref{le:TraceQ} below). Note that, for centered observables $K$, the term $L^{0,0}_I$ appears in the definition of quantum ergodicity \eqref{eq:QEKN}, the term $L^{\tau,\eta}_I$ in the definition of quantum weak mixing \eqref{eq:QMZ}, while $L_{IJ}$ appears in the definition of small scale quantum weak mixing \eqref{eq:QMKN}. We will be interested in situations where these moments are small. Since $L_{IJ}(A,K)$ is unitarily invariant, it will be easier to upper bound. However for $|\Lambda_J|$ large, $L_{IJ}(A,K)$ cannot be expected to be small in general.

Our strategy to upper bound $L_I^{\tau,\eta} $ will be to decompose $I$ into small sets counting few eigenvalues. The following lemma shows that  small scale quantum weak mixing implies quantum ergodicity and quantum weak mixing (using the second part of the statement with $K_N-\langle K_N\rangle$ instead of $K_N$).

\begin{lemma}[Small scale quantum weak mixing implies quantum ergodicity and mixing]\label{le:mixtoergo}
Let $I\subset \mathbb{R}$ be an interval. For any self-adjoint $A \in M_N(\mathbb C)$, $K \in M_N(\mathbb C)$, real $\tau$ and $\eta>0$, we have 
$$
L^{\tau,\eta}_I(A,K) \leq \max_{E \in I} L_{J^{\eta}_{E} , J^{2\eta}_{E+\tau} \cap I} (A,K)
$$
where $J^\eta_E = [E - \eta,E+ \eta]$. In particular, if along a subsequence $(A_N,K_N)$ of matrices in $M_N(\mathbb C)$, we have 
$$ \liminf_{\eta \to 0} \limsup_{N \to \infty} \max_{E_1,E_2 \in I} L_{J^\eta_{E_1}  J^\eta_{E_2}} (A_N,K_N) = 0,$$
then $\lim_{N\to \infty} L^{0,0}_I(A_N,K_N) = 0$ and $\lim_{\eta \to 0} \limsup_{N \to \infty} L^{\tau,\eta}_I(A_N,K_N) = 0$.
\end{lemma}

\begin{proof}
If $I = \cup_j I_j$, it follows from  \eqref{eq:defLL0} that 
\begin{equation}\label{eq:LoL}
L^{\tau,\eta}_I(A,K) 
\leq \sum_{j}\frac{|\Lambda_{I_j}|}{|\Lambda_I|} L_{I_j  I^{\tau,\eta}_j} (A,K),
\end{equation}
where, if $\bar I_j = [a_j,b_j]$, we have set $I^{\tau,\eta}_j = [a_j+\tau-\eta,b_j+\tau + \eta ] \cap I$. Taking the $(I_j)$ as a partition of $I$ by intervals $J_{E_j}^\eta$ of size $2\eta$, we get the first part of the statement. The second part follows from the first part and the fact that $L^{\tau,\eta}_{I}$ is a non-decreasing function of $\eta$.
\end{proof}

Note that for the conclusion $\lim_{N\to \infty} L^{0,0}_I(A_N,K_N) = 0$ to hold, we only need the weaker condition $ \liminf_{\eta \to 0} \limsup_{N \to \infty} \max_{E  \in I} L_{J^\eta_{E}  J^\eta_{E}} (A_N,K_N) = 0$.

\subsection{A general bound on quantum variance}

The quantity $L_{IJ}(A,K)$ introduced in \eqref{eq:defLL0} is the central object that must be bounded from above in order to prove small scale quantum weak mixing. One key tool to control this quantity is to observe that it can be expressed  as a trace involving $K$ and two functions of the operator $A$.

\begin{lemma}\label{le:TraceQ}
For $j = 1,2$, let $I_j \subset \mathbb R$ and $f_j:\mathbb R \to \mathbb R$, $f_j\geq 0$, such that $f_{j} (x) \geq 1 $ on $I_j$. Then, for any self-adjoint $A$ and $K$ in $M_N(\mathbb C)$, 
\begin{equation}\label{e:LI1I2Tr}
L_{I_1 I_2} (A,K) \leq  \frac{1}{|\Lambda_{I_1}|}  \Tr \left( K f_1(A) K^* f_2(A)\right),
\end{equation}
with equality if $f_j = \IND_{I_j}$. In particular, if $I_j = [E_j - \eta,E_j + \eta]$, we have 
$$
L_{I_1 I_2} (A,K) \leq  \frac{4 \eta^2}{|\Lambda_{I_1}|}  \Tr \left(K \Im R^{z_1} K^* \Im R^{z_2} \right),
$$
where $z_j = E_j + i \eta$ and $R^z  = (A-z\Id)^{-1}$.
\end{lemma}
\begin{proof}
Note that $$ | \langle \varphi_\beta, K \varphi_\alpha \rangle|^2 =  \langle \varphi_\beta, K \varphi_\alpha \rangle \langle  K \varphi_\alpha , \varphi_\beta \rangle = \varphi^*_\beta K \varphi_\alpha \varphi^*_\alpha K^* \varphi_\beta   =  \Tr ( K \varphi_\alpha \varphi_\alpha^* K^* \varphi_\beta \varphi_\beta^* ).$$ 
Hence,  
$$
\Tr \left( K f_1(A) K^* f_2(A)\right) = \sum_{\alpha \in \Lambda_{I_1}} \sum_{\beta \in \Lambda_{I_2}} f_1(\lambda_\alpha) f_2(\lambda_\beta)  | \langle \varphi_\beta, K \varphi_\alpha \rangle|^2.
$$
The first statement follows. For the second statement, we use that for $z = E + i \eta$, $\Im ( \lambda - z) = \eta / ((\lambda-E)^2 + \eta^2) \geq 1/( 2\eta)$ for all $\lambda \in [E-\eta,E+\eta]$. \end{proof}

Lemma~\ref{le:TraceQ} is a central motivation for the ``multi-resolvent local laws" for random matrices (see notably \cite{zbMATH07946620} and the recent survey \cite{erdos2025lecture}), which are concentration bounds on alternating products of the form $K_1R^{z_1}K_2\ldots K_\ell R^{z_\ell}$.

\subsection{Proof ideas}\label{s:proofideas}
Let us explain the key ideas of the rest of the proof of our main results. In this explanatory section, we focus on the proof of Theorem~\ref{t:QEstrongCV}, but large parts of the proofs of Theorems \ref{t:asympdecorr} and \ref{t:HNperp} are actually shared with this proof. To present the key ideas, we make the simplifying assumption that $K_N$ is a real-valued diagonal observable, which we denote by $a_N$. We assume without loss of generality that the average of $a_N$ vanishes, $\langle a_N \rangle=0$. 
For given $\eta \in (0,1)$, we decompose $I$ into intervals of length $\eta$ as in Section~\ref{s:generalbound}. Let $E\in\R$ be the center of one of these intervals, and $z=E+i\eta$. Let $P_N = \rho_N(p)$ 
and $R^z_N=(P_N-z\Id)^{-1}$ be the resolvent of $P_N$. Then, $\eta\Im R^z_N$ is a smooth approximation of the spectral projector onto the interval $J_E^\eta=[E-\eta,E+\eta]$. We pick a polynomial $f$ such that
\begin{equation}\label{e:Imgz222}
\| \eta\Im R^z_N-f(P_N)\|_{\rm op} \leq \eta.
\end{equation}
This polynomial $f$ depends on $\eta$ and $E$. We estimate
\begin{equation}\label{e:TrQaQa222}
\frac{1}{N\eta}{\rm Tr}(a_Nf(P_N)a_N^*f(P_N))=\frac{1}{N\eta}\sum_{x,y\in [N]} \overline{a_N(x)}|f(P_N)(x,y)|^2a_N(y) .
\end{equation}
A sufficiently good control of this trace implies quantum mixing, according to \eqref{e:LI1I2Tr}, since the number of eigenvalues contained in $J_E^\eta$ is of order $N\eta$.

We also consider the operator $f(P)$ acting on $\ell^2(\Gamma)$, it can be written as  $f(\lambda(p))= \lambda (f(p))$ with $f(p) \in \mathbb C[\Gamma]$. 
We define the subset $\Bad_N\subset [N]$ as the set of points $x\in [N]$ for which there exists $y\in [N]$ such that 
$$
\left|\{g\in\Gamma \mid x=g.y \text{ and } f(p)_g\neq 0\}\right|>1.
$$
These are the points around which $\Sch(\Gamma,S,\rho_N)$ does not look like $\Cay(\Gamma,S)$.
Using the Benjamini-Schramm convergence assumption \eqref{e:convN} and the fact that $f$ is a polynomial, one can show that $|\Bad_N|=o(N)$.

The next key point is to express the operator appearing in the right-hand side of \eqref{e:TrQaQa222}, whose entries are $|f(P_N)(x,y)|^2$, as $\rho_N(q)$ for some $q\in\mathbb{C}$ related to $p$. For this, we observe that if $x\in [N]\setminus \Bad_N$, then for any $y\in [N]$,
\begin{equation}\label{e:Good222}
|f(P_N)(x,y)|^2=\Bigl|\sum_{g\in\Gamma, \ x=g.y} f(p)_g\Bigr|^2=\sum_{g\in\Gamma, \ x=g. y} |f(p)_g|^2 = \rho_N(q)(x,y)
\end{equation}
with $q= \sum_{g\in \Gamma}|f(p)_g|^2g$. The middle inequality is due to the fact that the sum in \eqref{e:Good222} contains at most one element.
To estimate the contribution of the points of $\Bad_N$ to the trace \eqref{e:TrQaQa222}, we use that for any $x\in [N]$ 
\begin{equation}\label{e:Bad222}
\sum_{y\in [N]} |f(P_N)(x,y)|^2 |a_N(y)|\leq \sum_{y\in [N]} |f(P_N)(x,y)|^2 =  (f(P_N)f(P_N)^*)(x,x)\leq 2
\end{equation}
where in the last inequality we use $\eta \| \Im R^z_N \|_{\rm op}\leq 1$ and \eqref{e:Imgz222}. Inequality \eqref{e:Imgz222} also holds for $P$, so if $\widetilde{a}_N(x)=a_N(x)\mathbf{1}(x\in \Bad_N)$, we get similarly
\begin{equation}\label{e:veryBad}
\Big\langle \frac{\widetilde{a}_N}{\sqrt{|\Bad_N|}},\rho_N(q)\frac{a_N}{\sqrt{N}}\Big\rangle \leq \|\rho_N(q)\|_{\rm op}\leq \|q\|_{1}=(f(\lambda(p))f(\lambda(p))^*)(e,e)\leq 2.
\end{equation}

So finally combining \eqref{e:Good222}, \eqref{e:Bad222} and \eqref{e:veryBad}, we get
\begin{equation}\label{e:etaA(s)222}
\begin{aligned}
\frac{1}{N\eta}{\rm Tr}(f(P_N)a_Nf(P_N)a_N^*)&\leq \frac{1}{\eta}\Big\langle \frac{a_N}{\sqrt{N}}, \rho_N(q)\frac{a_N}{\sqrt{N}}\Big\rangle + \frac{4}{\eta}\sqrt{\frac{|\Bad_N|}{N}}\\
&\leq \frac{1}{\eta}\|\rho_N(q)_{|1^\perp}\|_{\rm op}+\frac{1}{\eta}o_N(1)\\
&\leq \frac{1}{\eta}\|\lambda(q)\|_{\rm op} (1+o_N(1))+\frac{1}{\eta}o_N(1)
\end{aligned}
\end{equation}
where in the last line we use strong convergence \eqref{e:AN1perpconv} and that $q$ has finite support. Due to \eqref{e:Imgz222},
\begin{equation}\label{e:A(s)op222}
\|\lambda(q)\|_{\rm op} = \bigg\|\sum_{g\in \Gamma}|f( p)_g|^2\lambda(g)\bigg\|_{\rm op} \leq \bigg\|\sum_{g\in\Gamma} \eta^2(\Im R^z)(e,g)^2 \lambda(g)\bigg\|_{\rm op} + C\eta^{3/2}.
\end{equation}
The precise remainder term in $\eta^{3/2}$ actually requires an additional argument which we do not detail in this sketch of proof.
Applying the rapid decay property (more precisely, Lemma~\ref{l:rdopnorm}), we get, for any $C'_1 > 2C_1 + 1$, for some $C>0$
\begin{equation}\label{e:utilRDeta222}
\bigg\|\sum_{g\in\Gamma} \eta^2(\Im R^z)(e,g)^2 \lambda(g)\bigg\|_{\rm op}\leq C \sqrt{\sum_{g\in\Gamma} \eta^4(\Im R^z)(e,g)^4|g|^{C_1'}}.
\end{equation}
Combining \eqref{e:etaA(s)222}, \eqref{e:A(s)op222} and \eqref{e:utilRDeta222}, we have obtained
\begin{equation}\label{e:pourcalculertaux}
\frac{1}{N\eta}{\rm Tr}(f(P_N)a_Nf(P_N)a_N^*)\leq C\Bigl(\sqrt{\sum_{g\in\Gamma} \eta^2(\Im R^z)(e,g)^4|g|^{C_1'}}+C\eta^{1/2}\Bigr) (1+o_N(1))+\frac{1}{\eta}o_N(1).
\end{equation}
If we take the successive limits $\lim_{\eta\rightarrow 0}\lim_{N\rightarrow +\infty}$ of the right-hand side, we obtain $0$ thanks to the control \eqref{e:RDrenforceeweak} of the $4$-th moment of the resolvent, and therefore
$$
\lim_{\eta\rightarrow 0}\limsup_{N\rightarrow +\infty} \frac{1}{N\eta}\sum_{x,y\in [N]} \overline{a_N(x)}|f(P_N)(x,y)|^2a_N(y)=0.
$$
The bounds are explicit and uniform over the intervals of length $\eta$ contained in $I$. 
This is sufficient to conclude the proof of Theorem~\ref{t:QEstrongCV}, due to \eqref{e:LI1I2Tr}.

\subsection{Approximation of resolvent}

Our primary goal in the sequel is to find upper estimates on the quantity $\Tr (K f_1(A) K^* f_2(A))$ appearing in Lemma~\ref{le:TraceQ}. For this, we take polynomials $f_j$, $j=1,2$, such that $f_j(A)$ approximates the imaginary part of the resolvent at $z = E_j +  i \eta$. That is, we take $f_j$ approximating the imaginary part of
$$
g^z (\lambda) = \frac{1}{\lambda - z},
$$
defined for $\lambda\in\R$ and depending on the parameter $z \in \mathbb{H}=\{z\in \mathbb{C}:\Im z>0\}$. For this, it would be sufficient to use Weierstrass's approximation theorem, but we prove in this section a more precise approximation lemma, which is a consequence of Jackson's approximation Theorem. This refined statement is convenient for Section~\ref{s:proofoftht:HNperp}, and is even necessary to obtain quantitative convergence rates in Appendix \ref{a:convrate}.
We denote by $\mathbb{C}_n[X]$ (or $\mathbb{R}_n[X]$) the set of polynomials in $\mathbb{C}[X]$ (or $\mathbb{R}[X]$) of degree at most $n$.

\begin{lemma}\label{l:approxpolpos}
There exists a numerical constant $c >0$ such that the following holds. For any $a > 0$, $\eta > 0 $, any integer $n \geq 1$ and $ 0 <\varepsilon \leq 1$, if $n\varepsilon \geq c \max( a / \eta ,1 )$, then for any $z=E+i\eta$ with $E\in\R$  there exists a non-negative polynomial $s_{z,n} \in \mathbb{R}_{2n}[X]$ such that
\begin{equation}\label{e:apppolres}
\sup_{\lambda \in [-a,a]}  | \eta \Im g^z (\lambda)  - s_{z,n} ( \lambda) | \leq e^{-1/\varepsilon}.
\end{equation}
\end{lemma}
\begin{proof}
Using translation and dilatation, we may assume without loss of generality that $a =1$ and $E=0$, without affecting the constant $c$. We set
$$
h_z(\lambda)=\sqrt{\eta\Im g^z(\lambda)}=\left(\frac{\eta^2}{(E-\lambda)^2+\eta^2}\right)^{\frac12}.
$$
From Jackson's approximation Theorem \cite[Chap. 7, \textsection 8]{zbMATH00477682}, for any integer $n\geq 1$, there exists a polynomial $p_z \in \mathbb{C}_n[X]$ such that, for any integer $0 \leq k \leq n$, 
$$
\sup_{\lambda \in [-1,1]}  | h_z (\lambda) - p_z ( \lambda) | \leq  \left( \frac{\pi}{2} \right)^k \frac{(n-k+1)!}{(n+1)!}
\| \partial^{(k)} h_z \|_\infty,
$$
where $\| \partial^{(k)} h_z \|_\infty$ is the $L^\infty$-norm of the $k$-th derivative of $h_z$. We have $\| \partial^{(k)} h_z \|_\infty = \eta^{-k}\|\partial^{(k)}f\|_{\infty}$ where $f:x\mapsto (1+x^2)^{-1/2}$ and
\begin{equation}\label{e:partialf}
\|\partial^{(k)}f\|_{\infty}\leq (5k)^k
\end{equation}
(see the proof of \eqref{e:partialf} below). If $ 1\leq k \leq n/2$, we get 
$$
\sup_{\lambda \in [-1,1]}  | h_z (\lambda) - p_z ( \lambda) | \leq  \left( \frac{c k }{n\eta} \right)^k
$$
for $c=10\pi$.
We fix $\varepsilon\in[c/n,1]$, so that $k= \lceil 1/\varepsilon \rceil$ satisfies $1\leq k\leq n/2$. For $\eta \geq e c  k / 3n  $, we obtain
$$
\sup_{\lambda \in [-1,1]}  | h_z (\lambda) - p_z ( \lambda) | \leq  (e/3)^{-1/\varepsilon}.
$$
We deduce that $s_{z,n}=p_z^2\in \R_{2n}[X]$ is non-negative and satisfies
$$
| \eta \Im g^z (\lambda) - s_{z,n}( \lambda) |=| h_z(\lambda) - p_z ( \lambda) | |h_z(\lambda) + p_z ( \lambda) |\leq 3(e/3)^{-1/\varepsilon}\leq e^{-1/\varepsilon}.
$$
since $\|h_z\|_\infty\leq 1$.

We finally prove \eqref{e:partialf}. Notice that 
\begin{equation}\label{e:derfracpol}
\left(x^j(1+x^2)^{-n/2}\right)'=jx^{j-1}(1+x^2)^{-n/2}-nx^{j+1}(1+x^2)^{-(n+2)/2}.
\end{equation}
and that $\partial^{(k)} (1 + x^2)^{-1/2}$ is a sum of terms of the form $\alpha_{k,j}x^j(1+x^2)^{-(1+2j+2k)/2}$ with $j\leq k$ and $\alpha_{k,j}\in\R$. From \eqref{e:derfracpol} we deduce
$$
\sum_{j=1}^{k+1} |\alpha_{k+1,j}|\leq   (5k+1) \sum_{j=1}^k |\alpha_{k,j}|,
$$
and since $x^j(1+x^2)^{-(1+2j+2k)/2}$ is bounded above by $1$, we get that 
$$
\| \partial^{(k)} (1 + x^2)^{-1/2} \|_{\infty} \leq \prod_{j=1}^k (5j-4)\leq (5 k )^k,
$$
which concludes the proof of \eqref{e:partialf}.
\end{proof}

\subsection{Resolvent identities}\label{s:resolvent}

We explained in the previous section that the imaginary part of the resolvent appears in the trace computations needed to prove our main results. We gather here some elementary facts regarding resolvents. Let $X$ be a countable set and $P$ be a bounded self-adjoint operator  on $\ell^2 (X)$. In our case, we are interested in $X = \Gamma$, $P = \lambda(p)$ with $ p = p^* \in \mathbb C[\Gamma]$. For $z \in \mathbb C^+ = \{z \in \mathbb C : \Im(z) > 0\}$, we define the resolvent of $P$ as 
$ R^z = ( P - z \mathrm{Id})^{-1} $. We will make use of the following resolvent identity, known as Ward's identity: for any $x \in X$,
\begin{equation}\label{e:ward}
 \sum_{y \in X} \eta |R^z(x,y)|^2= \Im R^z(x,x),
\end{equation}
where $\eta = \Im(z)$.
It is a consequence of the fact that $\sum_{y\in X}   |R^z(x,y)|^2 =  (R^z (R^z)^*)(x,x)$ together with the spectral theorem which implies that $\eta R^z (R^z)^* = \Im R $ since $\Im (1/z) = - \Im (z) / |z|^2$ and $\Im (P -z \mathrm{Id}) = - \eta \mathrm{Id}$, $|P - z \mathrm{Id}|^{-2} =  R^z (R^z)^*$. 
In the same vein,  writing $z  = E + i \eta$ and denoting by $\mu_P^{\delta_x}$ the spectral measure at vector $\delta_x$ (see \eqref{e:muPpsi}), we have
\begin{align}
 \sum_{y \in X}  \eta |(\Im R^{z})(x,y)|^2 & = \int \frac{\eta^3}{(|\lambda-E|^2+\eta^2)^2} \dd\mu_P^{\delta_x}(\lambda)  \nonumber
 \\
 &\leq  \int \frac{\eta}{|\lambda-E|^2+\eta^2} \dd\mu_P^{\delta_x}(\lambda) \nonumber \\
 &  = \Im R^z(x,x).
\label{eq:wardineq}\end{align}
 
 Note that, for any real probability measure $\mu$ on $\mathbb R$, for Lebesgue almost all $E \in \mathbb R$, we have 
$$\int \frac{\eta^3}{(|\lambda-E|^2+\eta^2)^2} \dd\mu (\lambda)\underset{\eta\rightarrow 0}{\longrightarrow} \frac12 \frac{\dd \mu}{\dd \lambda}(E)
$$

We note also, that since $\Im R^z$ is a non-negative operator, we have, for any $x,y \in X$
$$
(\Im R^z) (x,y) \leq \sqrt{\Im R^z(x,x)\Im R^z(y,y) }.
$$
In particular,  we get from \eqref{eq:wardineq}
\begin{align*}
     \sum_{y \in X}  \eta |(\Im R^{z})(x,y)|^4 & \leq   \Im R^z (x,x) \sup_{y \in X}  \big( \Im R^z (y,y) \big)\sum_{y \in X}  \eta | (\Im R^{z})(x,y)|^2  \\
     & = ( \Im R^z (x,x) )^{2}\sup_{y \in X}  \big( \Im R^z (y,y) \big). 
\end{align*}
This last inequality implies \eqref{eq:AC0} in the case $X = \Gamma$, $P = \lambda(p)$ with $p = p^*$.

\section{Convergence of Schreier graphs}
\label{sec:Schreier}

In this section, as usual, $\Gamma$ is a finitely generated group with unit $e$, $\mathbb C [\Gamma]$ its group algebra and $\lambda$ the left regular representation. Along an infinite subsequence of integers $N$, we consider a permutation representation $\rho_N\in {\rm Hom}(\Gamma,S_N)$ whose action on $ \INT{N}$ is denoted by $g.x := \rho_N(g)(x)$. In this section, we gather geometric and spectral consequences of the convergence in distribution of $\rho_N$ toward the left regular representation $\lambda$ of $\Gamma$.

\subsection{Benjamini-Schramm convergence}

The convergence in distribution as defined in \eqref{e:convN} is equivalent to the convergence of the character of $\rho_N$ toward the character of $\lambda$ since $\langle \delta_e , \lambda(g) \delta_e \rangle = \IND ( g = e )$. In particular, by linearity, the convergence in distribution is equivalent to: 
\begin{equation}\label{e:convNbis}
\forall  p \in \mathbb C [\Gamma], \quad \frac 1 N \Tr (\rho_N(p)) \underset{N\rightarrow +\infty}{\longrightarrow} \langle \delta_e , \lambda(p) \delta_e \rangle  =p_e.
\end{equation}
Restricted to our setting, this notion of convergence coincides with the convergence in distribution in operator algebras, see for example \cite[Definition 4.0.1]{zbMATH00428977}. 

The aim of this section is to show that this convergence in distribution is equivalent to the {\em Benjamini-Schramm convergence} of rooted marked graphs. Some basic definitions are in order.  A graph $G = (V,E)$ is the pair formed by a countable vertex set $V$, and $E$ a collection of directed edges: an edge $e$ has a start $e_- = u \in V$ and an end $e_+ =v \in V$. The set $E$ is equipped with an involution $e \mapsto e^{-1}$ such  that $e^{-1}_{\pm} = e_{\mp}$. A marked graph is a triple $G = (V,E,\xi)$ where $\xi : E \to X$ is a map and the mark set $X$ is a finite set (for simplicity). 

Let $S = S^{-1}$ be a finite symmetric set of generators of $\Gamma$. With these definitions, the graphs $\Cay(\Gamma,S)$ and $\Sch(\Gamma,S,\rho_N)$ are defined as marked graphs on the mark set $S$: the edge $e = (g,s)\in \Gamma \times S$  or $e = (x,s) \in \INT{N}\times S$ receives the mark $\xi(e) = s \in S$. In $\Cay(\Gamma,S)$, if $e = (g,s)$, we set $e_- = g$ and $e^{-1} = (s.g,s^{-1})$, and similarly in $\Sch(\Gamma,S,\rho_N)$ for $e = (x,s) \in \INT{N}\times S$. Note that  since $S = S^{-1}$, $\Cay(\Gamma,S)$ and $\Sch(\Gamma,S,\rho_N)$ are symmetric graphs in the sense that for all edges $\xi(e^{-1}) = \xi(e)^{-1}$.

Next, a graph $G = (V,E)$ is locally finite if for all $v \in V$, the degree of $v$ (the number of adjacent edges) is finite. A rooted marked graph $G = (V,E,\xi,o)$ is a connected marked graph $(V,E,\xi)$ and a distinguished vertex $o \in V$. If $G = (V,E,\xi)$ is a locally finite marked graph, for $r \geq 0$  and $v \in V$, we denote by $(G,v)_r$ the rooted graph obtained by the restriction of $G$ to vertices at graphical distance at most $r$ from $v$. Since $G$ is locally finite, $(G,v)_r$ is a finite rooted graph. We say that two rooted marked graphs $G_i = (V_i,E_i,\xi_i,o_i)$, $i=1,2$ are isomorphic, which we denote by $G_1 \sim G_2$, if there exists a graph isomorphism $\sigma : E_1 \to E_2$  (adjacent edges are mapped to adjacent edges) such that $\sigma(G_1) = G_2$. In combinatorial terminology, an equivalence class is an unlabeled rooted marked graph. 
We say that $G_N = \Sch(\Gamma,S,\rho_N)$ converges in the Benjamini-Schramm (BS) sense toward $G = \Cay(\Gamma,S)$, if 
\begin{equation}\label{eq:convBS}
\forall r \geq 0 , \quad  \frac1N \sum_{x =1}^N \IND ( (G_N,x)_r \sim (G,e)_r ) \underset{N\rightarrow +\infty}{\longrightarrow} 1.
\end{equation}
This notion of convergence coincides with the Benjamini-Schramm local weak convergence for finite graph sequences, see for example \cite[Chapter 2]{zbMATH07782966}. For simplicity, it is stated here in the setting of Schreier graphs.

As advertised above, the two notions of convergence are equivalent.

\begin{proposition}\label{p:conv} Let $S = S^{-1}$ be a finite symmetric set of generators of $\Gamma$. Then $\Sch(\Gamma,S,\rho_N)$   converges to  $G = \Cay(\Gamma,S)$ in the Benjamini-Schramm sense if and only if $\rho_N$ converges in distribution toward $\lambda$.
\end{proposition} 

Note that Proposition~\ref{p:conv} notably implies that if for some symmetric generating set $S$, $\Sch(\Gamma,S,\rho_N)$   converges locally to  $G = \Cay(\Gamma,S)$ then for all finite symmetric generating set $S$, the convergence occurs. In other words, for Schreier graphs the local convergence toward the corresponding Cayley graph does not depend on the specific choice of the generating sets.

\begin{proof}[Proof of Proposition~\ref{p:conv}]
We first prove that convergence in distribution implies BS convergence. We fix some integer $r$, let $B = B(e,r)$ be the ball of radius $r$ in $G = \Cay(\Gamma,S)$ with center $e$. As in \eqref{eq:convBS}, we denote by $(G,e)_r$ the rooted graph obtained by the restriction of $G$ to $B(e,r)$.
Similarly, for $x \in [N]$, the ball of radius $r$ and center $x$ in $G_N = \Sch(\Gamma,S,\rho_N)$ is $B^x_{N} = \{ y \in \INT{N} : y = \rho_N(g)(x), \hbox{ for some } g \in B\}$ and $(G_N,x)_r$ is the rooted graph obtained by the restriction of $G_N$ to $B_N^x$.

We consider the map $\psi_N^x : B \to B_N^x$ defined by $\psi^x_N (g) = \rho_N(g)(x)$. Note that by  construction $\psi_N^x$ is surjective and $\psi_N^x ( (G,e )_r ) = (G_N,x)_r$, that is $\psi_N^x$ is a covering map. If it is injective then by definition, we have  $(G_N,x)_r \sim (G,e )_r$. However, if $\psi_N^x$ is not injective, then there exist $g \ne h \in B$ such that $\rho_N(g)(x) = \rho_N(h)(x)$. In other word, from the morphism property, $\rho_N( g h^{-1})(x) = x$. Since 
$$
\Tr(\rho_N(g) ) = \sum_{x =1}^N \IND ( \rho_N(g) (x)  = x), 
$$
we deduce that 
$$
 \sum_{x =1}^N \IND ( (G_N,x)_r \not\sim (G,e)_r )  \leq \sum_{g \in B_2 \backslash \{e \} } \Tr(\rho_N(g) ).
$$
where $B_2 = B(e,2r)$ is the ball of radius $2 r$ in $G = \Cay(\Gamma,S)$ with center $e$. This proves that \eqref{e:convN} implies \eqref{eq:convBS}.

Conversely, let $g \ne e$ and $x \in \INT{N}$. Since $S$ is a generating set of $\Gamma$, there exists some  $r$ such that $g \in B = B(e,r)$, that is $g = g_{i_1}\cdots g_{i_r}$ with $g_{i_t} \in S$. In particular, if $\rho_N(g) (x) = x$ then $(G_N,x)_r \not\sim (G,e)_r$ because we have a closed path of length $r$ in $G_N$ starting at $x$ obtained by following the marks $(g_{i_1},\ldots,g_{i_r})$ and the corresponding path started at $e$ is not closed in $G$. In particular, 
$$
 \Tr(\rho_N(g) ) \leq \sum_{x =1}^N \IND ( (G_N,x)_r \not\sim (G,e)_r ).
$$
This concludes the proof.
\end{proof}

\subsection{The number of eigenvalues in small intervals}

\label{subsec:eigenvalues}
In this subsection, we take $p \in \mathbb{C}[\Gamma]$ self-adjoint $p = p^*$. We assume the spectral condition \eqref{eq:AC} holds on an interval $I_0$. We denote by $\mu_p$ the spectral measure of $\lambda(p)$, that is the probability measure in $\mathbb R$ such that for all $z \in \mathbb{H}$,
\begin{equation}\label{eq:defmup}
R^z (e,e) = \int \frac{1}{\lambda-z} d\mu_p (\lambda), 
\end{equation}
where $R^z = (\lambda(p) -z\Id)^{-1}$ is the resolvent operator of $\lambda(p)$. Let $I$ be a closed interval in the interior of $I_0$. By assumption \eqref{eq:AC}, $\mu_p$ is purely  absolutely continuous inside $I$ with a density $f_p$ with respect to the Lebesgue measure satisfying for all $\lambda \in I$,
$$
{C'_0}^{-1} \leq f_p(\lambda) \leq C'_0
$$
where $C'_0 = \pi C_0$ and $C_0 $ as in \eqref{eq:AC} (indeed, for almost all $\lambda$, $f_p(\lambda) = \lim_{\eta \to 0} \frac{1}{\pi} \Im R^{\lambda + i\eta} (e,e)$). In particular, 
\begin{equation}\label{eq:ACJ}
{C'_0}^{-1} |J|  \leq   \mu_p (J) \leq {C'_0} |J| \quad \hbox{ for all Borel $J \subset I$}.
\end{equation}

We next take as usual $\rho_N$ converging in distribution toward $\lambda$ as $N\to \infty$.
For an interval $J \subset I$, our goal in this subsection is to compare the number of eigenvalues of $\rho_N(p)$ in $J$ and $N \mu_p (J)$. For this, we prove the weak convergence to $\mu_p$ of the empirical distribution of eigenvalues of $\rho_N(p)$. More precisely, we establish the following uniform result.

\begin{lemma}\label{cor:CMS2}
    Let $p \in \mathbb{C}[\Gamma]$ with $p = p^*$. If $\rho_N$ converges in distribution to $\lambda$ and if \eqref{eq:AC} holds and $I \subset I_0$ is an interval whose closure is contained in the interior of $I_0$, then 
$$
\frac{1}{{C'_0}} \leq \liminf_{\eta \to 0} \liminf_{N\to \infty}  \inf_{ J \subset I : |J|  = \eta} \frac{|\Lambda_{J}|} {N \eta } \leq \limsup_{\eta \to 0}  \limsup_{N\to \infty} \sup_{ J \subset I : |J| = \eta} \frac{|\Lambda_{J}|} {N \eta } \leq {C'_0},
$$
where $J\subset I$ is an interval and $C'_0 = \pi C_0$. 
\end{lemma}

\begin{proof}
Set $P_N = \rho_N(p) \in M_N(\mathbb C)$. The empirical distribution of eigenvalues of $P_N$ is
\begin{equation}\label{e:mupN}
\mu_{P_N} = \frac{1}{N} \sum_{\alpha =1}^N \delta_{\lambda_\alpha}.
\end{equation}
In other words, for any  $J \subset \mathbb{R}$, $\mu_{P_N}(J)$ is equal to $|\Lambda_J^{(N)}|/N$. The spectral measure at unit vector $\psi \in \mathbb S^{N-1}$ is the probability measure 
\begin{equation}\label{e:muPpsi}
\mu_{P_N}^\psi = \sum_{\alpha =1}^N \delta_{\lambda_\alpha} | \langle  \varphi_\alpha, \psi \rangle |^2,
\end{equation}
where $(\varphi_\alpha)$ is an orthonormal basis of eigenvectors of $P_N$. In particular, the spatial average of spectral measures is the empirical distribution of eigenvalues: 
\begin{equation}\label{e:averageP}
\frac 1 N \sum_{x=1} ^N \mu_{P_N}^{\delta_x} = \mu_{P_N}.
\end{equation}

Note that for any continuous function
$$
\langle \psi , f (P_N) \psi \rangle = \int f(\lambda) d \mu_{P_N}^{\psi}. 
$$
In particular, 
$$
\int \lambda^k  d \mu_{P_N}^{\delta_x} = \langle \delta_x , P_N^k \delta_x \rangle = \sum_{g_{i_1},\ldots ,g_{i_k}} p_{g_{i_1}} \cdots p_{g_{i_k}} \IND ( x = g_{i_1} \cdots g_{i_k}  x). 
$$
Similarly, since $\mu_p$ is the spectral measure of $\lambda(p)$ at vector $\delta_e$:
$$
\int \lambda^k  d \mu_p = \langle \delta_e , \lambda(p)^k \delta_e \rangle = \sum_{g_{i_1},\ldots ,g_{i_k}} p_{g_{i_1}} \cdots p_{g_{i_k}} \IND ( e = g_{i_1} \cdots g_{i_k} ). 
$$

Next, let $\Bad_{S} (n)$ be the set of $x \in \INT{N}$ such that $g.x = x$ for some $g \ne e$, $g \in B_S(2n)$, the ball of radius $2n$ in $\mathrm{Sch}(\Gamma,S,\rho_N)$. Remark that if $x \notin \Bad_{S}(n)$ then the ball of radius $n$ in $\mathrm{Sch}(\Gamma,S,\rho_N)$ is isomorphic to $B_S(n)$ (as argued in the proof of Proposition~\ref{p:conv}).
It follows that if $x \notin \Bad_{S}(n)$, we have for all $k \in \INT{2n}$,
$$
\int \lambda^k d \mu^{\delta_x}_{P_N} = \int \lambda^k d \mu_p  .
$$
We deduce thanks to \eqref{e:averageP} that $\mu_{P_N}$ converges weakly to $\mu_p$ as $N\rightarrow +\infty$.

Let us fix $\eta>0$ and show that
\begin{equation}\label{e:justeuneeq}
\limsup_{N\rightarrow \infty} \sup_{J\subset I : |J|=\eta} \frac{|\Lambda_J^{(N)}|}{N\eta}\leq C_0'.
\end{equation}
Assume for the sake of a contradiction that there exists $\varepsilon>0$ and (omitting the extraction of a sequence in the notation) a sequence $J_N$ of intervals with $|J_N|=\eta$ such that $\mu_{P_N}(J_N)\geq (C_0'+\varepsilon)\eta$. Let $J$ be an interval of length $\eta$ obtained as the limit (up to extraction of another subsequence, omitted in the notation too) of this sequence of intervals.  Take $U$ an open set of length $\leq \frac{\varepsilon\eta}{4C_0'}$ such that $J_N\setminus J \subset U$ for infinitely many $N$. Since $\mu_{P_N}(J)\rightarrow \mu_p(J)\leq C_0'\eta$, we have $\mu_{P_N}(J)\leq (C_0'+\frac{\varepsilon}{2})\eta$ for $N$ large enough, hence we get 
$$
\mu_{P_N}(U)\geq \mu_{P_N}(J_N\setminus J)\geq \frac{\varepsilon\eta}{2}.
$$
But $\mu_{P_N}(U)\rightarrow \mu_p(U)\leq \frac{\varepsilon \eta}{4}$ by weak convergence, which yields a contradiction. Therefore, \eqref{e:justeuneeq} holds for any $\eta>0$, which gives the right-hand side inequality in the lemma, and the left-hand side inequality can be proved in the same way.
 \end{proof}

\begin{remark}\label{rk:mostafa} 
In the proof of our main theorems, the lower bound of \eqref{eq:AC} is  only used to prove the lower bound in Lemma~\ref{cor:CMS2}.
Recalling the definition of $L^{\tau,\eta}_I$ in \eqref{eq:defLL0}, let us explain how in order to prove that   
\begin{equation}
    \lim_{\eta\to 0}\lim_{N\to\infty}  L^{\tau,\eta/2}_I(P_N,K_N) = 0 \label{eq:popea}
\end{equation}(meaning that \eqref{eq:QMZ} and \eqref{eq:QEKN} hold), the  uniform lower bound in \eqref{eq:AC} can be replaced by the weaker assumption $\mu_p(I)>0$.
For a fixed interval $I=[a,b]$, we divide $I$ into $s_\eta \le \frac{2(b-a)}{\eta}+1$ intervals $I_\eta^{E_r} = [E_r-\frac{\eta}{2},E_r+\frac{\eta}{2}]$. Recalling $J_\eta^{E_r} =[E_r-\eta,E_r+\eta]$, instead of \eqref{eq:LoL}, we write
\begin{align*}
L_I^{\tau,\eta/2}(A,K) &= \sum_{r=1}^{s_\eta} \frac{N\eta}{|\Lambda_I|} \cdot  \Big(\frac{1}{N\eta}\sum_{\lambda_{\alpha}\in I_{\eta}^{E_r}}  \sum_{\substack{\lambda_\beta\in I\\|\lambda_\beta-\lambda_\alpha-\tau|<\frac{\eta}{2}}} | \langle \varphi_\beta, K \varphi_\alpha \rangle|^2\Big)\\
&\le \frac{N\eta}{|\Lambda_I|} \sum_{r=1}^{s_\eta}   \Big(\frac{1}{N\eta}\sum_{\substack{\lambda_{\alpha}\in J_{\eta}^{E_r}\\\lambda_\beta\in (J_\eta^{E_r}+\tau)\cap I}}   | \langle \varphi_\beta, K \varphi_\alpha \rangle|^2\Big).
\end{align*}
In proving our main results, we show that if the upper bound in \eqref{eq:AC} holds, then we have, for an eigenvector basis of $P_N = \rho_N(p)$ and a sequence of centered observables $K_N$,
\begin{equation}\label{eq:dedeap}\frac{1}{N\eta}\sum\limits_{\lambda_{\alpha}\in J_{\eta}^{E_r}, \lambda_\beta\in (J_\eta^{E_r}+\tau\cap I)} | \langle \varphi_\beta, K_N \varphi_\alpha \rangle|^2\le \Delta(\eta,N)
\end{equation} uniformly in $r$ and $\tau$, with $\lim\limits_{\eta\to 0}\lim\limits_{N\to\infty} \Delta(\eta,N)=0$.  Since $\eta s_\eta \leq 2 (b-a) +\eta$, it follows that for all $0< \eta \leq 1$,
\[
L_I^{\tau,\eta/2}(P_N,K_N) \le \frac{C \Delta(\eta,N)}{\mu_{P_N}(I)},
\]
where $C =2(b-a)+1$, and $\mu_{P_N}(I)=|\Lambda_I|/N$. Now, as shown in the proof of Lemma~\ref{cor:CMS2}, $\mu_{P_N}(I) \to \mu_p(I)$. Hence if $\mu_p(I) >0$, we obtain \eqref{eq:popea} from \eqref{eq:dedeap}.
\end{remark}

\subsection{Trace estimate for local operators}
\label{s:computtrace}

In this section, we formally justify the computations expressing the trace ${\rm Tr}(a_Nf(P_N)a_N^*f(P_N))$ as 
$\langle a_N, \rho_N(q)a_N\rangle$ for some $q\in \mathbb{C}[\Gamma]$, up to Benjamini-Schramm terms (i.e., we justify how to go from \eqref{e:TrQaQa222} to the first line in \eqref{e:etaA(s)222}). This is a key step in the proofs of Theorems \ref{t:asympdecorr}, \ref{t:HNperp} and \ref{t:QEstrongCV}. The notation is heavier than in Section~\ref{s:proofideas}, because we provide statements for general $T$-local observables in Lemma~\ref{le:traceBS} and Corollary \ref{cor:traceBS}, rather than diagonal ones. But the ideas are the same, and we actually somehow reduce the computation to the diagonal case using expansion \eqref{eq:defkernelA} below.

Let $T \subset \Gamma$ be a finite subset (not necessarily symmetric). We say that $A \in M_N(\mathbb C)$ is $T$-local if $A(x,y) = 0$ unless for some $g \in T$, $x = g.y$. If $A$ is $T$-local, we call a kernel of $A$ any element $a   = \sum_g a_{g} g  \in \mathbb{C}^N[\Gamma]$ such that 
\begin{equation}\label{eq:defkernelA}
A (x,y) = \sum_{g \in T} a_{g}(x) \IND ( x = g.y) 
\end{equation}
where $a_g\in \mathbb{C}^{N}$.
Such a decomposition always exists: if for fixed $x,y$ there is a unique $g$ such that $x=g.y$, then $a_g(x)$ is uniquely defined by $a_g(x) := A(x,g^{-1}.x)$. This will be the typical situation considered below. If not, the decomposition is not unique. One could for example pick arbitrarily one $g$ such that $x=g. y$, put $a_{g}(x) = A(x,g^{-1}.x)$ and $a_t(x)=0$ for $t\neq g$ such that $t^{-1}.x=g^{-1}.x=y$. In the sequel, we make the more canonical choice $a_g(x) = A(x,g^{-1}.x)/m_g(x)$ where $m_g(x):= \sum_{t\in T} \mathbf{1}(g^{-1}.x=t^{-1}.x)\ge 1$.

For example, the operator $K_N$ defined in \eqref{eq:defKN} is $S$-local.  Also, if $p \in \mathbb {C}[\Gamma]$ and the support of $p$ is $T$ then $\rho_N(p)$ is $T$-local with kernel $p$ by construction. More generally, a Schr\"odinger operator of the form $A_N = D_N + \rho_N(p)$ and $D_N$ a diagonal operator is a $(T \cup \{e \})$-local operator.

We will use the following norms in $M_N(\mathbb C)$: for $A \in M_N(\mathbb C)$, $\| A \|_{\rm op} = \sup_{f \ne 0} \| A f \|_2 / \| f \|_2$ denotes the operator norm ($\ell^2$ to $\ell^2$ norm), $\| A \|_{1,\infty} = \sup_{x,y} |A(x,y)|$ is the $\ell^1$ to $\ell^\infty$ norm. If $A$ is $T$-local and $a \in \mathbb{C}^N[\Gamma]$ is a kernel of $A$, we set 
$$
\| a \|_\infty = \sum_{g} \| a_g \|_\infty \leq |T| \|A \|_{1,\infty}.$$

For $g \in \Gamma$, we set $\FIX_N(g) = \FIX_N(g^{-1}) = \{ x \in \INT{N} : g.x = x\}$ the set of fixed points of the permutation $\rho_N(g)$. If $T \subset \Gamma$, we set 
\begin{equation}\label{eq:defFIX}
    \FIX^*_N(T) = \FIX^*_N(T^{-1}) = \cup_{g \in T \backslash \{e \}} \FIX_N(g).
\end{equation}
By definition,  $\rho_N$ converges in distribution  toward $\lambda$, if and only if for any finite subset $T \subset \Gamma$, we have
$$
\frac{  |\FIX^*_N(T)|}{N} \to 0
$$
as $N\rightarrow +\infty$. 
Finally, if $S, T$ are subsets of $\Gamma$, we set $ST = \{ s.t : s \in S , t \in T\}$.

The goal of this section is to prove Corollary~\ref{cor:traceBS}, which will be used in conjunction with Lemma~\ref{le:TraceQ}. The result, roughly speaking, rephrases the trace in a dynamical way, as a scalar product of the observables $\langle K_2, Q_NK_1\rangle$, where $Q_N$ is a propagator living on the finite graph. From there, we will either use Cauchy-Schwarz and relate $Q_N$ to a limiting operator $Q$ with nice decay properties using the convergence assumptions (Theorems~\ref{t:HNperp}-\ref{t:QEstrongCV}), or use more specific properties of $K_i$ to conclude (Theorem~\ref{t:asympdecorr}). We need our propagator $Q_N$ to have the form $\rho_N(q)$ for some $q$ to be able to apply our assumptions, and this will hold true up to Benjamini-Schramm errors, which are controlled thanks to the assumption of convergence in distribution. In the case of diagonal observables $K_1,K_2$ (corresponding to  $T_1=T_2=\{e\}$ below), the statement simplifies dramatically and boils down to \eqref{e:TrQaQa222}-\eqref{e:Good222}-\eqref{e:Bad222}.

\begin{lemma}\label{le:traceBS}
For $i =1,2$, let $T_i,S_i$ be finite subsets of $\Gamma$, $K_i$ and $P_i$ be $T_i$ and $S_i$-local operators in $M_N(\mathbb C)$ with kernels $k_i$ and $p_i$ respectively.  We have 
\begin{align*}
& \bigg|
 \Tr ( K_1 P_1 K_2^* P_2) - \sum_{t_i \in T_i, i= 1,2} \left\langle k_{2,t_2}, Q_{t_1t_2} k_{1,t_2} \right\rangle \bigg| \\
& \quad \leq   \| k_1\|_{\infty} \| k_2\|_{\infty} \left| \FIX^*_N(T_1S_1T_2^{-1} S_2) \right| \Big( \|P_1\|_{\rm op} \|P_2\|_{\rm op} + \max_{t_1,t_2} \| q_{t_1t_2} \|_{\infty} \Big),
\end{align*}
where $Q_{t_1t_2}$ is the $(t_1 S_1 t_2^{-1} \cap S_2)$-local operator with kernel $q_{t_1t_2,g} (x) = p_{1,t_1^{-1} g^{-1} t_2} (t_1^{-1} g^{-1}. x) p_{2,g} (x)$, that is $Q_{t_1t_2} (x,y) = \sum_g q_{t_1t_2,g} (x) \IND ( x = g.y)$. 
\end{lemma}

\begin{proof}
For $x \in \INT{N}$, we write 
\begin{align*}
& ( K_1 P_1 K_2^* P_2)(x,x)  \\
& = \sum_{x_1,x_2,y}
K_{1} (x,x_1) P_1 (x_1,x_2) \bar K_2(y,x_2 )  P_2 (y,x)\\
&  = \sum_{t_1,g_1,t_2,g_2} \sum_y k_{1,t_1}(x) p_{1,g_1} (t_1^{-1} x) \bar k_{2,t_2} (y)  p_{2,g_2} (y)  \IND( t_1^{-1}  x = g_1 t^{-1}_2 y )  \IND (y = g_2 x) ,
\end{align*}
where the sum is over $t_i \in T_i$, $g_i \in S_i$. Now if $x \notin \FIX_N^* ( T_1S_1T_2^{-1} S_2 ) $, then $ x = g^{-1}_2 t_2 g_1^{-1} t_1^{-1}x$ implies that $g^{-1}_2 t_2 g_1^{-1} t_1^{-1} = e$. In particular, writing $g = g_2$ and $g_1  = t_1^{-1} g^{-1} t_2$,  we find, if  $x \notin \FIX_N^* ( T_1S_1T_2^{-1} S_2) $,
\begin{align*}
 ( K_1 P_1 K_2^* P_2)(x,x) 
 & =  \sum_{t_1,t_2,g} \sum_{y} k_{1,t_1}(x)  \bar k_{2,t_2} (y) p_{1,t_1^{-1} g^{-1} t_2} (t_1^{-1} g^{-1} y)  p_{2,g} (y) \IND ( y = g. x) 
 \\ 
 &  = \sum_{t_1,t_2}   \sum_{y} \bar k_{2,t_2} (y) Q_{t_1,t_2} (y,x) k_{1,t_1}(x) .
\end{align*}
To handle the elements $x\in  \FIX_N^* ( T_1S_1T_2^{-1} S_2)$, we proceed as follows. For any $x \in \INT{N}$,
\begin{align*}
& | ( K_1 P_1 K_2^* P_2)(x,x) |  = \left|  \sum_{t_1,t_2} \sum_{y} k_{1,t_1}(x) \bar k_{2,t_2} (y) P_1 (t_1^{-1} x, t^{-1}_2 y) P_2 (y,x) \right| \\
&  \leq   \sum_{t_1,t_2} |k_{1,t_1}(x)| \sum_{y}  |k_{2,t_2} (y)| | P_1 (t_1^{-1} x, t^{-1}_2 y)| |  P_2 (y,x) | \\
& \leq \sum_{t_1,t_2} |k_{1,t_1}(x)| \|k_{2,t_2} \|_\infty \sqrt{  \sum_{y}   | P_1 (t_1^{-1} x, t^{-1}_2 y)|^2  } \sqrt{ \sum_y  |  P_2 (y,x) |^2 } \\
& =  \sum_{t_1,t_2} |k_{1,t_1}(x)| \|k_{2,t_2} \|_\infty  \sqrt{(P_1 P_1^*) (t_1^{-1}x,t_1^{-1}x)} \sqrt{  (P_2^* P_2) (x,x)  },\end{align*}
where at the third line, we have used Cauchy-Schwarz inequality. Since $|(M^* M)(x,x) | \leq \| M \|^2 = \| M^*\|^2$, we arrive at 
$$
| ( K_1 P_1 K_2^* P_2)(x,x) | \leq \sum_{t_1,t_2} |k_{1,t_1}(x)| \|k_{2,t_2} \|_\infty \| P_1 \|_{\rm op} \| P_2\|_{\rm op} \leq \| k_1 \|_{\infty} \| k _2 \|_\infty \| P_1 \|_{\rm op} \| P_2\|_{\rm op} .
$$
Similarly, 
$$
\sum_{t_1,t_2} \sum_y |k_{1,t_1} (x)| |k_{2,t_2} (y)| |Q_{t_1t_2}(y,x)| \leq \| k_1 \|_{\infty} \|k_2 \|_{\infty} \max_{t_1,t_2} \sum_{y} | Q_{t_1,t_2} (y,x)  |. 
$$
We finally note that $\sum_{y} | Q_{t_1t_2} (y,x)  | \leq \sum_y \sum_g |q_{t_1t_2g} (y)| \IND ( y = g.x) = \sum_g |q_{t_1t_2g} (g.x)|  $. The conclusion follows.
\end{proof}

We will use the following corollary of Lemma~\ref{le:traceBS} where the local operators $P_i$ are taken as functions of the operator $\rho_N(p)$:
\begin{corollary}\label{cor:traceBS}
    Let  $p \in \mathbb C[\Gamma]$ with $p = p^*$, let $S \subset \Gamma$ be a symmetric generated set containing the support of $p$ and let $P = \rho_N(p) \in M_N(\mathbb C)$. For integer $n \geq 1$, let $\Bad_{S}(n) = \FIX_N^*( B_S(2n))$ where $B_S(n)$ is the ball of radius $n$ in $\mathrm{Cay}(\Gamma,S)$. For $i =1,2$, let $K_i \in M_N(\mathbb C)$ be a $B_{S}(r_0)$-local operator   with kernel $k_i$ and let $f_i$ be a polynomial of degree at most $r_1$. Then
\begin{align*}
& \left|
 \Tr ( K_1 f_1(P) K_2^* f_2(P) )  - \sum_{t_i \in T_i, i= 1,2} \langle k_{2,t_2}, Q_{t_1t_2} k_{1,t_2} \rangle \right| \\
& \leq   \| k_1\|_{\infty} \| k_2\|_{\infty} |\Bad_S(r_0+r_1)| \left( \|f_1(P)\|_{\rm op} \|f_2(P)\|_{\rm op} + \|f_1(p) \|_2 \| f_2(p) \|_2  \right),
\end{align*}
where $Q_{t_1t_2}  = \rho_N(q_{t_1t_2})$ with $q_{t_1t_2} = \sum_g q_{t_1t_2,g} g  \in \mathbb C[\Gamma]$, $q_{t_1t_2,g} = \bar f_{1}(p)_{t_2^{-1} g t_1}  f_2(p)_g$.
\end{corollary}
%\textbf{\textcolor{cyan}{I don't think we defined/explained what is $f_i(p)_g$. I suppose $f_i(p)_g:= f_i(P)(g,e)$.}}
\begin{proof}
We have $B_S(r_0)B_S(r_1)B_S(r_0)B_S(r_1) \subset B_S (2 (r_0 + r_1))$ and $f_i( P ) = \rho_N( f_i(p))$ by the morphism property. We apply Lemma~\ref{le:traceBS} with $p_{i,g}(x) = p_{i,g}$ independent of $x \in \INT{N}$. We note also that $f_i(p)_g = \bar f_i(p)_{g^{-1}}$ since $p = p^*$. In particular, $\bar f_{1}(p)_{t_2^{-1} g t_1}  = f_1(p)_{t_1^{-1}g^{-1} t_2}$. We also have that $\sum_{g} |q_{t_1t_2,g}| = \sum_{g}  |f_{1}(p)_{t_2^{-1} g t_1} | | f_2(p)_g| =   \sum_{g}  |f_{1}(p)_{g t_1} | | f_2(p)_{t_2 g}|  \leq \|f_1(p)\|_2 \| f_2 (p) \|_2$ by Cauchy-Schwarz.  The conclusion then follows from Lemma~\ref{le:traceBS}. 
\end{proof}

\section{Proof of Theorems \ref{t:asympdecorr}, \ref{t:HNperp} and \ref{t:QEstrongCV}}\label{s:pfthasymp}

\subsection{Quantum mixing estimate for Schreier graphs} \label{s:quamixestsch}

The proof of Theorems \ref{t:asympdecorr}, \ref{t:HNperp}, \ref{t:QEstrongCV} starts from a common inequality \eqref{eq:TrC} below. This inequality is a by-product of what was established in Sections \ref{sec:general} and  \ref{sec:Schreier}.

Set $P_N = \rho_N(p) \in M_N(\mathbb C)$. To prove quantum ergodicity/mixing for $P_N$ and observable $K_N \in M_N(\mathbb C)$, we can assume without loss of generality that $\|p\|_1 = \sum_g |p_g| = 1$ and that for some $t\in \Gamma$ and $k_N \in \mathbb{C}^N$ with $\| k_N \|_\infty \leq 1$,
\begin{equation}\label{e:kNKN}
K_N (x,y)  = k_N(x) \IND( x = t.y) \quad \hbox{ with } \quad \langle k_N \rangle = \frac{1}{N}\sum_{x=1}^N k_N(x) = 0.
\end{equation}
The general case then follows from the decomposition 
$$
K_N(x,y) = \sum_{t \in T} k_{N,t}(x) \IND ( x = t.y) =  \sum_{t \in T} (k_{N,t}(x) -  \langle k_{N,t} \rangle ) \IND ( x = t.y)  + \langle K_N \rangle (x,y).
$$
 
Due to \eqref{eq:defLL0}, to prove quantum mixing (namely \eqref{eq:QMKN}) it suffices to establish a vanishing upper bound as $N\to \infty$ and then $\eta \to 0$ of
$
L_{J_{E_1}^\eta J_{E_2}^\eta} (P_N ,K_N)$ uniformly over all $E_1,E_2 \in I$ (recall $J_E ^\eta = [E-\eta,E+\eta]$). Moreover, by Lemma~\ref{le:TraceQ} and Lemma~\ref{cor:CMS2}, it suffices to establish such vanishing upper bound for 
\begin{equation}\label{eq:trKPKP}
\frac{1}{N\eta} \Tr (K_N f_1(P_N)  K_N^* f_2 (P_N)),
\end{equation}
where $f_j = f_{j,\eta}$ is a well-chosen polynomial such that $f_j \geq 0$ and $f_j (\lambda) \geq 1$ on $J^\eta_{E_j}$.

We start by choosing $f_j$ and give some of its properties. In the sequel, we assume $\eta\leq 1/4$, which is no restriction since in the end we are interested in the limit $\eta\rightarrow 0$. We set  
\begin{equation}\label{eq:defeps}
\epsilon = \frac{1}{   \ln \left(\frac{1}{\eta}\right)}
\end{equation}
(we omit the dependence of $\epsilon$ in $\eta$ in the notation)
and we pick $f_j = 4 s_{z_j,n}$, with $s_{z_j,n}$ as in Lemma~\ref{l:approxpolpos} and $n = n_\eta$ large enough so that $\eta \epsilon \geq c/ n $ where $c>0$ is the numerical constant of Lemma~\ref{l:approxpolpos}. Since $\| p\|_1 = 1$, the spectrum of $P_N$ is in $[-1,1]$. For our choice of $f_j$, $\epsilon$, and for $\lambda \in J_{E_j}^\eta$ we have $f_j(\lambda) \geq 4 \eta^2 / ((\lambda-E_j)^2 + \eta^2) - 4 e^{-1/\epsilon} \geq 2 - 1 \geq 1$. In the same vein, for any $\lambda \in [-1,1]$, $f_j(\lambda) \leq 4   +4  e^{-1/\epsilon}\leq 5$. In particular,
\begin{equation}\label{eq:fj0}
\| f_j (P_N) \|_{\rm op} \leq 5.
\end{equation}

Also, if $P = \lambda(p)$, $R^z = (P - z \Id)^{-1}$ is its resolvent and $\mu_p$ its spectral measure defined by \eqref{eq:defmup}, from the spectral theorem
\begin{equation*}
\| f_j(p) \|_2^2 = \sum_g |f_j(p)_g|^2 = \langle \delta_e , f_j(P)^2 \delta_e \rangle = \int f_j(\lambda)^2 d \mu_p (\lambda).
\end{equation*}
Thus, from the triangle inequality, for some $C >0$, arguing as in \eqref{eq:wardineq},
\begin{align}
\| f_j(p) \|_2 &  \leq 4 \sqrt{\int \left( \frac{ \eta^2 }{(\lambda-E_j)^2 + \eta^2} \right)^2 d\mu_p(\lambda)} + 4 e^{-1/\epsilon} \nonumber \\
& \leq 4 \sqrt{\int   \frac{ \eta^2 }{(\lambda-E_j)^2 + \eta^2} d\mu_p(\lambda)} +   4\eta \nonumber\\
& \leq C \sqrt{\eta}, \label{eq:fj1}
\end{align}
where we have used  assumption \eqref{eq:AC} and the spectral theorem to get  
$$
\int  \frac{ \eta^2 }{(\lambda-E_j)^2 + \eta^2} d\mu_p(\lambda) = \eta \Im  R^{z_j} (e,e)  \leq C_0 \eta .
$$

We note also that, for some new $C >0$, for any $g \in \Gamma$, 
\begin{equation}\label{eq:fj2}
| f_j(p)_g | = | f_j(P) (g,e) | \leq 4 \eta | (\Im R^{z_j}) (g,e) | + 4 e^{-1/\epsilon} \leq C \eta, 
\end{equation}
where we have used that since $\Im R^z$ is non-negative and \eqref{eq:AC} holds,
$$|\Im (R^{z_j}) (g,e)| \leq \sqrt{\Im R^{z_j} (e,e) \Im R^{z_j}(g,g)} =  \Im R^{z_j} (e,e) \leq C_0.$$

We may now evaluate \eqref{eq:trKPKP} for our choice of $f_j$. We denote by $S$ a finite symmetric generating set of $\Gamma$ containing the support of $p$  and $t$, where we recall $t\in \Gamma$ is as in \eqref{e:kNKN}. For integer $r \geq 1$, $B_S(r)$ denotes the ball of radius $r$ in $\mathrm{Cay}(\Gamma,S)$ and $\Bad_{S}(r)  = \FIX^*_N(B_S(2r))$. From  Corollary \ref{cor:traceBS}, we have
\begin{align*}
& \left|
 \Tr ( K_N f_1(P_N) K_N^* f_2(P_N) )  -   \langle k_{N}, Q_N k_{N} \rangle \right| \\
&\quad \leq   \| k_N\|^2_{\infty}  |\Bad_S(n+1)| \left( \|f_1(P)\|_{\rm op} \|f_2(P)\|_{\rm op} + \|f_1(p) \|_2 \| f_2(p) \|_2  \right),
\end{align*}
where $Q_N  = \rho_N(q) \in M_N(\mathbb C)$ is self-adjoint with $q = \sum_g q_{ g} g \in \mathbb C[\Gamma]$,
\begin{equation}\label{eq:defqg}
    q_{ g} = \bar f_{1}(p)_{t^{-1} g t}  f_2(p)_g.
\end{equation} 
Since $f_i \geq 0$, we have $q = a^*a$ for some $a \in \mathbb C[\Gamma]$ as a consequence of Schur product theorem in $\mathbb C[\Gamma]$ %\footnote{See this \href{https://mathoverflow.net/questions/278042/a-schur-like-product-theorem-on-groups}{post} on \texttt{mathoverflow}, 278042.}
and thus $Q_N = \rho_N(q) = \rho_N(a)^*\rho_N(a)$ is non-negative.
From \eqref{eq:fj0}-\eqref{eq:fj1} and $\|k_N\|_\infty  \leq 1$, we deduce our main upper bound: for some new $C >0$, %\textbf{\textcolor{cyan}{I think that for everything we do afterwards, it is enough to consider $|\langle k_N, Q_N k_N\rangle|$, so if we give mathoverflow as a reference without really explaining, we should probably add in the footnote ``(it is natural to compare the positive trace with a positive scalar product, but this extra info is not important for what follows)'' or something similar}}
\begin{equation}\label{eq:TrC}
\frac{1}{N\eta} \Tr (K_N f_1(P_N)  K_N^* f_2 (P_N)) \leq   \frac{\langle k_{N}, Q_N k_{N} \rangle}{N\eta} + C \frac{|\Bad_S(n+1)|}{N\eta}.
\end{equation}

%\textbf{\textcolor{cyan}{Another way to do it without using positivity of the element: the modulus of the error is bounded, so the modulus of the real part of the error is bounded, so the (real) trace is bounded by the real part of the inner product, which is bounded by its positive part, so we only need to control the positive part of the real part of $\frac{\langle k_{N}, Q_N k_{N} \rangle}{N\eta}$. And we can mention in footnote that the quantity is in fact positive by mathoverflow.}}

Since $\rho_N$ converges in distribution toward $\lambda$, we have 
$$
\lim_{N\to \infty} \frac{|\Bad_S(n+1)|}{N\eta} = 0.
$$
Since $\Bad_S(n+1)$ does not depend on $E_1,E_2$, this means that for fixed $\eta>0$, the second term in the right-hand side of \eqref{eq:TrC} converges to $0$ as $N\rightarrow +\infty$, uniformly over $E_1, E_2\in I$.
Our goal will then be to prove that, uniformly over $E_1,E_2 \in I$,
\begin{equation}\label{eq:TrCC}
\lim_{\eta \to 0} \limsup_{N \to \infty} \frac{\langle k_{N}, Q_N k_{N} \rangle}{N\eta}  = 0.
\end{equation}

Depending on the various assumptions stated in Theorems \ref{t:asympdecorr}, \ref{t:HNperp} and \ref{t:QEstrongCV}, we will prove  upper bounds on
$
 \frac{\langle k_{N}, Q_N k_{N} \rangle}{N\eta}
$
which will imply \eqref{eq:TrCC}.
\subsection{Proof of Theorem~\ref{t:asympdecorr}} 

We now prove Theorem~\ref{t:asympdecorr}. Recalling $Q_N = \rho_N(q)$ with $q \in \mathbb C[\Gamma]$, we write 
\begin{align*}
\frac{\langle k_{N}, Q_N k_{N} \rangle}{N\eta} & = \frac{1}{N \eta} \sum_{x , y \in \INT{N}} \sum_{g \in \Gamma} \bar k_N(y) q_g \IND (y = g.x) k_N(x)  \\
& = \frac{1}{N \eta} \sum_{x \in \INT{N}} \sum_{g \in \Gamma} \bar k_N(g.x) k_N(x)  q_g   \\
& = \sum_{g \in \Gamma} \sigma_N(g) \frac{q_g}{\eta},
\end{align*}
where $\sigma_N( g) = \bar \sigma_N(g^{-1})$ is the empirical covariance:
$$
\sigma_N(g)=\frac1N \sum_{x\in \INT{N}} k_N(x) \bar k_N(g.x).
$$
Consequently, 
\begin{equation}\label{eq:hdeihde}
\limsup_{N\to \infty} \frac{\langle k_{N}, Q_N k_{N} \rangle}{N\eta} \leq \sum_{g \in \Gamma} \sigma(g) \frac{|q_g|}{\eta},
\end{equation}
where $\sigma(g) = \limsup_N |\sigma_N(g)|$ is defined in \eqref{e:epsgconv}. Recalling the definition of $q_g$ in \eqref{eq:defqg}, we find
$$
\sum_g |q_g| = \sum_g |f_1(p)_{t^{-1} g t} f_2 (p)_g | = \sum_g |f_1(p)_{g t} f_2 (p)_{t g} | \leq  \| f_1(p) \|_2 \| f_2(p)\|_2, 
$$
where we have used the Cauchy-Schwarz inequality. From \eqref{eq:fj1}, we get
$$
\sum_{g \in \Gamma}   \frac{|q_g|}{\eta} \leq C.
$$

Fix $\delta >0$, our assumption implies that there exists an integer $r \geq 0$ such that $\sigma(g) \leq \delta$ for all $g \notin B_S(r)$. Since $\sigma(g) \leq 1$, we deduce that 
$$
\sum_{g \in \Gamma} \sigma(g) \frac{|q_g|}{\eta} \leq \sum_{g \in B_S(r)} \frac{|q_g|}{\eta} + \delta \sum_{g \in \Gamma} \frac{|q_g|}{\eta}  \leq C^2_1 |B_S(r)| \eta + \delta C,
$$
where we have used the bound $|q_g| = | f_1(p)_{t^{-1} g t}|\, | f_2(p)_{g} |\leq C_1^2 \eta^2$ from \eqref{eq:fj2}. Note that this upper bound is uniform over all $E_1,E_2 \in I$. We deduce that 
$$
\limsup_{\eta \to 0} \limsup_{N \to \infty} \frac{\langle k_{N}, Q_N k_{N} \rangle}{N\eta}  \leq \delta.
$$
Since $\delta$ can be arbitrarily small, this proves \eqref{eq:TrCC} and concludes the proof of Theorem~\ref{t:asympdecorr}.\qed

\subsection{Norm estimation with RD property} \label{s:normestrdprop}

Let $q = \sum_g q_g g \in \mathbb{C}[\Gamma]$ with $q_g$ defined in \eqref{eq:defqg}, and recall that it depends implicitly on $E_1,E_2 \in I$, $\eta >0$ and $t \in T$. In this subsection, we let $Q = \lambda(q)$ and give an upper bound \eqref{eq:QRD} on $\|Q \|_{\rm op}$ by using the RD property (which was introduced in Definition \ref{d:RD}). More precisely, we use a well-known corollary of the RD property, Lemma~\ref{l:rdopnorm} in the appendix. In this section, we therefore focus on $Q$, and implications for $Q_N=\rho_N(q)$ will appear in Sections \ref{s:usingstrong} and \ref{s:proofoftht:HNperp}.

Let $\gamma^z_{g}= 4 \eta (\Im R^{z})(g,e)$ and $a_i = f_i(P) - 4 \eta \Im R^{z_i}$. From the triangle inequality
\begin{align*}
    \| Q \|_{\rm op} & =   \bigg\| \sum_{g}\bar f_1(p)_{t^{-1} gt} f_2(p)_g \lambda(g) \bigg\|_{\rm op} \\
    &  \leq \bigg\| \sum_{g}\bar \gamma^{z_1}_{t^{-1} gt} \gamma^{z_2}_g \lambda(g) \bigg\|_{\rm op} +  \bigg\| \sum_{g}\bar \gamma^{z_1}_{t^{-1} gt} a_{2,g} \lambda(g) \bigg\|_{\rm op}  +  \bigg\| \sum_{g}\bar a_{1,t^{-1} gt} f_2(p)_g \lambda(g) \bigg\|_{\rm op} .
\end{align*}
We next use $\|\sum_{g\in\Gamma} a_g\lambda(g) \|_{\rm op} \leq \| a \|_1$ for $a \in \mathbb{C}[\Gamma]$  and $\| a_i \|_2 = \| a_i \delta_e \| \leq \| a_i \|_{\rm op} \leq 4e^{-1/\epsilon}$. This gives, for some $C >0$, 
\begin{align*}
    \| Q \|_{\rm op}     &  \leq \bigg\| \sum_{g}\bar \gamma^{z_1}_{t^{-1} gt} \gamma^{z_2}_g \lambda(g) \bigg\|_{\rm op} +   \sum_{g}  | \gamma^{z_1}_{1,gt} a_{2,tg} |   +    \sum_{g}  | a_{1, gt} f_2(p)_{tg} | \\ 
    & \leq \bigg\| \sum_{g}\bar \gamma^{z_1}_{t^{-1} gt} \gamma^{z_2}_g \lambda(g) \bigg\|_{\rm op} + 4 e^{-1/\epsilon}  (\| \gamma^{z_1} \|_2 + \|f_2(p) \|_2 ) \\
    & \leq \bigg\| \sum_{g}\bar \gamma^{z_1}_{t^{-1} gt} \gamma^{z_2}_g \lambda(g) \bigg\|_{\rm op}  + C \eta^{3/2},
\end{align*}
where we have used Cauchy-Schwarz at the second line and then \eqref{eq:fj1} and \eqref{eq:wardineq} which implies that $\| \gamma^{z} \|_2 \leq 4 C_0 \sqrt{\eta}$.

We can now use the RD property to evaluate the norm of  $\sum_{g}\bar \gamma^{z_1}_{t^{-1} gt} \gamma^{z_2}_g \lambda(g)$. If $C'_1 > 2C_1 + 1$, we get using Lemma~\ref{l:rdopnorm},
$$
\bigg\| \sum_{g}\bar \gamma^{z_1}_{t^{-1} gt} \gamma^{z_2}_g \lambda(g) \bigg\|_{\rm op} \leq C \sqrt{ \sum_{g} (|g| +1 )^{C'_1} | \gamma^{z_1}_{t^{-1} gt} |^2 |\gamma^{z_2}_g|^2  }.
$$
Since $|g| - |t| \leq |t^{-1} g| \leq |g| + |t|$, we find some new $C > 0$, %\textbf{\textcolor{cyan}{can't we just replace the sum over $g$ by the sum over $tg$, $t$ being fixed ? and then back to a sum over $g$ ? shouldn't it be $tg$ instead of $gt$ below ? I'm not sure I got the use of the previous inequality.}}
\begin{align*}
\bigg\| \sum_{g}\bar \gamma^{z_1}_{t^{-1} gt} \gamma^{z_2}_g \lambda(g) \bigg\|_{\rm op} & \leq C \sqrt{ \sum_{g} (|gt| +1 )^{C'_1} | \gamma^{z_1}_{gt} |^2 |\gamma^{z_2}_{tg}|^2  } \\
 & \leq C \left( \sum_{g} (|g| +1 )^{C'_1} | \gamma^{z_1}_{g} |^4\right)^{1/4}  \left(\sum_{g} (|g| +1 )^{C'_1} | \gamma^{z_2}_{g} |^4\right)^{1/4} .
\end{align*}
So finally, since $\gamma^z_g= 4 \eta (\Im R^z)(g,e) $, we find, for some new $C >0$,
\[
\| Q \|_{\rm op} \leq C \left(\sum_{g} (|g| +1 )^{C'_1} \eta^4 | (\Im R^{z_1})(e,g) |^4\right)^{1/4}  \left(\sum_{g} (|g| +1 )^{C'_1} \eta^4 | (\Im R^{z_2})(e,g) |^4\right)^{1/4} + C \eta^{3/2}. 
\]
The above inequality is our main estimate on the norm of $Q$.
In particular if \eqref{eq:AC}-\eqref{e:RDrenforceeweak} holds, we deduce that for any $t \in T$,
\begin{equation}\label{eq:QRD}
\lim_{\eta \to 0} \sup_{E_1,E_2 \in I} \frac{\| Q \|_{\rm op}}{\eta} = 0,
\end{equation}
(recall again that $Q = \lambda(q)$ depends on $(\eta,E_1,E_2,t)$).

\subsection{Proof of Theorem~\ref{t:QEstrongCV}}\label{s:usingstrong}

In this subsection, we prove Theorem~\ref{t:QEstrongCV}, i.e. quantum mixing under strong convergence. We build upon Sections \ref{s:quamixestsch} and \ref{s:normestrdprop}. Recall that $Q_N =\rho_N(q)$ and $q$ depends on $(\eta,E_1,E_2)$. It is enough to prove that  \eqref{eq:TrCC} holds for all $E_1,E_2 \in I$. We consider a subset  $\mathcal E = \mathcal E_\eta$ of $ C / \eta$ points in $I$ such that all $E \in I$ are within distance $\eta$ to $\mathcal E$. Then, by the discussion above \eqref{eq:TrCC}, it suffices to prove that 
\begin{equation}\label{eq:TrCC2}
\lim_{\eta \to 0}  \limsup_{N \to \infty} \max_{E_1,E_2 \in \mathcal E_\eta} \frac{\langle k_N , Q_N k_N \rangle}{N\eta} = 0. 
\end{equation}

To this end, since $k_N \in 1^\perp$, we have
$$
 \langle k_N , Q_N k_N \rangle  \leq \| k_N  \|^2_2 \| (Q_N)_{1^\perp} \|_{\rm op}.
$$
Using that $\|k_N\|_\infty \leq 1$, we get $\| k_N\|_2 \leq \sqrt{N}$ and thus
$$
 \frac{\langle k_N , Q_N k_N \rangle }{N\eta} \leq \frac{ \| (Q_N)_{1^\perp} \|_{\rm op}}{\eta}.
$$
Since $Q_N = \rho_N( q)$, the strong convergence \eqref{e:strconvinitial}  implies that for all $N$ large enough, 
\begin{equation}\label{eq:stcvQ}
\| (Q_N)_{1^\perp} \|_{\rm op} \leq 2 \| Q \|_{\rm op}. 
\end{equation}
We may choose a large $N_0 = N_0(\eta)$ such that \eqref{eq:stcvQ} holds for all $N \geq N_0$ and all $E_1, E_2 \in \mathcal E_\eta$. From \eqref{eq:QRD}, we deduce that \eqref{eq:TrCC2} holds. It concludes the proof. \qed

\subsection{Proof of Theorem~\ref{t:HNperp}} \label{s:proofoftht:HNperp}

In this subsection we prove Theorem~\ref{t:HNperp}. We build upon Sections \ref{s:quamixestsch} and \ref{s:normestrdprop}. As we will see, the assumption that $\rho_N$ converges in distribution toward $\lambda$ implies that for any $a \in \mathbb{C}[\Gamma]$ with $a = a^*$,  there are $o(N)$ eigenvalues of $\rho_N(a)$ which are far from the spectrum of $\lambda(a)$.

Since $\rho_N$ converges in distribution to $\lambda$, we can find thanks to a diagonal extraction process a sequence $r_N \to \infty$ such that $|\Bad_S(r_N)| / N \to 0$, where as usual, for integer $r \geq 0$, $\Bad_{S}(r)  = \FIX^*_N(B_S(2r))$. We may further find integer sequences $n_N \to \infty$, $\ell_N \to \infty$ such that 
\begin{equation*}
\lim_{N \to \infty} \frac{|\Bad_S(n_N \ell_N)|}{ N} =0
\end{equation*}
(for instance $n_N=\ell_N=\lfloor \sqrt{r_N}\rfloor $). Let 
\begin{equation}\label{e:afNbas}
f_N=\frac{N}{|\Bad_S(n_N \ell_N)|+N4^{-\ell_N}}.
\end{equation}
Then $f_N\rightarrow +\infty$.

Fix $m\geq 2$, and set $\eta_m=2^{-m}$. As in \eqref{eq:defeps}, let $\epsilon_m=1/\log(1/\eta_m)$ and $N_m$ large enough so that $\eta_m\epsilon_m\geq c/n_N$ for any $N\geq N_m$, where $c>0$ is the numerical constant of Lemma \ref{l:approxpolpos}. As usual, let $q = q_{t,E_1,E_2,\eta_m} \in \mathbb C[\Gamma]$ as defined in \eqref{eq:defqg}, $Q_N = \rho_N(q)$ and $Q = \lambda(q)$ (they depend implicitly on $(t,E_1,E_2,\eta_m)$). Let $ \mathcal E_{\eta_m} \subset I$ be the subset of $C/\eta_m$ points in $I$ defined above \eqref{eq:TrCC2}. To prove Theorem~\ref{t:HNperp}, it is enough to prove that there exists a vector subspace $H_N \subset \mathbb C^N$  of dimension $\mathrm{dim}(H_N)  = o(N)$ (note that $H_N$ depends on $N$ but not on $m$) such that for any $k_N \in H_N^\perp$ with $\| k_N \|_\infty \leq 1$, we have
\begin{equation}\label{eq:dkode}
\lim_{m\rightarrow+\infty}\lim_{N \to \infty} \max_{E_1,E_2 \in \mathcal E_{\eta_m}} \frac{\langle k_N , Q_N k_N \rangle }{N \eta_m} = 0.
\end{equation}
Here we note that the definition \eqref{eq:QMKN} of small scale quantum mixing is written with the limit $\lim_{\eta\rightarrow 0}$, but it is actually equivalent to consider the limit along the subsequence $\eta_m\rightarrow 0$. Indeed, the quantity $\sum_{\alpha \in \Lambda_{J_{E_1}^\eta}} \sum_{\beta  \in \Lambda_{J_{E_2}^\eta}} \left|\langle \varphi_\beta , K_N \varphi_\alpha \rangle  - \langle \varphi_\beta,  \langle K_N \rangle  \varphi_\alpha \rangle \right|^2$ is non-decreasing in  $\eta$, and the quantity $\frac{1}{|\Lambda_{J_{E_1}^\eta}|}$ is comparable to $N\eta$ owing to Lemma~\ref{cor:CMS2}, hence it varies at most by a constant factor on dyadic intervals $[\eta_m, \eta_{m-1}]$. Thus we do not lose generality by focusing on the case where $\eta$ belongs to the sequence $(\eta_m)$.

As in the proof of Lemma~\ref{cor:CMS2}, let $\mu_{Q_N}$ be the empirical distribution of eigenvalues of $Q_N$ and let $\mu^{\delta_x}_{Q_N}$ be the spectral measure of $Q_N$ at vector $\delta_x$. By construction $q $ has support inside $B_S(n_N)$. In particular, for any $x \notin \Bad_S(n_N\ell_N)$, 
$$
\int \lambda^{2\ell_N} d \mu^{\delta_x}_{Q_N} (\lambda)= \langle \delta_x , Q^{2\ell_N}_N \delta_x \rangle = \langle \delta_e , Q^{2\ell_N} \delta_e \rangle \leq \| Q \|_{\rm op}^{2\ell_N},
$$
From Markov inequality, we find that 
\begin{align*}
N \mu_{Q_N} ([2\|Q\|_{\rm op}, +\infty)) & = \sum_{x=1}^N \mu_{Q_N}^{\delta_x} ([2\|Q\|_{\rm op}, +\infty)) \\
& \leq |\Bad_S(n_N \ell_N)| + \sum_{x \notin \Bad_S(n_N \ell_N)} \frac{\int \lambda^{2\ell_N} d \mu^{\delta_x}_{Q_N}(\lambda)}{ (2 \|Q \|_{\rm op}) ^{2\ell_N}}\\
 & \leq |\Bad_S(n_N \ell_N)| + N 4^{-\ell_N}.
\end{align*}
In particular, for each $E_1,E_2 \in I$, the vector space $H_{N,E_1,E_2,m}$ spanned by eigenvectors of $Q_{N} =\rho_N(q_{t,E_1,E_2,\eta_m})$ associated with eigenvalues larger than $2 \|Q\|_{\rm op}$ has dimension at most $|\Bad_S(n_N \ell_N)| + N 4^{-\ell_N}$. Consider the sum of vector spaces $H_N^{(m)} = \sum_{E_1,E_2 \in \mathcal E_{\eta_m}} H_{N,E_1,E_2,\eta_m}$. Since $|\mathcal E_{\eta_m}| \leq C/\eta_m$, 
\begin{equation}\label{e:dimHNm}
\dim(H_N^{(m)})\leq |\mathcal E_{\eta_m}|^2 \left( |\Bad_S(n_N \ell_N)| + N 4^{-\ell_N} \right) \leq  C^2 4^m \frac{N}{f_N}.
\end{equation}
For $N\in\N$, we finally form the sum of vector spaces
$$
H_N=\sum_{1\leq m\leq \alpha_N} H_N^{(m)}
$$
where $\alpha_N\geq 0$ is the largest integer such that 
\begin{equation}\label{e:domHoNm}
\sum_{1\leq m\leq \alpha_N} \dim(H_N^{(m)})\leq \frac{N}{\sqrt{f_N}}. 
\end{equation}
If $\alpha_N=0$, then we let $H_N=\emptyset$. By definition, $\dim(H_N)\leq \frac{N}{\sqrt{f_N}}=o(N)$. Moreover, $\alpha_N\rightarrow +\infty$ (due to \eqref{e:dimHNm} and the fact that $f_N\rightarrow +\infty$), and therefore for any $m\in\N$, there exists $N_m'\geq N_m$ such that $H_N^{(m)}$ appears in $H_N$ for any $N\geq N_m$.

Finally, for $m\in\N$, for any $E_1,E_2 \in \mathcal E_{\eta_m}$, if $N\geq N_m'$ and $k_N \in H_N^\perp$, then $k_N\in (H_N^{(m)})^\perp$, and therefore 
$$
\langle k_N , Q_N k_N \rangle \leq 2 \| Q\|_{\rm op} \| k_N \|_2^2. 
$$
Using \eqref{eq:QRD}, this concludes the proofs of \eqref{eq:dkode} and of Theorem~\ref{t:HNperp}.

%Next, we may find a sequence $\eta_N \to 0$ such that \begin{equation}\label{eq:HN1}n_N \geq \frac{c}{\eta_N \epsilon_N} \; , \quad \lim_{N \to \infty}      \frac{|\Bad_S(n_N \ell_N)|}{N \eta_N^2 }  = 0 \quad \hbox{ and }\quad \lim_{N \to \infty} \eta_N^{2} 4^{\ell_N} = +\infty.\end{equation}with $\epsilon_N = 1/ ( \ln (1/\eta_N))$ as in \eqref{eq:defeps} and $c >0$ as in Lemma~\ref{l:approxpolpos}. 

\subsection{Proof of Lemma~\ref{p:asympuncorrproba1}}

We prove here that random observables are uncorrelated. Since $\mathbb{E}a_N(x)$ does not depend on $x$, by linearity we may assume without loss of generality that 
$ \mathbb{E} a_N(x) = 0$. Also, Hoeffding's 
inequality (see \cite[Theorem 2.8]{zbMATH06138349}) and Borel-Cantelli lemma imply that almost-surely 
\begin{equation}\label{eq:BCAH}
\lim_{N \to \infty} \langle a_N \rangle = \lim_{N \to \infty}  \frac 1 N \sum_{x=1}^N a_N(x)= 0.
\end{equation}

Hence, it is sufficient to prove that for $g\neq e$, almost surely,
$$
\lim_{N \to \infty} \frac{1}{N} \sum_{x = 1}^N a_N(x) \bar a_N(g.x) = 0
$$
Let $C_N = \INT{N} \backslash \mathrm{Fix}_N(g)$ be the set of $ 1 \leq x \leq N$ such that $g.x \ne x$.  Due to the Benjamini-Schramm convergence assumption, we have $|C_N| / N \to 1$. Hence, since $|a_N(x)| \leq 1$, it is sufficient to prove that, almost surely,
$$
\lim_{N \to \infty} \frac{1}{N} \sum_{x \in C_N } a_N(x) \bar a_N(g.x) = 0.
$$
Note that by construction, $a_N(x)$ and $a_N(g.x)$ are independent for all $x \in C_N$. Therefore, 
$$
\mathbb{E} \frac{1}{N} \sum_{x \in C_N } a_N(x) \bar a_N(g.x) = 0.
$$
Observe also that $\frac{1}{N} \sum_{x \in C_N } a_N(x) \bar a_N(g.x) = F (a_N)$ is a function of $N$ independent variables $(a_N(x))_{1 \leq x \leq N}$. Explicitly,  for $a \in B^N$, $B$ being the unit complex ball,
$$
F(a) = \frac 1 N \sum_{x \in C_N} a_x \bar a_{g.x}.
$$
Now, each coordinate $x$ appears in at most $2$ terms in the above summand. Hence, for any $a,b \in B^N$,
$$
| F(a) - F(b) | \leq \frac 1  N \sum_{x = 1}^N  4  \cdot \mathbf{1} ( a_x \neq b_x ).
$$
%\textbf{\textcolor{cyan}{In fact the difference is zero if $a_x=b_x$ and $a_{gx}=b_{gx}$, so at most $\#\{x:a_x\neq b_x\} + \#\{x:a_{gx}\neq b_{gx}\}$ terms contribute, and in this case the contribution is bounded by $2$. (Mostafa: shouldn't it be 4 above ?)}}
From the bounded difference inequality (see \cite[Theorem 6.2]{zbMATH06138349}), we deduce that 
$$
\mathbb{P} \bigg( \Big|\frac{1}{N} \sum_{x \in C_N } a_N(x) \bar a_N(g.x) \Big| \geq t  \bigg) \leq 2 e^{- \frac{N t^2}{8}} .
$$
%\textbf{\textcolor{cyan}{I think $|F(a_i,a_i') - F(b_i,a_i')| \le 2/N$, so according to the reference, we should have $N^2$ in the exponent right ? Actually this inequality seems a lot more obvious than the one on $F(a)-F(b)$ which is not really needed.}}
The conclusion follows from Borel-Cantelli lemma. \qed

\begin{remark}
It is also true that given a (fixed, deterministic) sequence of observables $a_N\in \mathbb{C}^N$ indexed by $N\in\N$, with $\|a_N\|_\infty \leq 1$, the sequence $(a_N^\tau)_{N\in\N}$ is almost surely asymptotically uncorrelated, where $\tau \in S_N$ is a uniformly random permutation for each $N\in\N$ and $a_N^\tau=a_N\circ \tau$. The proof is very similar to the above one, except that Maurey's concentration inequality on the symmetric group \cite{zbMATH03618783} replaces the bounded difference inequality.
\end{remark}

\begin{remark}\label{r:notalluncorrel} If $\Gamma$ has elements of infinite order, it is also possible to prove that for any sequence $\rho_N$ converging in distribution to $\lambda$, there exists a bounded sequence of observables $(a_N(x))_{x \in \INT{N}} \in \mathbb C^N$ which is \emph{not} asymptotically uncorrelated. For this, one fixes $g\in \Gamma$ of infinite order, and considers the decomposition of the permutation $\rho_N(g)$ into cycles. On a cycle of length $r$, we define $a_N$ with value $+1$ at $\lfloor r/2\rfloor$ consecutive sites, and value $-1$ at the other sites. Since $g^\ell \ne e$ for all $\ell \geq 1$, the convergence in distribution implies that most $x \in \INT{N}$ belong to a cycle of $\rho_N(g)$ of diverging length. It is then not difficult to see that $\sigma(g^\ell)$ converges to $1$, not to $0$, for any $\ell \geq 1$.
Interestingly, Theorem~\ref{t:QEstrongCV} says that quantum mixing holds for these observables as well, if one has a strong spectral control on $P_N$ and $P$.
\end{remark}

\section{Extension to finite coverings}
\label{s:matricial}

\subsection{Setup}
In this section, as announced in Section~\ref{subsec:matricial}, we state an extension of our main theorems to the case of finite coverings. This is an important extension in view of applications, notably to coverings of graphs and product of graphs, see Examples \ref{ex:cartesian}-\ref{ex:covering} below.

For integer $r \geq 1$, we denote by $M_r(\mathbb C)[\Gamma]$ the formal finite sums $p = \sum_g p_g g$ with $p_g \in M_r (\mathbb C)$. The support of $p$  is the subset of $\Gamma$ such that $p_g$ is not the null matrix. We then define $\lambda(p)$ as the operator acting on $\mathbb C^r \otimes \ell^2 (\Gamma)$ defined by 
\begin{equation}\label{e:matrframe}
\lambda(p) = \sum_g p_g \otimes \lambda(g). 
\end{equation}
A vector $f \in \mathbb C^r \otimes \ell^2 (\Gamma)$ is decomposed as $f = (f_g)_{g \in \Gamma}$ with $f_g \in \mathbb C^r$. We then have the equivalent reformulation of \eqref{e:matrframe}: for any $f \in \mathbb C^r \otimes \ell^2 (\Gamma)$:
$$
(\lambda(p) f ) _g = \sum_s p_s f_{s^{-1} g} \in \mathbb C^r. 
$$
We also define the corresponding operator through the representation $\rho_N$:
\begin{equation}\label{e:ANmatcase}
\rho_N(p)=\sum_{g\in \Gamma} p_g\otimes  \rho_N(g). 
\end{equation}
The operator $\rho_N(p) \in M_r(\mathbb C) \otimes M_N(\mathbb C) \simeq  M_{r N} (\mathbb C)$ acts on $\mathbb C^r \otimes \mathbb C^N$.  For $p \in M_r(\mathbb C)[\Gamma]$, we set $p^* = \sum_g p_g^* g$ where $p_g^*$ denotes the Hermitian conjugate of $p_g\in M_r(\mathbb C)$. If $p = p^* $ then both $\lambda(p)$ and $\rho_N(p)$ are self-adjoint.

The operators \eqref{e:matrframe} and \eqref{e:ANmatcase} have a probabilistic interpretation if all coefficients $p_g(i,j)$ are non-negative and $\sum_{g\in\Gamma} p_g$ is a stochastic matrix. Indeed then $\lambda(p)$ is the transition matrix of a random walk on $\INT{r} \times \Gamma$ where the probability to jump from $(i,g)$ to $(j,h)$ is 
$p_{gh^{-1}}(i,j)$.
We can think of an element of $\INT{r}$ as a color, of $\INT{r} \times \{ g\}$ as the fiber above $g \in \Gamma$ and of $\{i\} \times \Gamma$ as the copy of $\Gamma$ with color $i$. The same interpretation holds for $\rho_N(p)$.

We illustrate the above framework with some examples:
\begin{example}[Cartesian product]\label{ex:cartesian}
If $H = (V,E)$ is a finite graph and $ G = {\rm Cay}(\Gamma,S)$ is a Cayley graph, then the adjacency matrix of the Cartesian product\footnote{Recall that the Cartesian product $G_1\Box G_2$ of 
two graphs $G_1=(V_1,E_1)$ and $G_2=(V_2,E_2)$ has vertex set $V_1\times V_2$, and there is an edge between $(u_1,v_1)\in V_1\times V_2$ and $(u_2,v_2)\in V_1\times V_2$ if and only if 
\begin{itemize}
    \item $(u_1,u_2)\in E_1$ and $v_1=v_2$ or
    \item $u_1=u_2$ and $(v_1,v_2)\in E_2$.
\end{itemize}} $H\Box G$ is 
\begin{equation}\label{e:cartprodasmatrix}
A_H \otimes {\rm Id}_\Gamma + {\rm Id}_V \otimes A_G = A_H \otimes \lambda(e) + \sum_{g\in S} {\rm Id}_V \otimes \lambda(g)
\end{equation}
where $A_H$ denotes the adjacency matrix of $H$ and $A_G = \sum_{g \in S} \lambda(g)$ is the adjacency operator of $G$. This is of the form \eqref{e:matrframe} with $r = |V|$. Replacing $\lambda$ by $\rho_N$ in \eqref{e:cartprodasmatrix}, we get an expression of the form \eqref{e:ANmatcase} which is the adjacency matrix of the Cartesian product $H\Box \Sch(\Gamma,S,\rho_N)$. Similarly, the adjacency matrix of the tensor product of $H$ and ${\rm Cay}(\Gamma,S)$ (resp. $\Sch(\Gamma,S,\rho_N)$) can be put in the form \eqref{e:matrframe} (resp. \eqref{e:ANmatcase}). 
\end{example}

\begin{example}[Graph covering]\label{ex:covering}
The adjacency matrix of an $N$-covering of a base graph $H = (V,E)$ with vertex set $ V = [r]$ can also be put in the form \eqref{e:ANmatcase}. If $d  = |E|$ denotes the number of undirected edges of $H$, let $\Gamma = F_d$ be the free group with $d$ free generators and their $d$ inverses denoted by $(g_1,\ldots,g_{2d})$, $g_{i^*} = g_i^{-1}$ for some involution $i \mapsto i^*$ on $\INT{2d}$. Each generator $g_i$ is associated to a directed edge $g_i=(u_i,v_i)$ of $H$, and $g_i^{-1}  =g_{i^*}=(v_i,u_i)$. We consider the action of $\Gamma$ on $[N]$ defined by $\rho_N(g_i)(x)=\sigma_i(x)$ where $(\sigma_1,\ldots,\sigma_d)$ are any permutation matrices such that $\sigma_i^{-1}=\sigma_{i^*}$. Then an $N$-covering is a graph whose adjacency matrix can be written as $\rho_N(1_E) = \sum_i E_{v_iu_i}\otimes \rho_N(g_i)$ where $E_{v_iu_i}\in M_r(\R)$ is the matrix whose only non-zero element is $(u_i,v_i)$ (equal to $1$) and $1_E = \sum_i E_{v_iu_i} g_i \in M_r(\mathbb C) [F_d]$. The vertex $(u,x) \in \INT{r} \times \INT{N}$ is connected to vertex $(v,y)$ through the $i$-th directed edge if $(u_i,v_i) = (u,v)$ and $y = \sigma_i(x)$. The operator $\lambda(1_E) = \sum_i E_{v_iu_i}\otimes \lambda(g_i)$ is the adjacency operator of $r$ disjoint copies of the universal covering tree of $H$. Note that $H$ can be a multigraph. By replacing the free group $F_{d}$ by $\mathbb Z^d$, we obtain maximal Abelian coverings (see \cite{bordenave2025sparsegraphsbenjaminischrammlimits} for the more general notion of $(\Gamma,S)$-covering graphs). 
\end{example}

\subsection{Main assumptions}

 We consider a sequence of permutation representations $\rho_N : \Gamma \to S_N$ and consider $p = p^* \in M_r(\mathbb C)[\Gamma]$. There are natural analogs of our main results, Theorems \ref{t:asympdecorr}-\ref{t:HNperp}-\ref{t:QEstrongCV} to this matricial framework. We essentially simply need to modify our definitions accordingly and all arguments go through: we have to replace scalars by matrices in $M_r(\mathbb C)$. In this paragraph, we extend the assumptions of our main theorems to this matrix setting.
 
We start with the absolutely continuous spectrum  assumptions \eqref{eq:AC} and \eqref{e:RDrenforceeweak}. Let $P = \lambda(p)$ and, for $z \in \mathbb H$,  $R^z = (P - z \Id)^{-1}$ be the resolvent of $P$. We assume that in a compact interval $I_0\subset \mathbb R$, for some $C_0 \geq 1$, 
\begin{equation}\label{eq:ACmat}
C_0^{-1} \leq \Tr \left\{ (\Im  R^z ) (e,e) \right\}\leq C_0 , \hbox{ for all $z = E + i \eta$, $E \in I_0$, $\eta > 0$,}
\end{equation}
where $e$ is the unit of $\Gamma$ and for $g,h \in \Gamma$, $R^z(g,h) = \Pi_g R^z\Pi^*_h \in M_r(\mathbb C)$, $\Pi_h$ being the orthogonal projection $\Pi_h f= f_h$ with $f = (f_g)_{g \in \Gamma} \in \mathbb C^r \otimes \ell^2 (\Gamma)$. 
The trace is taken in $M_r(\mathbb{C})$.

When $\Gamma$ has the RD property with constant $C_1$, the analog of assumption \eqref{e:RDrenforceeweak} is that for some $C'_1 > 2C_1 + 1$,
\begin{equation}\label{e:RDrenforceeweakmat}
\lim_{\eta \to 0} \sup_{\lambda_0 \in I} \sum_{g\in \Gamma}\eta^2 \|(\Im R^{\lambda_0+i\eta})(e,g)\|_2^4 |g|^{C'_1} = 0,
\end{equation} 
where the Frobenius (or Euclidean) norm $\| \cdot \|_2$ on $M_r(\mathbb C)$ is 
 $$
 \| A \|_2 = \sqrt{ \Tr ( AA^*) } = \sqrt{ \sum_{i,j} |A(i,j)|^2}.
 $$
 Note that $\| A \|_{\rm op} \leq \| A\|_2 \leq \sqrt{r} \| A \|_{\rm op}$.

We next define local operators.  If $K_N \in M_r(\mathbb C) \otimes M_N(\mathbb C)$ and $x,y \in \INT{N}$, we write $K_N(x,y) = \Pi_xK_N\Pi_y^\ast \in M_r(\mathbb C)$  for the corresponding block matrix. Let $T \subset \Gamma$ be a finite subset. An operator $K_N \in M_r (\mathbb C) \otimes M_N(\mathbb C)$ is $T$-local if for all $x,y \in \INT{N}$, 
 \begin{equation}\label{eq:defKNmat}
 K_N(x,y) = 0 \hbox{ unless $x = g.y$, for some $g \in T$},
 \end{equation}
and $\|K_{N}\|_{1,\infty}  = \max_{x,y} \| K_{N}(x,y) \|_2 \leq 1$.  For $T  = \{ e\}$, $K_N$ is block diagonal. The average $\langle K_N \rangle$ is defined as in \eqref{eq:KNav}: 
\begin{equation}\label{eq:KNavmat}
\langle K_N \rangle  = \rho_N (k_N),  \quad k_N = \sum_g k_{N,g}  g \in  M_r(\mathbb C) [\Gamma] \quad \hbox{ and } \quad k_{N,g}= \frac 1 N \sum_{x  = 1}^N  K_N(g.x,x).  
\end{equation}
Our definitions of quantum ergodicity \eqref{eq:QEKN} and (small scale) quantum (weak) mixing \eqref{eq:QMKN} are then considered with this definition of the average $\langle K_N \rangle$. If $K_N$ is a diagonal matrix, its average  $\langle K_N \rangle$ is a multiple of the identity $\Id_r \otimes \Id_N$.

In the same vein, the asymptotic covariance of a sequence of vectors $a_N \in \mathbb C^N$ introduced in Definition \ref{def:corr} extends to a vector $a_N  = (a_N(x)) \in  \mathbb C ^r \otimes \mathbb C^N$ with $a_N (x) \in \mathbb C^r$, by setting 
\begin{equation}\label{eq:defsigmamat}
\sigma(g) = \limsup_{N\rightarrow +\infty}\Bigl\| \frac1N \sum_{x\in \INT{N}} \overline{(a_N(x) - \langle a_N\rangle)} \otimes (a_N(g.x) - \langle a_N\rangle) \Bigr\|_{\rm op},    
\end{equation}
where as usual $\| \cdot \|_{\rm op}$ is the operator norm (here on $\mathbb{C}^r$).

If $\lim_{|g| \to \infty} \sigma(g) = 0$ we say that $(a_N)$ is asymptotically uncorrelated. Similarly, for a finite $T \subset \Gamma$ and a sequence of $T$-local matrices $K_N \in M_r(\mathbb C) \otimes M_N(\mathbb C)$, we say that $(K_N)$ is asymptotically uncorrelated if for any $t \in T$, $a_N (x) = K_N (t.x,x)$ is asymptotically uncorrelated.

\subsection{Results on quantum mixing}

We may now state our main results in this matrix setup all at once. Below $1 \in \mathbb C^N$ is the vector with constant coordinates equal to $1$.

\begin{theorem}\label{t:QEmat}
Assume that $\rho_N$ converges in distribution toward $\lambda$ as defined in \eqref{e:convN}. Let $ p \in M_r(\mathbb C)[\Gamma]$ such that $p = p^*$ and an interval $I_0$ such that \eqref{eq:ACmat} holds for $P = \lambda(p)$. Let $I$ be a closed interval in the interior of $I_0$ and  $(\varphi_\alpha)$ be an orthonormal eigenbasis $P_N = \rho_N(p)$. Then, for any finite $T \subset \Gamma$ and sequence of $T$-local matrices $(K_N)$ in $M_r (\mathbb C) \otimes M_N(\mathbb C)$ such that $\|K_N\|_{1,\infty} \leq 1$, quantum ergodicity \eqref{eq:QEKN} and quantum mixing \eqref{eq:QMKN} hold in $I$ provided that one of the following set of conditions holds: 

\begin{enumerate}[(i)]
    \item $(K_N)$ is asymptotically uncorrelated,
    \item $\Gamma$ has property RD with constant $C_1$, \eqref{e:RDrenforceeweakmat} holds for some $C'_1 > 2 C_1 +1$ and for all $t \in T$, $(K_N(t.x,x))_{x \in \INT{N}} \in H_N^\perp$ where $H_N \subset M_r(\mathbb C) \otimes \mathbb C^N$ is a vector subspace with $\mathrm{dim}(H_N) = o(N) $ depending on $T,p$ and $\rho_N$.
    \item $\Gamma$ has property RD with constant $C_1$, \eqref{e:RDrenforceeweakmat} holds for some $C'_1 > 2 C_1 +1$, $\rho_N$ converges strongly in distribution toward $\lambda$ and for all $t \in T$, $(K_N (t.x,x))_{x \in \INT{N}} \in M_r(\mathbb C) \otimes 1^\perp$.
\end{enumerate}
\end{theorem}

A few comments are in order. The claim with condition (i) is the matrix analog of Theorem~\ref{t:asympdecorr}, with condition (ii) of Theorem~\ref{t:HNperp} and with condition (iii) of Theorem~\ref{t:QEstrongCV}. The rapid decay property and the strong convergence were defined for $p \in \mathbb C [\Gamma]$. It turns out that their analog statements for $p \in M_r(\mathbb C)[\Gamma]$ are automatically satisfied. Namely, if $\Gamma$ has property RD with constant $C_1$ then for some $C >0$ and any $p \in M_r(\mathbb C)[\Gamma]$ we have 
$$
\| \lambda(p) \|_{\rm op} \leq C r  \mathrm{diam}_S(p)^{C_1} \| p \|_2 \quad \hbox{ with }  \quad \|p \|^2_2 = \sum_g \|p_g\|_2^2.
$$
Indeed, we have $\lambda(p) = \sum_{i,j} E_{ij} \otimes \lambda(p_{ij})$ where for $i,j \in \INT{r}$, $p_{ij}  = \sum_g p_g(i,j) g \in \mathbb C [\Gamma]$. Thus, we get
\begin{align*}
\| \lambda(p) \|_{\rm op} & \leq  \sum_{i,j} \| \lambda(p_{ij}) \|_{\rm op} \\
& \leq  C  \mathrm{diam}_S(p)^{C_1} \sum_{i,j}   \sqrt{ \sum_g |p_g(i,j)|^2 } \\
& \leq  C  r \mathrm{diam}_S(p)^{C_1} \sqrt{ \sum_{i,j} \sum_g |p_g(i,j)|^2 },
\end{align*}
the first line is subadditivity of norms, the second line is property RD for $\Gamma$ and the last line is Cauchy-Schwarz. Similarly, if $\rho_N$ converges strongly toward $\lambda$ as in Definition~\ref{d:strongconv}, then for any $p \in M_r (\mathbb C)[\Gamma]$, we have 
\begin{equation}\label{eq:CFmat}
\lim_{N \to \infty} \| \rho_N (p)_{| \mathbb C^r \otimes 1^\perp} \|_{\rm op} = \| \lambda(p)\|_{\rm op}, 
\end{equation}
see for example \cite[Lemma 2.16]{van2025strong} for a proof of this classical fact.

In Condition (iii) in Theorem~\ref{t:QEmat}, the assumption that for all $t \in T$ $(K_N(t.x,x))_{x \in \INT{N}} \in M_r(\mathbb C) \otimes 1^\perp$ is important: we show in Section~\ref{s:cartesiancycle} below, that without this assumption, quantum ergodicity is no longer guaranteed to hold. Assume for simplicity that $T = \{e \}$ and $K_N(x,x) = \mathrm{diag} ( a_N(x,1), \ldots, a_N(x,r))$, $a_N \in \mathbb C^r \otimes \mathbb C^N$. If $a_N\notin \mathbb C^r \otimes 1^\perp$, one can nevertheless apply Theorem~\ref{t:QEmat}(iii) to $\Pi a_N$, the orthogonal projection of $a_N$ onto $ \mathbb C^r \otimes 1^\perp$. Since 
$$
(\Pi a_N)(\cdot,i) = a_N(\cdot,i) - \langle a_N(\cdot,i)\rangle, 
$$
where $a_N(\cdot,i)\in \mathbb{C}^N$ with $\langle a_N(\cdot,i)\rangle = \frac{1}{N}\sum_{x\in [N]} a_N(x,i)$, we get that
\begin{equation}\label{e:correctiontermmatrix}
\lim_{N\rightarrow +\infty}\frac{1}{|\Lambda_I|}\sum_{\alpha\in\Lambda_I}\Bigl|\langle \varphi_\alpha,a_N\varphi_\alpha\rangle-\sum_{i\in[r]} \langle a_N(\cdot,i)\rangle \sum_{x\in [N]} |\varphi_\alpha(x,i)|^2\Bigr|^2 =0,
\end{equation}
as the quantity in the square modulus is precisely $\langle \varphi_\alpha, (\Pi a_N)\varphi_\alpha\rangle$. 
We show in Remark \ref{r:depeigvect} thanks to an example first considered by McKenzie in \cite{zbMATH07514683} that the quantity $\|\varphi_\alpha(\cdot,i)\|_{2}^2$ subtracted in \eqref{e:correctiontermmatrix} depends on the eigenvector $\varphi_\alpha$ and is thus not universal.

\subsection{Proof of Theorem~\ref{t:QEmat}}

The proof of Theorem~\ref{t:QEmat} is an immediate generalization of the proofs given in Section~\ref{s:pfthasymp} once we have introduced the proper setting.  We review here these proofs and explain the minor modifications needed to prove Theorem~\ref{t:QEmat}. We fix $p = p^* \in M_r(\mathbb C)[\Gamma]$ and set $P = \lambda(p)$, $P_N = \rho_N(p)$. We may assume without loss of generality that $\|p\|_1 := \sum_g \|p_g\|_2 = 1$. We fix a closed interval $I$ in the interior of $I_0$ as in \eqref{eq:ACmat}. 
As in Section~\ref{s:quamixestsch}, we may further assume that for some $t\in \Gamma$ and $k_N  = (k_N(x))_{x \in \INT{N}} \in M_r(\mathbb C) \otimes \mathbb{C}^N$ with $\| k_N \|_\infty := \sup_{x \in \INT{N}} \| k_N(x)\|_2 \leq 1$,
\begin{equation*}\label{e:kNKNmat}
K_N (x,y)  = k_N(x) \IND( x = t.y) \quad \hbox{ with } \quad \langle k_N \rangle = \frac{1}{N}\sum_{x=1}^N k_N(x) = 0.
\end{equation*}

\paragraph{Spectral identities. } Let us start with the consequences of the spectral theorem in our matrix setting. If $\mu_P^{\psi}$ is the spectral measure of $P$ at vector $\psi \in \mathbb C^r \otimes \ell^2 (\Gamma)$, we denote by $\mu_p$ the probability measure 
$$
\mu_p = \frac{1}{r} \sum_{i=1}^r \mu^{\delta_i \otimes \delta_{e}}_{P}.
$$
We deduce that for any continuous function $f : \mathbb R \to \mathbb R$,
\begin{equation}\label{eq:STfmat}
\int f(\lambda) d \mu_p(\lambda) = \frac 1 {r} \Tr \left\{ f(P)(e,e)  \right\}.
\end{equation}
Also, setting $\| q \|_2 := \sqrt{ \sum_g \|q_g \|^2_2}$ for $q \in M_r(\mathbb C)[\Gamma]$, we find (for $f$ polynomial),
\begin{equation}\label{eq:STfmat2}
\| f (p) \|_2^2 = \sum_g \| f(p)_g \|_2^2 = \Tr \sum_g f(p)_g f(p)_g^* = \Tr \left\{ f(P)^2(e,e) \right\} = r \int f(\lambda)^2 d \mu_p(\lambda).
\end{equation}

Notably, applying \eqref{eq:STfmat} to $f^2 = \eta / ( \lambda-E)^2+ \eta^2) = \Im ( 1/ (\lambda -z))$, $z = E + i \eta$, we deduce that if \eqref{eq:ACmat} holds then $\mu_p$ admits a bounded density on $I$ which is lower bounded away from $0$. This allows to extend the proof of Lemma~\ref{cor:CMS2} to our $p = p^* \in M_r(\mathbb C)[\Gamma]$.

\paragraph{Quantum mixing bound. } Let $E_1, E_2 \in I$ and $J^\eta_{E_j} = [E_j - \eta, E_j + \eta]$, $j = 1,2$. From what precedes, as argued in Section~\ref{s:quamixestsch}, it then suffices to establish a vanishing upper bound for 
\begin{equation}\label{eq:trKPKP2}
\frac{1}{N\eta} \Tr (K_N f_1(P_N)  K_N^* f_2 (P_N)),
\end{equation}
where $f_j = f_{j,\eta}$ is a well-chosen polynomial such that $f_j \geq 0$ and $f_j (\lambda) \geq 1$ on $J^\eta_{E_j}$. Exactly, as  in Section~\ref{s:quamixestsch}, we choose $f_j = 4 s_{z_j,n}$ with $z_j = E_j + i \eta$ and $\eta \epsilon \geq c / n$ with $s_{z,n}$ and $c$ as in Lemma~\ref{l:approxpolpos} and $\epsilon$ given by \eqref{eq:defeps}. In all estimates, benign polynomial factors in $r$ will appear in the various constants $C> 0$ of  Section~\ref{s:pfthasymp}.

Using \eqref{eq:STfmat2}, the argument leading to \eqref{eq:fj1} gives that for some $C > 0$,
\begin{equation}\label{eq:fj1mat}
\| f_j (p) \|_2 \leq C \sqrt {\eta}.
\end{equation}
Similarly, the argument leading to \eqref{eq:fj2} gives for some $C > 0$, for all $g \in \Gamma$,
$$
\| f_j(p)_g \|_2 \leq C \eta.
$$
Indeed, since $\Im R^z$ is non-negative, for any $x = (i,g)$ and $y = (j,h)$ in $\INT{r} \times \Gamma$, we have
$$
| \Im (R^{z_j}) (g,h) (i,j)| = | \langle \delta_x , \Im R^{z_j} \delta_y \rangle | \leq \sqrt{ \langle \delta_x ,\Im R^{z_j} \delta_x \rangle \langle \delta_y ,\Im R^{z_j} \delta_y \rangle } \leq \max_{ i \in \INT{r}} \Im R^{z_j} (e,e) (i,i) \leq C_0,
$$
by assumption \eqref{eq:ACmat}.

To evaluate \eqref{eq:trKPKP2}, we expand the trace over indices in $\INT{r}$. More precisely, we may decompose an element $A \in M_r(\mathbb C) \otimes M_N(\mathbb C)$ as a block matrix of size $r \times r$ with blocks of size $N \times N$:  
$$
A = \sum_{1 \leq i,j \leq r} E_{ij} \otimes A_{ij},
$$
with $A_{ij} \in M_N(\mathbb C)$. We get
\begin{equation*}
\frac{1}{N\eta}  \Tr (K_N f_1(P_N)  K_N^* f_2 (P_N))  =   \sum_{i \in \INT{r}^4}\frac{1}{N\eta} \Tr \left\{ K_{N,i_1 i_2} f_1(P_N)_{i_2i_3} K^*_{N,i_3 i_4}f_2(P_N)_{i_4i_1} \right\}.
\end{equation*}
The traces are in $M_r(\mathbb{C})\otimes M_N(\mathbb{C})$ and $M_N(\mathbb{C})$, respectively.
It suffices to prove that (the non-negative real part of) each of the $r^4$ terms on the right-hand side are vanishing in the regime $N \to \infty$ and then $\eta \to 0$. To this end, we observe that for any $ a = \sum_g a_g g \in M_r(\mathbb C)[\Gamma]$ and $(i,j) \in \INT{r}^2$, we have $\rho_N(a)_{ij} = \rho_N (a_{ij}) \in M_N(\mathbb C)$, where $a_{ij} = \sum_g a_{g} (i,j) g \in \mathbb C [\Gamma]$. In particular, for each $ i  \in \INT{r}^4$, we can apply Lemma~\ref{le:traceBS} and get
\begin{align*}
& \left|
 \Tr \left\{ K_{N,i_1 i_2} f_1(P_N)_{i_2i_3} K^*_{N,i_3 i_4}f_2(P_N)_{i_4i_1} \right\} -   \langle k_{N,i_4i_3}, Q_{N,i} k_{N,i_1i_2} \rangle \right| \\
&\quad \leq   \| k_N\|^2_{\infty}  |\Bad_S(n+1)| \left( \|f_1(P)\|_{\rm op} \|f_2(P)\|_{\rm op} + \|f_1(p) \|_2 \| f_2(p) \|_2  \right),
\end{align*}
where $k_{N,ij} = (k_{N} (x) (i,j))_{x \in \INT{N}} \in \mathbb C^N$ and $Q_{N,i}  = \rho_N(q_i) \in M_N(\mathbb C)$ with $q_i = \sum_g q_{i, g} g \in \mathbb C[\Gamma]$, 
\begin{equation*}\label{eq:defqgmat}
    q_{i, g} = \bar f_{1}(p)_{t^{-1} g t} (i_3,i_2)  f_2(p)_g (i_4,i_1).
\end{equation*} 
The only noticeable difference with \eqref{eq:defqg} is that $q_i$ is not necessarily self-adjoint. Next, from \eqref{eq:fj0}, \eqref{eq:fj1mat} and $\|k_N\|_\infty  \leq 1$, we deduce for some $C >0$,
\begin{equation*}\label{eq:TrCmat}
\frac{1}{N\eta} \Tr (K_N f_1(P_N)  K_N^* f_2 (P_N)) \leq  \sum_{i \in \INT{r}^4} \Re \left( \frac{\langle k_{N,i_4i_3}, Q_{N,i} k_{N,i_1i_2} \rangle}{N\eta} \right)  + C \frac{|\Bad_S(n+1)|}{N\eta}.
\end{equation*}

At this stage, the remainder of the proof of Theorem  \ref{t:QEmat} is essentially the same as the proofs given in  Section~\ref{s:pfthasymp}. The three conditions given in Theorem  \ref{t:QEmat} are tailored to fit in the proofs given  in  Section~\ref{s:pfthasymp}. Let us review briefly the three cases.

\paragraph{Condition (i): } For each $i \in \INT{r}^4$, we follow the proof of Theorem~\ref{t:asympdecorr} and get 
\begin{align*}
\frac{\langle k_{N,i_4i_3}, Q_{N,i} k_{N,i_1i_2} \rangle}{N\eta} = \sum_{g \in \Gamma} \sigma_{N,i}(g) \frac{q_{i,g}}{\eta},
\end{align*}
where 
$$
\sigma_{N,i} (g)=\frac1N \sum_{x\in \INT{N}} k_N(x)(i_1,i_2) \bar k_N(g.x)(i_4,i_3).
$$
Consequently, 
\[
\limsup_{N\to \infty}  \left| \frac{\langle k_{N,i_4i_3}, Q_{N,i} k_{N,i_1i_2} \rangle}{N\eta} \right| \leq \sum_{g \in \Gamma} \sigma_i(g) \frac{|q_{i,g}|}{\eta},
\]
where $\sigma_i(g) = \limsup_N |\sigma_{N,i} (g)| \leq \sigma(g)$ and $\sigma(g)$ is defined in \eqref{eq:defsigmamat} for $a_N  = k_N$. The remainder of the proof is identical to the proof of Theorem~\ref{t:asympdecorr}.

\paragraph{Condition (iii): } By construction the vector space $1^\perp$ is invariant for $Q_{N,i}$ and $Q_{N,i}^*$. Also, the assumption $(k_N(x))_{x \in \INT{N}} \in M_r(\mathbb C) \otimes 1^\perp$ implies that for any $(i,j)$, $k_{N,ij} \in 1^\perp$. Therefore, since $\| k_N \|_{\infty} \leq 1$: 
$$
\left| \frac{\langle k_{N,i_4i_3}, Q_{N,i} k_{N,i_1i_2} \rangle}{N\eta} \right|  \leq \frac{\| (Q_{N,i})_{1^\perp} \|_{\rm op} }{\eta}.  
$$
The remainder of the proof is identical to the proof of Theorem~\ref{t:asympdecorr}. The assumption \eqref{e:RDrenforceeweakmat} ensures that $\| \lambda(q_i) \|_{\rm op}$ goes to $0$ as $\eta \to 0$ for any $i \in \INT{r}^4$ thanks to the  computation in Section~\ref{s:normestrdprop}.

\paragraph{Condition (ii): } We pick $\eta_m, \epsilon_m$, $n_N,  \ell_N, f_N$ and $\mathcal E_{\eta_m}$ as in the proof of Theorem~\ref{t:HNperp}.  As explained above \eqref{eq:dkode}, it is enough to prove that for each $i \in \INT{r}^4$, there exists a vector subspace $H_{N,i} \subset \mathbb C^N$  of dimension $\mathrm{dim}(H_{N,i})  = o(N)$ such that for any $f  \in H_{N,i}^\perp$ with $\| f \|_2 \leq 1$, we have 
\begin{equation}\label{eq:dkodemat}
\lim_{m\rightarrow +\infty}\lim_{N \to \infty} \max_{E_1,E_2 \in \mathcal E_{\eta_m}} \frac{\|Q_{N,i} f \|_2 }{N\eta_m} = 0,
\end{equation}
where we recall that $Q_{N,i}$ depends implicitly on $(t,E_1,E_2,\eta_m)$.
 Indeed, if \eqref{eq:dkodemat} holds true, we set $\bar H_N =  \sum_{i\in[r]^4} H_{N,i} \subset \mathbb C^N$ and define $H_N = \{ M \in M_r(\mathbb C) \otimes M_N(\mathbb C): \hbox{for all } (i_1,i_2) \in \INT{r}^2: M_{i_1i_2} \in \bar H_{N} \}$ and Theorem~\ref{t:QEmat}(ii) follows for this choice of $H_N$.

To establish \eqref{eq:dkodemat}, the difference with the proof of Theorem~\ref{t:HNperp} is that $q_{i}$ is not self-adjoint. To overcome this, we let $\mu_{|Q_{N,i}|}$ be the empirical distribution of the eigenvalues of $|Q_{N,i}| = \sqrt{ Q_{N,i} ^* Q_{N,i}}$ (that is, the empirical distribution of the singular values of $Q_{N,i}$). Let $\mu_{|Q_{N,i}|}^\psi$ be the spectral measure of $|Q_{N,i}|$ at vector $\psi \in \mathbb C^N$ and let $Q_i = \lambda(q_i)$. If $x \notin \Bad_S (n_N\ell_N)$, we note that
$$
\int \lambda^{2\ell_N} d \mu^{\delta_x}_{|Q_{N,i}|} (\lambda)= \langle  \delta_x , (Q^*_{N,i} Q_{N,i})^{\ell_N} \delta_x \rangle = \langle  \delta_e , (Q_i^* Q_i)^{\ell_N}  \delta_e \rangle  \leq \| Q_i \|_{\rm op}^{2\ell_N}.
$$
Therefore, as in the proof of Theorem~\ref{t:HNperp}, Markov inequality implies
\begin{align*}
N \mu_{|Q_{N,i}|} ([2\|Q_i\|_{\rm op}, +\infty))  & \leq |\Bad_S(n_N \ell_N)| + N 2^{-2\ell_N}.
\end{align*}
In particular, for each $E_1,E_2 \in I$, the vector space $H_{N,i,E_1,E_2}^{(m)}$ spanned by eigenvectors of $|Q_{N,i}| = | \rho_N(q_{i,t,E_1,E_2,\eta_m})|$ associated with an eigenvalue larger that $2 \|Q_i\|_{\rm op}$ has dimension at most $|\Bad_S(n_N \ell_N)| + N 2^{-2\ell_N}$. Let $H_{N,i}^{(m)} = \sum_{E_1,E_2 \in \mathcal E_{\eta_m}} H_{N,i,E_1,E_2}^{(m)}$. This vector space has dimension at most 
$C^2 4^m\frac{N}{f_N}$ by our choice of parameters as explained in the proof of Theorem~\ref{t:HNperp}. For $N\in\N$, we finally form the sum of vector spaces $H_{N,i}=\sum_{1\leq m\leq \alpha_{N,i}}H_{N,i}^{(m)}$ where $\alpha_{N,i}$ is the largest integer such that $\sum_{1\leq m\leq \alpha_{N,i}}\dim(H_{N,i}^{(m)})\leq N/\sqrt{f_N}$. Finally, if $f \in H_{N,i}^\perp$ then by construction, for all $E_1,E_2 \in \mathcal E_{\eta_m}$, we have 
$$
\|Q_{N,i} f \|_2 \leq 2 \| Q_i\|_{\rm op} \|f \|_2.
$$
Claim \eqref{eq:dkodemat} follows as in the proof of Theorem~\ref{t:HNperp}. \qed

\section{Applications of the main results} 
\label{s:applications}

In this section, we present various examples to which our main results apply:

\begin{itemize}
\item We apply Theorems \ref{t:asympdecorr}-\ref{t:HNperp}-\ref{t:QEstrongCV} to free products of at least three non-trivial groups (or free products of two complete graphs) in Section~\ref{s:freeproducts};
\item We apply Theorems \ref{t:asympdecorr}-\ref{t:HNperp}-\ref{t:QEstrongCV}  to some class of right-angled Coxeter groups in Section~\ref{s:RACG}. This shows that our results apply to families of graphs with many short cycles which are not obtained via free products.
\item We apply Theorem~\ref{t:QEmat} to lifts of a finite base graph in Section~\ref{s:liftsofgraphs}.
\end{itemize}
Theorem~\ref{t:QEmat}(i) also allows to prove quantum mixing for uncorrelated observables in $\mathbb{Z}^d$-periodic graphs in spectral intervals where the density of states is bounded from above and below. We leave this application to interested readers. 

There are three types of assumptions to check in order to apply our results: 

\begin{itemize}
\item {\em Properties of $\Gamma$}: to apply Theorems \ref{t:HNperp}-\ref{t:QEstrongCV} and Theorem~\ref{t:QEmat}(ii)-(iii), we need the RD property. There is a rather clear understanding of which groups satisfy the RD property, see \cite{zbMATH06859874}.

\item {\em Properties of $p \in \mathbb C[\Gamma]$}: to apply any of our results, we need to check that $\lambda(p)$ has purely absolutely continuous spectrum in some interval $I_0$, \eqref{eq:AC} in the scalar case and \eqref{eq:ACmat} for finite coverings. This property turns out to be subtle to check in general and only partial results are known. To apply Theorems \ref{t:HNperp}-\ref{t:QEstrongCV} and \ref{t:QEmat}(ii)-(iii), we further need to prove our $4$-th moment bound on the resolvent $(\lambda(p) - z \mathrm{Id})^{-1}$, assumptions \eqref{e:RDrenforceeweak} and \eqref{e:RDrenforceeweakmat}. This property has not been previously considered. We provide two different strategies to deal with it: (i)  the multiplicativity of resolvent on tree-like weighted graphs, and (ii) a symmetry trick satisfied when $\lambda(p)$ is the adjacency operator of a Cayley graph with many automorphisms fixing an element. We illustrate strategy (i) for nearest-neighbor operators in free product of finite groups and for universal covering tree of a base graph. We rely on strategy (ii) to deal with some adjacency operators on right-angled Coxeter groups.

\item {\em Properties of $\rho_N \in \mathrm{Hom}(\Gamma,S_N)$}:  to apply any of our results, we need to check that $\rho_N$ converges in distribution toward $\lambda$, assumption \eqref{e:convN}, or equivalently the Benjamini-Schramm convergence of the underlying Schreier graph to its corresponding Cayley graph, see \eqref{eq:convBS}. In concrete examples where this is true, this is usually not very difficult to check. In all examples that we consider in Section~\ref{s:applications}, this assumption is known to be true for some sequence of random representation $\rho_N \in \mathrm{Hom}(\Gamma,S_N)$ along a subsequence $N \to \infty$. This shows in particular that such converging sequences of Schreier graphs exist. To apply Theorems \ref{t:QEstrongCV} and \ref{t:QEmat}(iii), we require also that $\rho_N$ converges strongly in distribution toward $\lambda$ or at least that $\rho_N$ converges strongly in distribution toward $\lambda$ when restricted to the orthogonal of an invariant vector space of small dimension (see Remark \ref{r:QEstrongCV}). Although the theory of strong convergence is undergoing a fast development, the set of groups for which it is known that strongly convergent families of permutation-valued representations exist is still limited, see \cite{arXiv:2503.21619,van2025strong}. Finally to apply Theorems \ref{t:asympdecorr} and \ref{t:QEmat}(i), we need to prove that the observable is asymptotically uncorrelated.  
\end{itemize}

To keep this section to a reasonable length, we shall focus here on the properties of $p \in \mathbb C[\Gamma]$. We will only briefly comment on the representations $\rho_N$ which are known to converge toward $\lambda$.

\subsection{Free products}\label{s:freeproducts}

Quantum ergodicity has been established in \cite{zbMATH06434640} in the case where $\Cay(\Gamma,S)$ is a regular tree (with $\Gamma$ the free group, and identical weights on the edges), and extended to the regular tree with anisotropic weights in \cite{anantharaman2017quantum}. We generalize here these results by showing that quantum ergodicity and quantum mixing hold in sequences of Schreier graphs converging to the Cayley graphs of free products of 
non-trivial groups.

\begin{proposition}\label{t:freeproductQEQM}
Let $m\geq 3$ be an integer, and $\Gamma_1,\ldots,\Gamma_m$ be non-trivial finite groups. Let $\Gamma=\Gamma_1*\cdots*\Gamma_m$ be their free product. For $S_1,\ldots,S_m$ symmetric sets of generators of $\Gamma_1,\ldots,\Gamma_m$, we consider the Cayley graph $\Cay(\Gamma,S)$ where $S=S_1\cup \ldots\cup S_m$. Let $ p \in \mathbb C[\Gamma]$ with support $S$ and such that $p = p^*$. Then there exists an interval $I_0$ with non-empty interior such that \eqref{eq:AC} and \eqref{e:RDrenforceeweak} hold for $P = \lambda(p)$.  

The same conclusion holds for $m=2$ for complete graphs, i.e. when $S_i=\Gamma_i\setminus\{e\}$ for $i=1,2$, and $|\Gamma_1|\geq 3$, $|\Gamma_2|\geq 2$ and $p=p_1\mathbf{1}_{S_1}+p_2\mathbf{1}_{S_2}$, with $p_i>0$.
\end{proposition}

Let us describe the Cayley graph $G =\Cay(\Gamma,S)$ involved in this proposition. Let $(G_i,e_i)=\Cay(\Gamma_i,S_i)$ be the Cayley graph of $\Gamma_i$, rooted at $e_i$. First draw $G_1$, then at each vertex $v$ of $G_1$, glue a copy of $G_i$ by identifying $e_i$ with $v$, for each $i>1$. As $G_i$ is regular of degree $d_i$, the new degree of $v$ will be $d=\sum_{i=1}^m d_i$. Then repeat this procedure at the other vertices, gluing copies of $(G_j,e_j)$, $j\neq i$, at each vertex of $\Gamma_i$. See Figure~\ref{f:freeproduct} and \cite{McL} for more illustrations. The special case $\Gamma = \Z_2\ast \dots\ast\Z_2$ yields the $m$-regular tree considered in \cite{anantharaman2017quantum,zbMATH06434640}. Because we have no assumption on the weights $p_i\in \mathbb{C}$ except for self-adjointness, our results include anisotropic walks on $\Gamma$. The case $\Gamma=\mathbb{Z}_2\ast\mathbb{Z}_3$ gives $PSL(2,\Z)$.

Any free product $\Gamma$ as in Proposition~\ref{t:freeproductQEQM} has the RD property since $|\Gamma_i|<+\infty$ and the RD property is preserved under free product (see \cite[Theorem A]{zbMATH04169445}). Also, we will see in Section~\ref{s:freeprodconvseq} that  there exist sequences of permutation representations $(\rho_N)$ which converge in distribution toward the left regular representation $\lambda$ of $\Gamma$.\footnote{It is very likely that this random permutation representation converges strongly in probability, but this result has not appeared yet in the literature (and we will not fill this gap here).}
Therefore we deduce the following:
\begin{theorem}
With the same notation as in Proposition~\ref{t:freeproductQEQM}, let $I$ be a closed interval in the interior of $I_0$ and $(\varphi_\alpha)$ be an orthonormal eigenbasis for  $P_N = \rho_N(p)$. Then, for any finite $T \subset \Gamma$, and any sequence of $T$-local matrices $(K_N)$ in $M_N(\mathbb C)$ with  $\|K_N\|_{1,\infty} \leq 1$, quantum ergodicity \eqref{eq:QEKN} and quantum mixing \eqref{eq:QMKN} hold in $I$ provided that one of the following set of conditions holds: 
\begin{enumerate}[(i)]
    \item $(K_N)$ is asymptotically uncorrelated,
    \item For all $t \in T$, $(K_N(t.x,x))_{x \in \INT{N}} \in H_N^\perp$ where  $H_N\subset \mathbb C^N$ is a vector subspace with $\mathrm{dim}(H_N) = o(N) $ depending on $T,p$ and $\rho_N$.
\end{enumerate}
\end{theorem}

The rest of Section~\ref{s:freeproducts}, except Section~\ref{s:freeprodconvseq}, is devoted to the proof of Proposition~\ref{t:freeproductQEQM}. 

\subsubsection{Free product of two complete graphs}\label{s:completegraphs}
The spectral analysis of free products of complete graphs is well understood; such products are technically similar to the regular tree, which they generalize, and have been considered by several authors \cite{Kuhn,zbMATH01186131}. In this subsection we prove the last part of Proposition~\ref{t:freeproductQEQM}, i.e. we focus on the case of two copies, $G=K_{a+1}\ast K_{b+1}$ where $K_r$ denotes the complete graph with $r$ vertices (Figure \ref{f:freeproduct} illustrates the case $a=b=2$). The case of more copies is a special case of our later analysis.
It is shown in \cite[Theorem 7]{zbMATH01186131} that $\sigma(P)$ consists of two bands of absolutely continuous spectrum, plus two, one or no eigenvalues, depending on the values of $a,b$. The resolvent $R^z(e,e)$ is computed rather explicitly in \cite[Proposition 1]{zbMATH01186131}, it is algebraic in $z$ and satisfies \eqref{eq:AC} in the interior of the bands, see \cite[Proposition 3]{zbMATH01186131}.

\begin{figure}[h!]
\begin{center}
\begin{tikzpicture}[
    scale=0.8,
    every node/.style={font=\small},
    dot/.style={circle,fill,inner sep=1.2pt}
]

\node[dot,label=below:$e$] (E) at (0,0) {};

\node[dot,label=above:$\beta^2$] (LU) at (-2,1) {};
\node[dot,label=below:$\beta$] (LL) at (-2,-1) {};
\node[dot,label=left:$\alpha$] (RU) at (2,1) {};
\node[dot,label=below:$\alpha^2$] (RL) at (2,-1) {};

\node[dot,label=right:$\alpha\beta^2$] (LUT) at (-3.4,2.4) {};
\node[dot,label=right:$\beta\alpha\beta^2$] (LUTT) at (-3.4,3.7) {};
\node[dot,label=left:$\beta^2\alpha\beta^2$] (LUTL) at (-4.4,3.3) {};
\draw (LUT) -- (LUTT);
\draw (LUTT) -- (LUTL);
\draw (LUTL) -- (LUT);

\node[dot,label=below:$\alpha^2\beta^2$] (LUL) at (-3.9,1.4) {};
\draw (LU) -- (LUT);
\draw (LU) -- (LUL);
\draw (LUL) -- (LUT);
\node[dot,label=left:$\beta\alpha^2\beta^2$] (LULT) at (-5.1,1.4) {};
\node[dot,label=left:$\beta^2\alpha^2\beta^2$] (LULL) at (-4.6,0.4) {};
\draw (LUL) -- (LULT);
\draw (LUL) -- (LULL);
\draw (LULT) -- (LULL);

\node[dot,label=right:$\alpha^2\beta$] (LLT) at (-3.4,-2.4) {};
\node[dot,label=right:$\beta \alpha^2 \beta$] (LLTT) at (-3.4,-3.7) {};
\node[dot,label=left:$\beta^2 \alpha^2\beta$] (LLTL) at (-4.4,-3.3) {};
\draw (LLT) -- (LLTT);
\draw (LLTT) -- (LLTL);
\draw (LLTL) -- (LLT);

\node[dot,label=above:$\alpha\beta$] (LLL) at (-3.9,-1.4) {};
\draw (LL) -- (LLT);
\draw (LL) -- (LLL);
\draw (LLL) -- (LLT);
\node[dot,label=left:$\beta \alpha\beta$] (LLLT) at (-5.1,-1.4) {};
\node[dot,label=left:$\beta^2 \alpha\beta$] (LLLL) at (-4.6,-0.4) {};
\draw (LLL) -- (LLLT);
\draw (LLLT) -- (LLLL);
\draw (LLL) -- (LLLL);

\node[dot,label=right:$\beta\alpha^2$] (RLT) at (3.4,-2.4) {};
\node[dot,label=left:$\alpha\beta \alpha^2$] (RLTT) at (3.4,-3.7) {};
\node[dot,label=right:$\alpha^2\beta \alpha^2$] (RLTL) at (4.4,-3.3) {};
\draw (RLT) -- (RLTT);
\draw (RLTT) -- (RLTL);
\draw (RLTL) -- (RLT);

\node[dot,label=above:$\beta^2\alpha^2$] (RLR) at (3.9,-1.4) {};
\node[dot,label=right:$\alpha\beta^2\alpha^2$] (RLRT) at (5.1,-1.4) {};
\node[dot,label=right:$\alpha^2\beta^2\alpha^2$] (RLRL) at (4.6,-0.4) {};
\draw (RL) -- (RLT);
\draw (RL) -- (RLR);
\draw (RLR) -- (RLT);
\draw (RLR) -- (RLRT);
\draw (RLR) -- (RLRL);
\draw (RLRL) -- (RLRT);

\node[dot,label=right:$\beta^2\alpha$] (RUT) at (3.4,2.4) {};
\node[dot,label=left:$\alpha\beta^2\alpha$] (RUTT) at (3.4,3.7) {};
\node[dot,label=right:$\alpha^2\beta^2\alpha$] (RUTL) at (4.4,3.3) {};
\draw (RUT) -- (RUTT);
\draw (RUTT) -- (RUTL);
\draw (RUTL) -- (RUT);

\node[dot,label=below:$\beta\alpha$] (RUR) at (3.9,1.4) {};
\draw (RU) -- (RUT);
\draw (RU) -- (RUR);
\draw (RUR) -- (RUT);

\node[dot,label=right:$\alpha\beta \alpha$] (RURT) at (5.1,1.4) {};
\node[dot,label=right:$\alpha^2\beta \alpha$] (RURL) at (4.6,0.4) {};
\draw (RU) -- (RUT);
\draw (RU) -- (RUR);
\draw (RUR) -- (RUT);
\draw (RUR) -- (RURT);
\draw (RUR) -- (RURL);
\draw (RURT) -- (RURL);

\draw (LU) -- (RL);
\draw (RU) -- (LL);
\draw (LU) -- (LL);
\draw (RU) -- (RL);

\end{tikzpicture}
\caption{The ball of radius $3$ (in the word length metric) around the origin in the Cayley graph associated to $K_3*K_3$} \label{f:freeproduct}
\end{center}
\end{figure}
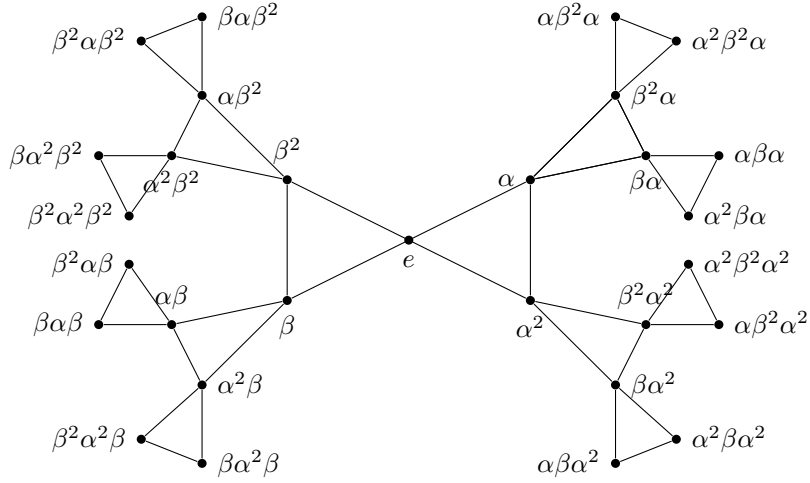

To check for \eqref{e:RDrenforceeweak}, we crucially need the assumption that $a>1$ or $b>1$. This will be sufficient, but it is also necessary, since $G=K_2\ast K_2=\mathbb{Z}$ does not satisfy \eqref{e:RDrenforceeweak}. It is known that $R^z(e,g)$ takes a product form, see \cite{Kuhn} and more generally Lemma~\ref{lem:prodzeta} below. Namely, if $g\in\Gamma$ is written canonically as $g=g_kg_{k-1}\dots g_1$ where $g_i$ and $g_{i+1}$ belong alternatively to $\Gamma_1$ and $\Gamma_2$ (by definition of the free product $\Gamma_1*\Gamma_2$) then
\begin{equation}\label{e:procom}
R^z(e,g)=R^z(e,e)\prod_{i=1}^k h_z(g_i),
\end{equation}
with $h_z(g_i)=\frac{R^z(e,g_i)}{R^z(e,e)}=:\zeta_1^z$ if $g_i\in \Gamma_1$ (this quantity is independent of $g_i\in \Gamma_1$ since $p=p_1\mathbf{1}_{S_1}+p_2\mathbf{1}_{S_2}$) and $h_z(g_i)=\zeta_2^z$ if $g_i\in \Gamma_2$. By definition of $\Gamma_1\ast\Gamma_2$, these values alternate.
That is, each element $g$ in $\Gamma_1\ast\Gamma_2$ belongs to one of the following sets: $\Gamma_1\Gamma_2\Gamma_1\ldots \Gamma_2$, or $\Gamma_1\Gamma_2\Gamma_1\ldots \Gamma_1$, or $\Gamma_2\Gamma_1\Gamma_2\ldots \Gamma_1$, or $\Gamma_2\Gamma_1\Gamma_2\ldots \Gamma_2$. We deduce by Ward's identity \eqref{e:ward} and \eqref{e:procom} that
\begin{equation}\label{e:free2}
1+ \sum_{n\in\N}\left(    |\zeta_1^z|^{2(n+1)}|\zeta_2^z|^{2n} a^{n+1} b^n + |\zeta_1^z|^{2n}|\zeta_2^z|^{2(n+1)} a^n b^{n+1} + 2|\zeta_1^z\zeta_2^z|^{2(n+1)} a^{n+1} b^{n+1}\right)  = \frac{\Im R^z(e,e)}{\eta |R^z(e,e)|^2}.
\end{equation}
Since the geometric series converge for any $\eta>0$, we get
$$
|\zeta_1^z\zeta_2^z|^2\leq \frac{1}{ab}<1
$$
for any $\eta>0$.
Now due to \eqref{e:procom} the quantity \eqref{e:RDrenforceeweak} can be written as the left-hand side in \eqref{e:free2}, except with powers $4$ instead of $2$ on the $\zeta^z_i$'s. For instance the first sum is
$$
|R^z(e,e)|^4\sum_{n\in\N} a^nb^n |\zeta_1^z|^{4n}|\zeta_2^z|^{4n} n^{C_1'}\leq C\sum_{n\in\N} a^{-n}b^{-n} n^{C_1'}<+\infty
$$
for any $\eta>0$. The other sums may be handled in the same way. Therefore \eqref{e:RDrenforceeweak} holds, which concludes the proof of the last part of Proposition~\ref{t:freeproductQEQM}.

\subsubsection{Properties of the resolvent}\label{s:freeprodpropres}

In this section, we present several properties of the resolvent which are needed to check \eqref{eq:AC} and \eqref{e:RDrenforceeweak} for more general free products, corresponding to the first part of Proposition~\ref{t:freeproductQEQM}.

Set for notational simplicity $p(x,y)=p_{yx^{-1}}$ for $x,y\in\Gamma$, and let
\begin{equation}\label{e:fn}
f^{(n)}(x,y)=\sum_{(x_0,\ldots,x_n)} p(x_0,x_1)\ldots p(x_{n-1},x_n),
\end{equation}
where the sum runs over all paths $(x_0,\ldots,x_n)$ of length $n$  such that $x_0=x$, $x_n=y$ and $x_j\neq y$ for any $0\leq j\leq n-1$. By convention $f^{(0)}(x,y)=\delta_{x=y}$. If $p_g\geq 0$ for all $g$ and $\|p\|_1=1$, $f^{(n)}(x,y)$ is the probability for a random walk starting at $x$ and moving on $\Gamma$ according to the probability $p$ to reach $y$ for the first time at time $n$. Let
$$
F^z(x,y)=\sum_{n=0}^{+\infty}f^{(n)}(x,y)z^n
$$
and
$$
R^z(x,y)=(\lambda(p)-z\Id)^{-1}(x,y), \qquad \zeta_x^z(y)=F^{1/z}(y,x).
$$
Finally, recall that $w$ is a cut point between $x$ and $y$ if every path from $x$ to $y$ passes through $w$.

\begin{lemma}
The functions $z\mapsto R^z(x,y)$ and $z\mapsto \zeta_x^z(y)$ are analytic on $\mathbb{H}=\{z\in \mathbb{C}:\Im z>0\}$ and the following relation holds for any $z\in \mathbb{H}$ :
\begin{equation}\label{e:rinv}
R^z(x,x)=\frac{1}{\sum_y p(x,y) \zeta_x^z(y) -z} \,.
\end{equation}
It holds that $\zeta_x^z(x)=1$ and
\begin{equation}\label{e:rprod}
R^z(x,y) = \zeta_y^z(x)R^z(y,y) \,.
\end{equation}
If $y\neq x$,
\begin{equation}\label{e:zedev}
\zeta_x^z(y) = \frac{1}{z} \sum_w p(y,w) \zeta_x^z(w) \,.
\end{equation}
If $w$ is a cut point between $x$ and $y$, then
\begin{equation}\label{e:zecut}
\zeta_y^z(x) = \zeta_y^z(w)\zeta_w^z(x) \,.
\end{equation}
\end{lemma}
\begin{proof}
All these relations are reformulations of corresponding ones in \cite{zbMATH05596588}. In more detail, it is shown in \cite{zbMATH05596588} that if $G^z = (G^z(x,y))_{x,y\in X}$, then $G^z = (I-zP)^{-1}$, hence $R^z(x,y) = \frac{-1}{z} G^{1/z}(x,y)$. Equations \eqref{e:rinv}, \eqref{e:rprod}, \eqref{e:zedev} and \eqref{e:zecut} thus follow from \cite[Thm. 1.38, Prp. 1.43]{zbMATH05596588} for $|z|$ large enough, if $p$ is a probability density. 
Since the proofs in \cite{zbMATH05596588}  of \eqref{e:rprod}, \eqref{e:zedev} and \eqref{e:zecut} are based on identities of generating functions of weighted walks, they extend more generally to any $p \in \mathbb{C}[\Gamma]$. Given \eqref{e:rprod}, equation \eqref{e:rinv} easily follows by combining $(PR^z\delta_x)(x)=1+zR^z(x,x)$ and $(PR^z\delta_x)(x)=\sum_y p(x,y)R^z(y,x)$ and dividing by $R^z(x,x)$; a similar argument was used in \cite[Lemma 1.1]{zbMATH04042923}.

This proves the relations for general $p$ and large $|z|$. However, $z\mapsto R^z(x,y)$ is well-defined and analytic in $\mathbb{H}=\{z:\Im z>0\}$, for any $x,y$. Moreover, $z\mapsto R^z(v,v)$ is Herglotz, i.e. $\Im R^z(v,v)>0$ if $\Im z>0$. In particular, $R^z(v,v)$ does not vanish on the upper half plane, hence we may extend $\zeta_y^z(x)$ as the analytic function $ \frac{R^z(x,y)}{R^z(y,y)}$ for all $z\in \mathbb{H}$ for any $x,y$. 
Thus, all relations above express equalities between functions outside of some large disk $\mathcal{D}$, and the functions are analytic in $\mathbb{H}$. Let us show this implies the equations all hold true in $\mathbb{H}$. 

Fix $\varepsilon>0$. We will show that all relations hold true on $\mathbb{H}_\varepsilon = \{z\in \mathbb{C} : \Im z>\varepsilon\}$. Since $\varepsilon$ is arbitrary, this will imply equality on $\mathbb{H}$.

So let $B = \{z\in \mathbb{H}_\varepsilon: f(z)\neq g(z)\}$ and suppose $B\neq \emptyset$. Let $R = \sup_{z\in B} |z|$. We know $R>0$ is finite since the functions coincide outside of $\mathcal{D}$. Let $\mathcal{C}_R$ be the semicircle in $\mathbb{H}_\varepsilon$ of radius $R$ and let $z_0\in \mathcal{C}_R$. Expand the analytic function $h(z):=f(z)-g(z)$ in a small disk $D_{z_0}$ around $z_0$. Since this function is zero for $z\in D_{z_0}$, $|z|>R$, it must be zero on all $D_{z_0}$. Now $\mathcal{C}_R$ is compact, so it can be covered by finitely many such disks $D_{z_0}$. This shows that $f(z)=g(z)$ for $z$ outside of some disk of radius strictly smaller than $R$. This contradicts the definition of $R$. Hence, $B=\emptyset$.
\end{proof}

The previous relations are true even in the presence of self-loops, meaning that $p(x,x)\neq 0$. We can rewrite \eqref{e:rinv} as 
$R^z(x,x) = 1/(p(x,x)+\sum_{y\neq x} p(x,y)\zeta_x^z(y)-z)$.

By construction of the free product, for any $x\in \Gamma\setminus\{e\}$, there is a unique $k\in\N$ and a unique sequence $(x_j)_{1\leq j\leq k}$, $x_j\in \Gamma_{i_j}\setminus \{e_{i_j}\}$, such that $i_j \neq i_{j-1}$ and $x=x_{k}\ldots x_{1}$. The group element $x$ is said to have length $k$ in the word metric, and $x=x_k\ldots x_1$ is its reduced expression.  
The following lemma is a generalization of \cite[Lemma 4]{zbMATH04020080}.

\begin{lemma}\label{lem:prodzeta}
Let $z\in\mathbb{H}$ and $x\in\Gamma$ with reduced expression $x=x_k\ldots x_1$. Let also $y=v_\ell \ldots v_1x$, where $\ell\geq 1$ and for $j\in[\ell]$,  $v_j\in\Gamma_{i_j}\setminus\{e\}$ for some $i_j\in[m]$. We also assume that $y$ is reduced, meaning that $i_j\neq i_{j-1}$ and $x_k\notin \Gamma_{i_1}$. Let $y_0=x$ and $y_j = v_j\ldots v_1x$ for $j\in[\ell]$. Then
\begin{equation}\label{e:zetabipro}
\zeta_y^z(x) = \zeta_{y_1}^z(y_0)\zeta_{y_2}^z(y_1)\cdots\zeta_{y_\ell}^z(y_{\ell-1})
\end{equation}
and $R^z(x,y) = R^z(y,y)\prod_{j=1}^\ell \zeta_{y_j}^z(y_{j-1})$.
\end{lemma}
\begin{proof}
This follows immediately from \eqref{e:zecut} and the definition of $\Gamma = \Gamma_1*\ldots*\Gamma_m$. Namely, any path from $x$ to $y$ passes through each of the points $y_j$ before reaching $y$.
\end{proof}
In the above statement, $y_j$ is not necessarily a neighbor of $y_{j-1}$. This depends whether $p(y_{j-1},y_j)\neq 0$ or not.
From now on we denote by $V_i$ the set of $x\in\Gamma$ whose reduced expression $x=x_k\ldots x_1$ satisfies $x_k\in\Gamma_i$. 
Lemma~\ref{lem:prodzeta} applies only to $x,y$ in different $V_i$. For $x,y$ in the same $\Gamma_i$, we have the following expansion instead. Here and below we see $\Gamma_i$ as a subset of $\Gamma$, and elements of $\Gamma_i$ have word length $0$ or $1$ in $\Gamma$.

\begin{lemma}
Let $z\in\mathbb{H}$ and let $x,y\in \Gamma_i$, $x\neq y$. Then 
\begin{equation}\label{e:zetaq}
z\zeta_x^z(y)= p(y,x) + p(y,y)\zeta_x^z(y)+\sum_{w\in\Gamma_i\setminus\{x,y\}} p(y,w)\zeta_x^z(w) + \sum_{j\neq i}\sum_{w\in V_j} p(y,w) \zeta_y^z(w)\zeta_x^z(y).
\end{equation}
\end{lemma}
Note that  $p(y,y)=0$ if there are no  self-loops, i.e. $p_e=0$, and that we do not assume $p(y,x)\neq 0$, i.e. $x,y$ may not be neighbors.
\begin{proof}
We have $z\zeta_x^z(y)=\sum_w p(y,w)\zeta_x^z(w)$ by \eqref{e:zedev}, so expanding,
\[
z\zeta_x^z(y) = p(y,x)+p(y,y)\zeta_x^z(y)+ \sum_{w\in\Gamma_i\setminus\{x,y\}} p(y,w) \zeta_x^z(w) + \sum_{w\notin\Gamma_i} p(y,w)\zeta_x^z(w) \,.
\]
If $w\notin \Gamma_{i}$ and $p(y,w)\neq 0$, then $w=vy$ for some $v\in \Gamma_j$ with $j\neq i$, and $y$ is a cut point between $x$ and $w$. So we get $\zeta_x^z(w)=\zeta_x^z(y)\zeta_y^z(w)$. All in all we obtain \eqref{e:zetaq}.
\end{proof}

\begin{proposition}\label{prp:alg}
The functions $z\mapsto \zeta_x^z(y)$ are algebraic for any $x,y\in\Gamma_i$ and any $i\in[m]$. Also, $z\mapsto R^z(v,w)$ is algebraic for any $v,w\in\Gamma$.
\end{proposition}

\begin{proof}
Let $n_i=|\Gamma_i|$. To each $\Gamma_i$ we associate the $\nu_i=n_i(n_i-1)$ functions $z\mapsto \zeta_x^z(y)$ for $(x,y)\in\Gamma_i\times \Gamma_i$ with $x\neq y$.\footnote{Actually we have $\zeta_x^z(y)= \zeta_e^z(yx^{-1}) $, so there are only $n_i$ distinct functions, but for our argument, there is no need to use this reduction.} Ordering them arbitrarily, we denote them by $\zeta_{i,j,k}$, where $i\in[m]$, $j\in[n_i]$ and $k\in[n_i-1]$. Equation \eqref{e:zetaq} thus gives the following finite system of quadratic equations, of which $\{h_{i,j,k}=\zeta_{i,j,k}^z\}$ (with $i\in[m]$, $j\in[n_i]$, and $k\in[n_i-1]$) is a solution: for each $i\in[m]$, $j\in[n_i]$, and $ k\in[n_i-1]$, 
\begin{equation}\label{e:syst}
q_{i,j,k}+(p_e-z)h_{i,j,k}+  \sum_{\substack{t\le n_i -1\\t\neq k}} q_{i,j,t}h_{i,j,t} + h_{i,j,k} \sum_{\substack{r\le m\\ r\neq i}} \sum_{t\le n_r-1} q_{r,k,t}h_{r,k,t} = 0,
\end{equation}
where the $q_{i,j,k}$ are complex numbers. To ensure the solution of this system is algebraic, we need to show that the equations can somehow be decoupled into separate polynomial equations, each involving only one variable $h_{i,j,k}$. For this, we compute the Jacobian determinant. If equation \eqref{e:syst} is denoted by $P_{i,j,k}(h)$, then 
\[
\frac{\partial P_{i,j,k}(h)}{\partial h_{\kappa,a,b}} = \begin{cases}p_e-z+\sum_{r\neq i}\sum_{t\le n_r-1} q_{r,k,t}h_{r,k,t}& \text{if } \kappa=i, a=j, b=k\\ q_{i,j,b}& \text{if } \kappa=i, a=j, b\neq k,\\ h_{i,j,k} q_{\kappa,a,b}& \text{if } \kappa\neq i. \end{cases}
\]
Now for large $|z|$, we have $\zeta_{i,j,k}^z = F^{1/z}_{i,j,k} = \sum_{n\ge 0} f^{(n)}(x,y)z^{-n}$, and $f^{(0)}(x,y)=0$ because $x\neq y$, so that $\zeta_{i,j,k}^z \sim z^{-1}$ for $|z|\gg 1$. Thus, due to the diagonal, we see that $\det (\frac{\partial P_{i,j,k}(h)}{\partial h_{\kappa,a,b}})|_{h=\zeta^z} = (-z)^N+O(|z|^{N-1})$ for $|z|\to\infty$, where $N=\sum_i n_i(n_i-1)$ is the size of the system, in particular, it is not identically zero as a function of $z$. It follows from \cite[Proposition VIII.5.3]{zbMATH01703931} that each $\zeta_{i,j,k}^z$ is algebraic as a function of $z$. Since algebraic functions form a field, we deduce from \eqref{e:rinv} that $R^z(v,v)$ is algebraic and the same for $R^z(v,w)$ using \eqref{e:rprod} and \eqref{e:zetabipro}. 
\end{proof}

The fact that $R^z(e,e)$ is algebraic appeared in a number of works, in various degrees of generality and justification \cite{zbMATH04042923,zbMATH01186131,zbMATH04020080}, sometimes by simply exhibiting a system of algebraic equations that $R^z(e,e)$ satisfies. This is not enough in general, for example $h_1+zh_2=0$ and $2h_1+2zh_2=0$ has a solution $h_1(z)=ze^{z}$ and $h_2(z)=-e^z$ which is not algebraic.

\begin{proposition}\label{p:zetafinitelim} 
The spectrum $\sigma(\lambda(p))$ has a non-trivial absolutely continuous part. There exists a discrete set $D_0\subset \R$ such that for any $E\in \R\setminus D_0$, $i\le m$ and $x,y\in\Gamma_i$, the function $\eta\mapsto \zeta_x^{E+i\eta}(y)$ has a finite limit as $\eta\rightarrow 0$, and the functions $E\mapsto \zeta_{x}^{E+i0}(y)$ are analytic on $\R\setminus D_0$. Furthermore, the maps $E\in \sigma(\lambda(p)) \setminus D_0 \mapsto R^{E+i0}(e,e)$ and $E \in \sigma(\lambda(p)) \setminus D_0 \mapsto \zeta_x^{E+i0}(y)$ have finitely many zeroes.
\end{proposition}
\begin{proof} 
The first statement comes from \cite[Theorem 4.1]{zbMATH05351824} and \cite[Theorem 7.4]{zbMATH01182655}.

Next, we know from Proposition~\ref{prp:alg} and the Newton-Puiseux theorem \cite[Theorem 3.5.2]{zbMATH06532717} that there exists a discrete set $D_1\subset \R$ such that for any $i\le m$ and $x,y\in\Gamma_i$, the map $z\mapsto \zeta_x^z(y)$ is analytic near any $z_0\notin D_1$ (the function has moreover an expansion as a Puiseux series near $z_0\in D_1$, and the set $D_1$ is necessarily real because we already know that $\zeta_{x}^z(y)$ is analytic on $\mathbb{C}\setminus \R$). The function $\zeta_x^z(y)$ is also nonzero on $\mathbb{C}\setminus \R$ by definition. It follows from analyticity that the zeroes of $\zeta_x^z(y)$ are in some discrete set $D_1\cup D_2$, thus inside the compact set $\sigma(\lambda(p))$ there are only finitely many of them.

We next prove the statement for $R^z$. As the function $E\mapsto \sum_{y} p(x,y)\zeta_x^{E+i0}(y)-E$ is analytic on $\R\setminus D_1$, its zeroes form a discrete set $D_3$, and its intersection with $\sigma(\lambda(p))$ which is compact is finite. By \eqref{e:rinv}, we deduce that $E\mapsto R^{E+i0}(e,e)$ is analytic on $\R\setminus (D_1\cup D_3)$.
Note that $R^{E+i0}(e,e)$ is necessarily nonzero on $\R\setminus D_1$ by \eqref{e:rinv}. 
\end{proof}

\begin{proposition}\label{prp:specfreep}
The spectrum $\sigma(\lambda(p))$ of a free product as above consists of a union of intervals $J_j$ of purely absolutely continuous spectrum, plus possibly finitely many eigenvalues, and \eqref{eq:AC} holds in any closed interval $I_0$ contained in the interior of some $J_j$.
\end{proposition}

Arguing as in \cite{zbMATH07162035} one can prove that the absolutely continuous part is in fact a \emph{finite} union of intervals, but we shall not pursue this here. 

\begin{proof}
Following the notation used in the proof of Proposition~\ref{p:zetafinitelim}, let $F=D_1\cup D_2\cup D_3$. 
Since $E\mapsto R^{E+i0}(e,e)$ is analytic on $\R\setminus F$, the set $M^{ac} =\{E\in \R\setminus F: 0< \Im R^{E+i0}(e,e)<\infty\}$ is non-empty, and open. Thus, $M^{ac}$ is a union of open intervals, and the absolutely continuous part $\overline{M^{ac}}$ of the spectrum is a union of intervals $J_j$. The eigenvalues of $\lambda(p)$ are poles of $E\mapsto \Im R^{E+i0}(e,e)$, which can only occur if $E\in F$, and since $\lambda(p)$ is bounded, $\sigma(\lambda(p))\cap F$ is a finite set. By definition and continuity of $E\mapsto \Im R^{E+i0}(e,e)$, we know that \eqref{eq:AC} holds true on each closed interval $I_0$ contained in the interior of some $J_j$. Such $I_0$ is compact as $\sigma(\lambda(p))$ is bounded. 
\end{proof}

\subsubsection{Checking the resolvent condition \eqref{e:RDrenforceeweak}} \label{s:freeprodres4}
In this section, we check that \eqref{e:RDrenforceeweak} holds. 
We first prove an elementary lemma which shows that in order to prove the $4$-th moment bound \eqref{e:RDrenforceeweak} in regions where \eqref{eq:AC} holds, it is sufficient to check the vanishing of a quantity involving a supremum, instead of a sum.
\begin{lemma}\label{l:reducresol}
If \eqref{eq:AC} holds in $I_0$ and 
\begin{equation}\label{e:convcarrepoly}
\sup_{\lambda \in I_0} \sup_{g\in\Gamma} \eta|(\Im R^{\lambda_0+i\eta})(e,g)|^2 |g|^{C_1'} \underset{\eta\rightarrow 0}{\longrightarrow} 0
\end{equation}
then the resolvent condition \eqref{e:RDrenforceeweak} holds in $I_0$. 
\end{lemma}
\begin{proof}
We have 
$$
\sum_{g\in \Gamma}\eta^2(\Im R^{\lambda_0+i\eta})(e,g)^4 |g|^{C_1'}\leq \Bigl(\sup_{g\in\Gamma} \eta|(\Im R^{\lambda_0+i\eta})(e,g)|^2 |g|^{C_1'} \Bigr) \sum_{g\in \Gamma}\eta(\Im R^{\lambda_0+i\eta})(e,g)^2.
$$
By \eqref{eq:wardineq} and \eqref{eq:AC}, the sum in the right-hand side is uniformly bounded as $\eta\rightarrow 0$, and this concludes the proof.
\end{proof}

Let us prove \eqref{e:convcarrepoly} for free products. We first show that there exists $\delta>0$ such that for any $E\in I$, any $\eta$ small enough, any distinct $i_1,i_2\in[m]$, and any $j_1,k_1\in[n_{i_1}], j_2,k_2\in[n_{i_2}]$ with $j_\ell\neq k_\ell$ for $\ell=1,2$, there holds
\begin{equation}\label{e:pp1-delta}
|\zeta_{i_1,j_1,k_1}^z \zeta_{i_2,j_2,k_2}^z| \leq 1-\delta
\end{equation}
where $z=E+i\eta$.
Using the above notation, we pick $i_0\in[m]\setminus\{i_1,i_2\}$, which exists since we assumed $m\geq 3$. Since $n_{i_0}\geq 2$ ($\Gamma_{i_0}$ is non-trivial), let $j_0,k_0\in[n_{i_0}]$ distinct. We set $\omega_0^z=\zeta_{i_0,j_0,k_0}^z \zeta_{i_2,j_2,k_2}^z$ and $\omega_1^z=\zeta_{i_1,j_1,k_1}^z \zeta_{i_2,j_2,k_2}^z$. Let $\mathcal{L}$ be the set of finite sequences valued in $\{0,1\}$. For $\underline{\ell}=(\ell_1,\ldots,\ell_r)\in\mathcal{L}$, we set
$$
\omega_{\underline{\ell}}^z=\omega_{\ell_r}^z\ldots\omega_{\ell_1}^z.
$$
Finally for $s=0,1,2$, we pick $g_s\in \Gamma_{i_s}$ such that $\zeta^z_{i_s,j_s,k_s}=\zeta^z_e(g_{s})$ (for this, we first pick $x_s,y_s$ such that $\zeta^z_{i_s,j_s,k_s}=\zeta^z_{x_s}(y_s)$ and then let $g_s=y_sx_s^{-1}$). Then $$
\omega_{\underline{\ell}}^z= \zeta_e^z(g_{\ell_r})\zeta_e^z(g_2) \zeta_e^z(g_{\ell_{r-1}})\zeta_e^z(g_2)\cdots\zeta_e^z(g_{\ell_1})\zeta_e^z(g_2).$$ Since $\ell_i\in\{0,1\}$, it follows that $g_{\ell_r}g_{2}g_{\ell_{r-1}}g_{2}\ldots g_{\ell_1}g_{2}\in\Gamma$, where $g_{2}$ appears every other time. %\textbf{\textcolor{cyan}{seems $g_2 g_{\ell_r} g_2\dots g_2g_{\ell_1}$ instead. For example, $\zeta_e(g_2g_{\ell_1})=\zeta_e(g_{\ell_1})\zeta_{g_{\ell_1}}(g_2g_{\ell_1})=\zeta_{e}(g_{\ell_1})\zeta_e(g_2)$}} 
Therefore we deduce by Lemma~\ref{lem:prodzeta} and the Ward identity \eqref{e:ward}
$$
\frac{\Im R^z(e,e)}{\eta|R^z(e,e)|^2}\geq \sum_{\underline{\ell}\in\mathcal{L}} |\omega^z_{\underline{\ell}}|^2 =\sum_{n=0}^{+\infty}\sum_{k=0}^n \binom{n}{k} |\omega_0^z|^{2k}|\omega_1^z|^{2(n-k)} = (1-|\omega_{0}^z|^2-|\omega_1^z|^2)^{-1}.
$$
We get $|\omega_0^z|^2+|\omega_1^z|^2<1$ for any $\eta$. 
Since $|\omega_0^z|$ has a strictly positive lower bound as $\eta\rightarrow 0$ according to Proposition~\ref{p:zetafinitelim} and the definition of $I_0$, \eqref{e:pp1-delta} immediately follows. 

Write any $y\in\Gamma$ as $y=g_{i_r,j_r}\ldots g_{i_1,j_1}$ where $g_{i_\ell,j_\ell}\in \Gamma_{i_\ell}$ and $i_\ell\neq i_{\ell+1}$ for any $\ell\le r-1$.  Set $L(y)=r$. We have $L(y)\geq |y|/M$ where $M=\sup_i |\Gamma_i|$ and recall that $|y|$ denotes the word metric of $y\in\Gamma$. Combining \eqref{e:pp1-delta} with Lemma~\ref{lem:prodzeta}, we obtain 
$$
|R^z(e,y)|^2\leq (1-\delta)^{L(y)}|R^z(e,e)|^2\leq (1-\delta)^{|y|/M}|R^z(e,e)|^2.
$$
Since $|y|^{C_1'}(1-\delta)^{|y|/M}$ tends to $0$ as $|y|\rightarrow +\infty$, it follows from \eqref{eq:AC} that \eqref{e:convcarrepoly} holds. Lemma~\ref{l:reducresol} then implies that \eqref{e:RDrenforceeweak} holds, which concludes the proof of Proposition~\ref{t:freeproductQEQM} in this case.

\subsubsection{Weakly converging sequences of permutation representations}\label{s:freeprodconvseq}

Our theorems require that a sequence of finite Schreier graphs converges in distribution to the Cayley graph, here on a free product $\Gamma = \Gamma_1* \cdots * \Gamma_m$. Such sequences indeed exist and are common. 
Recall that $n_j$ is the order of $\Gamma_j$ and let $q$ be the least common multiple of the $n_j$'s. Set $N  = N_n = q n$. Then, the set of permutation representations $\mathrm{Hom}(\Gamma,S_N)$ is non-empty for all $ n \geq 1$. Moreover, it follows from \cite{zbMATH07931115} that as $n \to \infty$, a uniformly random element $\rho_N \in \mathrm{Hom}(\Gamma,S_N)$ converges in distribution to the regular representation of $\Gamma$ in probability.

\subsection{Right-angled Coxeter groups} \label{s:RACG}

In this section, we check that Theorems~\ref{t:asympdecorr}-\ref{t:HNperp}-\ref{t:QEstrongCV} apply to a family of right-angled Coxeter groups (RACG). This is quite a specific class of groups, but our purpose it to show that our results apply to classes of finitely presented groups which are not free products. Until now the main tool for checking resolvent bounds needed for quantum ergodicity (see for instance \cite{zbMATH07162035}) has been to write the resolvent in a product form like \eqref{e:zetabipro}, but this is specific to tree-like structures and free-product structures. For RACG we need to develop another approach. Along the way, we provide methods to check the assumption \eqref{e:RDrenforceeweak} in more general groups. We start with general ideas to verify the resolvent condition \eqref{e:RDrenforceeweak} in Section~\ref{s:aroundres}, then we show in Section~\ref{s:resinRACG} that they apply in a specific class of RACG, and finally we show in Section~\ref{s:strongRACG} that in a subclass, there exist convergent sequences of permutation representations.  Our main result in this section is Corollary \ref{t:QEcoxeter}. A key step in the proof is Proposition~\ref{p:QEcoxeter}.

Let $\Gamma$ be a right-angled Coxeter group (RACG), i.e., a group with a presentation of the form
$$
\Gamma=\langle s_1,\ldots,s_n \mid (s_is_j)^{m_{ij}}\rangle
$$
where $m_{ij} \in \{1,2,+\infty\}$ is symmetric ($m_{ij}=m_{ji}$), $m_{ii}=1$ and $m_{ij}\geq 2$ if $i\neq j$. By convention, the condition $(s_is_j)^{+\infty}$ means  that no relation on $s_i$ and $s_j$ is imposed. Notice that $m_{ij}=2$ is equivalent to the commutation relation $s_is_j=s_js_i$. 
 We denote by 
$$
S=\{s_1,\ldots,s_n\}
$$
the set of generators. Since $m_{ii}=1$, all generators are of order $2$. It follows from \cite{doi:10.1142/S1793525313500052} that all RACG have the RD property. 

We represent the relations in a ``defining diagram": the vertices of the diagram are the elements of $S$, and there is an edge between distinct vertices $s_i$ and $s_j$ if and only if $m_{ij}=2$ (in which case we say that $s_i$ and $s_j$ are neighbors), see Figure \ref{f:defdiag}. The family of RACG is stable by direct and free products:   the defining diagram of the free product $\Gamma_1 * \Gamma_2$ is the disjoint union of the two diagrams and, for the defining diagram of the direct product $\Gamma_1 \times \Gamma_2$, we also add edges between all generators of $\Gamma_1$ and $\Gamma_2$.

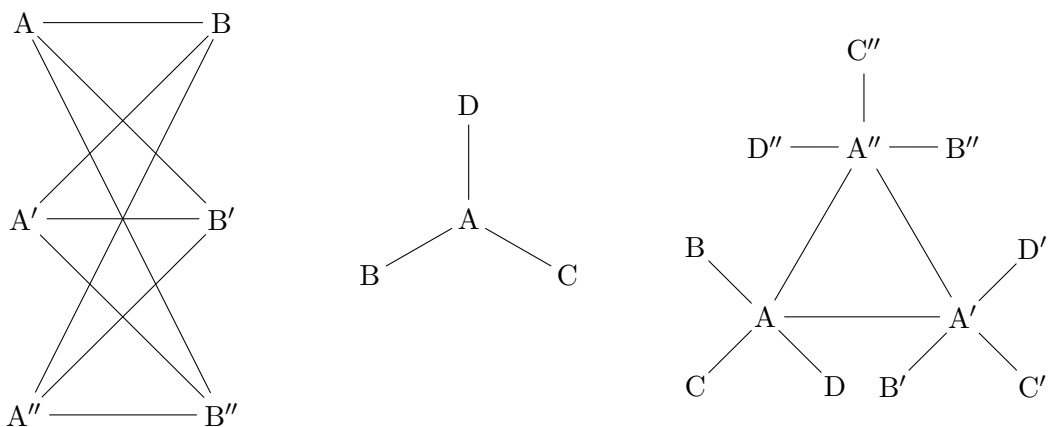
\begin{figure}[h!]
\begin{center}
\begin{tikzpicture}[scale=1.3, every node/.style={circle,inner sep=1pt}]

\node (a) at (0,2) {A};
\node (a') at (0,0) {A$'$};
\node (a'') at (0,-2) {A$''$};
\node (b) at (2,2) {B};
\node (b') at (2,0) {B$'$};
\node (b'') at (2,-2) {B$''$};

\draw (a)--(b);
\draw (a)--(b');
\draw (a)--(b'');
\draw (a')--(b);
\draw (a')--(b');
\draw (a')--(b'');
\draw (a'')--(b);
\draw (a'')--(b');
\draw (a'')--(b'');

  \coordinate (shift) at (3.5,-1.732050808/3);
    \begin{scope}[shift=(shift)]

\node (B) at (0,0) {B};
\node (C) at (2,0) {C};
\node (D) at (1,1.732050808) {D};
\node (A) at (1,1.732050808/3) {A};
\draw (A)--(B);
\draw (A)--(C);
\draw (A)--(D);

\end{scope}

  \coordinate (shift) at (7.5,-1);
    \begin{scope}[shift=(shift)]

\node (A) at (0,0) {A};
\node (A') at (2,0) {A$'$};
\node (A'') at (1,1.732050808) {A$''$};

\draw (A)--(A');
\draw (A)--(A'');
\draw (A'')--(A');

\node (B) at (-0.707106781,0.707106781) {B};
\node (C) at (-0.707106781,-0.707106781) {C};
\node (D) at (0.707106781,-0.707106781) {D};

\draw (A)--(C);
\draw (A)--(D);
\draw (A)--(B);

\node (C') at (2+0.707106781,0-0.707106781) {C$'$};
\node (D') at (2+0.707106781,0+0.707106781) {D$'$};
\node (B') at (2-0.707106781,0-0.707106781) {B$'$};

\draw (A')--(C');
\draw (A')--(D');
\draw (A')--(B');

 \node (C'') at (1,1.732050808+1) {C$''$};
 \node (D'') at (0,1.732050808) {D$''$};
 \node (B'') at (2,1.732050808) {B$''$};

\draw (A'')--(C'');
\draw (A'')--(D'');
\draw (A'')--(B'');
    
\end{scope}

\end{tikzpicture}
\caption{Super-flexible defining diagrams on $|S|=6$, $4$ and $12$ elements. The diagram on the left is a defining diagram for $\mathbb Z_2^{*3} \times \mathbb Z_2^{*3}$, the middle diagram for $\mathbb Z_2 \times \mathbb Z_2^{*3}$.} \label{f:defdiag}
\end{center}
\end{figure}

We introduce the following non-standard terminology, in opposition to the term ``rigid", which refers to structures whose automorphisms are all induced by a known, external action:

 \begin{definition}\label{d:superflex}
The defining diagram is \emph{super-flexible} if for any $s,t\in S$ which do not commute, there exists $u\in S\setminus\{s,t\}$ not commuting with $t$ nor $s$, and a graph automorphism $\phi$ of the defining diagram, which has one of the following two properties: 
\begin{enumerate}[(i)]
    \item either it fixes $s$, fixes all neighbors of $s$, and exchanges $t$ and $u$ (i.e., $\phi(t)=u$ and $\phi(u)=t$);
    \item or it fixes $t$, fixes all neighbors of $t$, and exchanges $s$ and $u$.
\end{enumerate}
We then say that $(\Gamma,S)$ (or simply $\Gamma$) is {\em superflexible}.
\end{definition}

\begin{example}
The three defining diagrams in Figure \ref{f:defdiag} are super-flexible. For instance, for the right diagram, for $s=B$ and $t\in \{X',X''\}$ with $X \in \{A,B,C,D\}$, we may choose $\phi$ which exchanges $B$ and $ u = C$, and fixes all other generators. Similarly, for $s=B$ and $t=C$, we may choose $\phi$ which exchanges $B$ and $u = D $, and fixes all other generators. 
\end{example}

The main message of this subsection is that if $\Gamma$ is superflexible, then  the purely ac condition \eqref{eq:AC} automatically implies \eqref{e:RDrenforceeweak} for the adjacency operator on $\mathrm{Cay}(\Gamma,S)$:

\begin{proposition} \label{p:QEcoxeter}
Let $\Gamma$ be a RACG whose defining diagram is superflexible and $I_0 \subset \mathbb R$ be a closed interval. If the RHS of \eqref{eq:AC} holds for $p=\mathbf{1}_{S}$ on $I_0$ then \eqref{e:RDrenforceeweak} also holds for $p=\mathbf{1}_{S}$ on $I_0$. 
\end{proposition}

As explained in \S \ref{s:aroundres}, the arguments of this subsection can also be adapted to handle other families of Cayley graphs (when $\Gamma$ is not necessarily a RACG). For instance, denote by $F_d$ the free group on $d$ generators $g_1,\ldots,g_d$, and let $j_1,\ldots,j_k$ be distinct and positive integers. Let $S$ be the set of generators given by all words of length $j_1,\ldots,j_k$ in the word metric given by $g_1,\ldots,g_d$ and their inverses. In $\Cay(F_d,S)$, the neighbors of the origin $e$ are all elements of $F_d$ whose distance to $e$ in the $d$-regular tree $\Cay(F_d,\{g_1,\ldots,g_k,g_1^{-1},\ldots,g_k^{-1})$ belongs to $\{j_1,\ldots,j_k\}$. Then our arguments also apply to $\Cay(F_d,S)$. 

For adjacency operators of $\mathrm{Cay}(\Gamma,S)$ with $\Gamma$ superflexible RACG, Proposition~\ref{p:QEcoxeter} asserts that in order to apply Theorems \ref{t:HNperp}-\ref{t:QEstrongCV},  in terms of properties of $p \in \mathbb C[\Gamma]$, we only need to check \eqref{eq:AC}. We will discuss in \S \ref{s:strongRACG} the existence of converging permutation representations of RACG. As an illustration, we will prove the following corollary which applies for example to the first two diagrams in Figure \ref{f:defdiag}.

\begin{corollary}
    \label{t:QEcoxeter}Let $\Gamma$ be a RACG with a set of generators $S$, whose defining diagram is super-flexible and contains three non-commuting generators commuting with all other generators. Then in $\Cay(\Gamma,S)$, both \eqref{eq:AC} and \eqref{e:RDrenforceeweak}  hold for $p=\mathbf{1}_{S}$ across the whole spectrum except possibly a finite number of points.  Moreover there exist sequences of permutation representations of $\Gamma$ of growing dimension converging in distribution toward $\lambda$. As a consequence, Theorems~\ref{t:asympdecorr} and~\ref{t:HNperp} apply. \end{corollary} 

We shall discuss in  \S \ref{s:strongRACG}  known cases of strongly convergent representations on the orthogonal of low dimensional invariant subspaces where Theorem~\ref{t:QEstrongCV} and Remark \ref{r:QEstrongCV} apply.

\subsubsection{AC condition \eqref{eq:AC} implies the resolvent condition \eqref{e:RDrenforceeweak} in graphs with many automorphisms} \label{s:aroundres}

We use here Lemma~\ref{l:reducresol} to show that the resolvent condition \eqref{e:RDrenforceeweak} is implied by \eqref{eq:AC} for graphs with sufficiently many automorphisms. Recall that an automorphism of a graph is a permutation $\sigma$ of the vertices such that there is an edge between $i$ and $j$ if and only if there is an edge between $\sigma(i)$ and $\sigma(j)$. 
\begin{definition}
In a Cayley graph $(\Gamma,S)$, we say that $g,h\in\Gamma$ are equivalent if there exists a graph automorphism of $(\Gamma,S)$ preserving the unit and sending $g$ to $h$. 
\end{definition}

In the sequel, we assume \eqref{e:lambdaSrhoNS}, i.e., the adjacency matrices are not weighted (equivalently, $p_g=\mathbf{1}(g\in S)$). As a consequence, the resolvent operators to be considered is 
$$
R^z=(\lambda(\mathbf{1}_S)-z\Id)^{-1}.
$$
\begin{lemma}\label{lem:equigreen}
If $g$ and $h$ are equivalent, then  $R^z(e,g)=R^z(e,h)$ for any $z $ not in the spectrum of $\lambda(\mathbf{1}_S)$.
\end{lemma}
\begin{proof}
Let $\sigma$ be a graph automorphism sending $g$ to $h$. Let $A = \lambda(\mathbf{1}_S)$. Since $A^k(i,j)=A^k(\sigma(i),\sigma(j))$ for any $k\in\N$ and $i,j\in\Gamma$, we have in particular $A^k(e,g)=A^k(e,h)$ for any $k\in\N$. Since for $|z|$ larger than the spectral radius of $A$, $R^z = \sum_k A^k / z^{k+1}$, we get that $R^z(e,g)=R^z(e,h)$ for all $|z|$ large enough. By analyticity of the resolvent, this is true for any $z$ not in the spectrum of $A$.
\end{proof}

\begin{proposition}\label{p:largeequivclass}
Let $\Gamma$ be a group with finite generating set $S = S^{-1}$. Assume that the RHS of \eqref{eq:AC} holds for $\lambda(\mathbf{1}_S)$ in an interval $I_0$. If there exists $c>0$ such that any $g\in\Gamma$ has an equivalence class of size $\geq c e^{c|g|}$, then \eqref{e:RDrenforceeweak} holds for any $\lambda_0$ in the interior of $I_0$.
\end{proposition}
\begin{proof}
According to Lemma~\ref{l:reducresol} we only need to prove \eqref{e:convcarrepoly}.
By \eqref{eq:wardineq} and Lemma~\ref{lem:equigreen}, there exists $C_0$ uniform in $g$ and $\eta$ such that 
$$
\eta|(\Im R^z)(e,g)|^2 e^{c|g|}\leq c^{-1}C_0.
$$
Let $\varepsilon>0$.  There exists $R>0$ such that $C_0 c^{-1}e^{-c|g|}|g|^{C_1'}\leq \varepsilon$ for any $g\in\Gamma$ with $|g|> R$.
Hence for any such $g$,
\begin{equation}\label{e:grandR}
\eta|(\Im R^z)(e,g)|^2|g|^{C_1'}\leq C_0 c^{-1}e^{-c|g|}|g|^{C_1'}\leq \varepsilon.
\end{equation}
To handle the $g\in\Gamma$ with $|g|\leq R$, we observe that there exists $\eta_0>0$ such that for any $\eta\leq \eta_0$, 
\begin{equation}\label{e:petitR}
\sup_{|g|\leq R} \eta|(\Im R^z)(e,g)|^2|g|^{C_1'}\leq \varepsilon.
\end{equation}
Indeed, for any $g$, $|(\Im R^z)(e,g)|\leq |\Im R^z(e,e)|^{1/2}|\Im R^z(g,g)|^{1/2}$ which is uniformly bounded as $\eta\rightarrow 0$ due to \eqref{eq:AC}. Multiplying by $\eta$ we get \eqref{e:petitR}. Combining with \eqref{e:grandR} and taking $\varepsilon$ arbitrarily small, we get Proposition~\ref{p:largeequivclass}.
\end{proof}

The assumption that any $g\in\Gamma$ has an equivalence class of size $\geq ce^{c|g|}$ is satisfied in many simple examples. For example, it holds for $\Gamma$ the free group on $d\geq 2$ elements, with generators given by all words of length $j_1,\ldots,j_k$ for some integers $j_1,\ldots,j_k\in\N$. 

\subsubsection{Proof of Proposition~\ref{p:QEcoxeter}} \label{s:resinRACG}

According to Proposition~\ref{p:largeequivclass}, Proposition~\ref{p:QEcoxeter} is immediately implied by the following lemma:
\begin{lemma}\label{l:coxeter}
If the defining diagram of a right-angled Coxeter group $\Gamma$ with generators $S$ is super-flexible, then there exists $c >0$ such that in $\mathrm{Cay}(\Gamma,S)$ each $g\in\Gamma$ has an equivalence class of size $\geq ce^{c|g|}$.
\end{lemma}

To prove this result, we need to recall a few definitions and facts regarding right-angled Coxeter groups. In the sequel, the word ``commutation" refers only to the operation of replacing the sequence $s_is_j$ by the sequence $s_js_i$ for some $i,j$ such that $m_{ij}=2$. The solution of the word problem for RACG (and, more generally, Coxeter groups) is known:

\begin{theorem}[Tits, see Theorem 3.4.2 in \cite{zbMATH08014470}] \label{t:titscoxeter} Let $\Gamma$ be a RACG with set of generators $S$. 
\begin{itemize}
\item A word $w$ in $S$ is reduced if and only if it cannot be shortened by a sequence of commutations and/or deletion of subwords $ss$, $s\in S$.
\item Two reduced words in $S$ define the same group element in $\Gamma$ if and only if one can be transformed into the other by a sequence of commutations.
\end{itemize}
\end{theorem}

For $\phi$ an automorphism of the defining diagram and $w = w_1\cdots w_k$ a word in $S$, we define the word $\phi(w)= \phi(w_1)\cdots \phi(w_k)$. Following \cite[Section 3.4]{white}, we define for $j = 1, 2, \ldots$ the map $\Psi_\phi^j$ on reduced words $w$:
\begin{equation}\label{e:psiphij}
\Psi_\phi^j(w)=\begin{cases} w_1s\phi(w_2) & \text{if \ } w=w_1sw_2 \text{ and $s$ appears exactly $j-1$ times in $w_1$}\\ 
w & \text{if $s$ appears $<j$ times in $w$.}\end{cases}
\end{equation}
In other words, $\Psi_\phi^{j} $ leaves every letter before the $j$-th $s$ unchanged, and applies $\phi$ to the remainder. If there are fewer than $j$ occurrences of $s$, then $\Psi_\phi^j$ fixes the word. We denote by $\mathrm{Stab}(\Gamma)$ the set of unit-preserving automorphisms of the Cayley graph. We will use the following properties:
\begin{enumerate}
    \item Any $\Psi_\phi^j$ is an element of $\mathrm{Stab}(\Gamma)$;
    \item Any $\Psi_\phi^j$ sends reduced words to reduced words.
\end{enumerate}
These properties are easy to check, following the same proof as in  \cite[Section 3.3]{white} (see \cite[Proposition 30]{white} where the first fact is actually stated for $\Psi_\phi^j$).

\begin{proof}[Proof of Lemma~\ref{l:coxeter}] We claim that the
family of automorphisms $\Psi_\phi^j$ allows to construct for each $g\in\Gamma$ a sufficiently large equivalence class (i.e., a sufficiently large orbit under $\mathrm{Stab}(\Gamma)$) to prove Lemma~\ref{l:coxeter}. Consider a reduced word $w$, of length $n\geq 4|S|^2$, and representing an element $g\in\Gamma$. There is one letter, let us call it $s$, which appears at least $n/|S|$ times in $w$.  Since $s$ is of order $2$ and $w$ is reduced, between two occurrences of $s$, there is necessarily another letter which does not commute with $s$. We call the subword between two occurrences of $s$ an ``$s$-segment". So one letter $t$ which does not commute with $s$ appears in $k\geq \frac{n}{|S|^2}-1$ $s$-segments. We denote by $\phi$ the automorphism provided by the definition of (super-flexible) defining diagrams (with $s,t$). 
There are two possibilities corresponding to the two bullet points in  Definition~\ref{d:superflex}. These possibilities are symmetric in $s$ and $t$. Notice also that $s$ appears in at least $(k-1)$ $t$-segments. Therefore, up to exchanging the roles of $s$ and $t$, we may assume that $\phi$ fixes $s$, fixes all neighbors of $s$, exchanges $t$ and $u$ for some $u$ not commuting with $s$ nor $t$, and that $t$ appears in at least $k-1\geq \frac{n}{2|S|^2}$ $s$-segments. We define the finite set $\mathcal{I}$ of natural numbers as follows: $j\in\mathcal{I}$ if and only if there exists an $s$-segment where $t$ appears such that $s$ occurs $j$ times in $w$ before this $s$-segment. We notice that $|\mathcal{I}|\geq k-1$ since there are at least $k-1$ $s$-segments.

We consider the words of the form 
\begin{equation}\label{e:Psiphij}
(\Psi_\phi^{j_\ell}\circ \ldots \circ \Psi_\phi^{j_{1}})(w)
\end{equation}
where $j_1,\ldots,j_\ell$ are distinct elements of $\mathcal{I}$,  with $j_1<j_2<\ldots<j_\ell$ and $\ell\leq |\mathcal{I}|$ (for $\ell=0$ \eqref{e:Psiphij} is equal to $w$). This yields $2^{|\mathcal{I}|+1}$ words. We prove below that no two of these words correspond to the same element of the group. Therefore this construction yields $2^{|\mathcal{I}|+1}\geq 2^k$ distinct elements in the orbit under $\mathrm{Stab}(\Gamma)$ of the group element represented by $w$. Since $k\geq \frac{n}{2|S|^2}$, this concludes the proof of Lemma~\ref{l:coxeter}.

So still for the same $w$, let us prove that no two words of the form \eqref{e:Psiphij} correspond to the same element of the group.
Assume that $(\Psi_\phi^{j_\ell}\circ \ldots \circ \Psi_\phi^{j_{1}})(w)$ and $(\Psi_\phi^{j_{\ell'}'}\circ \ldots \circ \Psi_\phi^{j_1'})(w)$ correspond to the same element in the group, and consider if it exists (for the sake of a contradiction) the least $i\leq \min(\ell,\ell')$ for which $j_i\neq j_i'$ (or $i=\min(\ell,\ell')+1$ otherwise). Without loss of generality, since the $\Psi_\phi^j$ are invertible, we may assume that $i=1$, and we set $j:=\min(j_1,j_1')$. We write
$$
w=w_1sw_2sw_3
$$
where $s$ appears exactly $(j-1)$ times in $w_1$, and moreover  $w_2$ is an $s$-segment containing the letter $t$. Recalling \eqref{e:psiphij},
since $j_1\neq j_1'$, the identity $(\Psi_\phi^{j_\ell}\circ \ldots \circ \Psi_\phi^{j_{1}})(w)=(\Psi_\phi^{j_{\ell'}'}\circ \ldots \circ \Psi_\phi^{j_1'})(w)$  in $\Gamma$ implies 
$$
w_1sw_2sw_4=w_1s\phi(w_2)sw_5
$$ for some words $w_4$ and $w_5$ of same length (because $j<\min(j_2,\ldots,j_\ell,j_2',\ldots,j_{\ell'})$).
Thus, after simplification on the left, we get
$$
w_2sw_4=\phi(w_2)sw_5
$$

Since $(\Psi_\phi^{j_1}\circ \ldots \circ \Psi_\phi^{j_{\ell}})(w)$ and $(\Psi_\phi^{j_1'}\circ \ldots \circ \Psi_\phi^{j_{\ell'}'})(w)$ are reduced words (due to Point 2 above), we know that both members in the above identity are reduced. Therefore, according to Point 2 in Theorem~\ref{t:titscoxeter}, there is a sequence of commutations which transforms $w_2sw_4$ into $\phi(w_2)sw_5$. We will show that this is not possible.

We know that $w_2$ contains the letter $t$. We distinguish two cases:
\begin{itemize}
\item If $t$ appears before $u$ in $w_2$, then writing $w_2=w_6tw_7$ where $w_6$ contains no $t$ and no $u$, we get
$$
w_6tw_7sw_4=\phi(w_6)u\phi(w_7)sw_5.
$$
Since $w_6$ contains no $u$, and since $t$ does not commute with $u$, no sequence of commutation relations applied to the word $w_6tw_7sw_4$ can bring it to a form where $u$ appears before $t$. This is a contradiction.
\item If $u$ appears before $t$, a similar contradiction arises (using $\phi(u)=t$).
\end{itemize}
This proves that $(\Psi_\phi^{j_\ell}\circ \ldots \circ \Psi_\phi^{j_{1}})(w)$ and $(\Psi_\phi^{j_{\ell'}'}\circ \ldots \circ \Psi_\phi^{j_{1}'})(w)$ correspond to different elements in the group.
\end{proof}

\subsubsection{Convergence of representations of RACG} \label{s:strongRACG}

Let $\Gamma$ be a RACG and $S$ its set of generators with its defining diagram. A representation $\rho_N \in \mathrm{Hom}(\Gamma,S_N)$ of dimension $N$ is uniquely characterized by the sequence $(\sigma_s)_{s \in S}$ of matchings (involutive permutations) in $S_N$, $\sigma_s = \rho_N(s)$, which satisfy $\sigma_s \circ \sigma_t = \sigma_t \circ \sigma_s$ for all $s,t$ sharing an edge in the diagram.

In the more general context of asymptotic $\varepsilon$-free independence, for some  $k \in \mathbb N$ (depending on the defining diagram) and $n \to \infty$,  the existence of random $\rho_N \in \mathrm{Hom}(\Gamma,S)$ of dimension $ N = n^k$ which converge in distribution toward $\lambda$ in probability is established in \cite{zbMATH08082671}. The probabilistic model defined in  \cite{zbMATH08082671} is explicit and builds upon  \cite{zbMATH07343938}. 

For RACGs of the form $\Gamma = \mathbb{Z}_2^{*l_1} \times \ldots \times \mathbb{Z}_2^{*l_k}$, i.e. a direct product of free products of $\mathbb Z_2 = \mathbb{Z}/2\mathbb{Z}$ with itself, there exist sequences of random permutation representations of $\Gamma$ of dimension $N = (2n)^{k}$ which converges strongly toward $\lambda$ on an invariant subspace of dimension $(2n-1)^{k}$, see \cite[Theorem 9.5]{arXiv:2304.05714} (actually this reference is for direct products of free groups but the same construction with random matchings in place of random permutations works equally well, as in \cite{zbMATH07128141}). Interestingly \cite[Proposition 2.7]{arXiv:2503.21619} asserts that for such groups (if $l_i \geq 3$ and $k \geq 3$) there are no strongly convergent sequences of permutation representations as per Definition \ref{d:strongconv}.

\subsubsection{Proof of Corollary \ref{t:QEcoxeter}}

We can now conclude the proof of Corollary \ref{t:QEcoxeter}. In view of Proposition~\ref{p:QEcoxeter} and \S \ref{s:strongRACG}, it remains to check that $\lambda ( \mathbf{1}_S)$ satisfies \eqref{eq:AC} across the whole spectrum except a possibly finite number of points.

By assumption, we may write $S$ as the disjoint union of $S_1$ and $S_2$ where $S_1$ is a triple of non-commuting generators, commuting with all elements in $S_2$. We may thus decompose  $\Gamma$ as $\Gamma  = \Gamma_1 \times \Gamma_2$ where $\Gamma_1 =\mathbb Z_2^{*3}$ with generators $S_1$ (and defining diagram three isolated generators) and $\Gamma_2$ generated by $S_2$.  In particular,  $\Cay(\Gamma,S) $ is the Cartesian product of $\Cay(\Gamma_1,S_1)$ and  $\Cay(\Gamma_2,S_2)$. The spectral measure of a Cartesian product is the convolution of the spectral measures. The spectral measure of $\Cay(\Gamma_1,S_1)$ is the Kesten-McKay distribution supported on $[-2 \sqrt{2}, 2 \sqrt{2}]$, it has a positive analytic density on $(-2 \sqrt{2}, 2 \sqrt{2})$. We thus obtain that the spectral measure of $\Cay(\Gamma,S)$ has a bounded continuous density with support a finite number of closed intervals. Moreover, by analyticity, it can only vanish finitely many times on its support.  
\qed

\subsection{Lifts of a finite base graph}\label{s:liftsofgraphs}

In this subsection, we apply Theorem~\ref{t:QEmat} to $N$-lifts of a fixed finite base graph $H = (V,E)$ (multiple edges and loops are allowed). We let $V = \INT{r}$ and denote by $d$ the number of unoriented edges. Let $\Gamma = F_d$ be the free group on $d$ generators. Recall that $F_d$ has the RD property (see \cite{zbMATH03634902}). As explained in Example \ref{ex:covering}, we associate to each free generator $g_i$ a directed edge $(u_i,v_i)$ of $H$ and to $g_i^{-1}$ the reversed edge $(v_i ,u_i)$. The adjacency operator of an $N$-lift of $H$ can be written as 
$$
\rho_N(1_E) = \sum_{i=1}^d \left( E_{v_iu_i} \otimes \rho_N(g_i) + E_{v_iu_i}^* \otimes \rho_N(g_i)^* \right),
$$
where $\rho_N(g_i)  = \sigma_i$ is some permutation matrix on $\INT{N}$ and $1_E = \sum_i (E_{v_iu_i} g_i + E_{v_iu_i}^* g_i^{-1} ) \in  M_r(\mathbb C)[\Gamma]$. It is immediate to check that $g_i \mapsto \rho_N(g_i)$ extends uniquely to a permutation representation $\rho_N \in \mathrm{Hom}(F_d,S_N)$. For $N=1$, we recover the adjacency matrix of $H$. More generally, we consider $p = p^* \in  M_r(\mathbb C)[\Gamma]$ of the form, for some $a \in \mathbb C ^d$, 
\begin{equation}\label{eq:defplift}
p = \sum_g p_g g = \sum_{i=1}^d \left( a_i E_{v_iu_i} g_i + \bar a_i E^*_{v_iu_i} g^{-1}_i \right).
\end{equation}
In words, $\sum_g p_g \in M_r(\mathbb C) $ is a weighted adjacency operator on $H$ and $\rho_N(p)$ is the weighted adjacency operator of an $N$-lift of $H$. Moreover, $\lambda(p)$ is the weighted adjacency operator of $d$ disjoint copies of the universal covering tree of $H$ (and we denote by $\sigma(\lambda(p))$ its spectrum).

Our main result in this section is the following: 
\begin{proposition}\label{t:randomNlift}
Assume that $H$ is a finite connected graph which is not a cycle and with minimal degree at least $2$. For any $p  = p^* \in M_r(\mathbb C)[\Gamma]$ of the form \eqref{eq:defplift} with $a_i \ne 0$ for all $i \in \INT{d}$, there exists a finite set $D \subset \mathbb  R$ such that the following holds. The spectrum of $\lambda(p)$ is a finite union of closed intervals and possibly a finite set of atoms $D_0$ with  $D_0 \subset D$. Any closed interval $I_0$ contained in the interior of $\sigma(\lambda(p))\setminus D$ satisfies properties \eqref{eq:ACmat}-\eqref{e:RDrenforceeweakmat}. 
\end{proposition}

This proposition has the following corollary. Let $(\sigma_i)_{i \in \INT{d}}$ be independent and uniform random permutations in $S_N$. Let $\rho_N \in \mathrm{Hom}(F_d,S_N)$ be defined by $\rho_N(g_i)= \sigma_i$. It is proved in \cite{zbMATH07128141} that $\rho_N$ converges strongly in distribution to $\lambda$ in probability as $N\rightarrow +\infty$. Consequently, from Theorem~\ref{t:QEmat}(iii) applied to $\rho_N(p)$, we get:

\begin{corollary} In the setting of Proposition~\ref{t:randomNlift}, let $I$ be a closed interval in $\sigma(\lambda(p)) \backslash D$. With high probability, the weighted adjacency operator $\rho_N(p)$ of a random $N$-lift of $H$ satisfies  quantum ergodicity \eqref{eq:QEKN} and quantum mixing \eqref{eq:QMKN} for any finite $T \subset \Gamma$ and sequence of $T$-local matrices $(K_N)$ in $M_r (\mathbb C) \otimes M_N(\mathbb C)$ such that for all $t \in T$, $(K_N (t.x,x))_{x \in \INT{N}} \in M_r(\mathbb C) \otimes 1^\perp$ and  $\|K_N\|_{1,\infty} \leq 1$. 
    \end{corollary}

The strength of this corollary is that it establishes quantum mixing, whereas \cite[Corollary 4.1 and Section 4.4]{zbMATH07162035} shows only quantum ergodicity. However, \cite[Corollary 4.1 and Section 4.4]{zbMATH07162035} has the advantage that instead of the centering term 
$$
\sum_{i\in[r]} \langle a_N(\cdot,i)\rangle \sum_{x\in [N]} |\varphi_\alpha(x,i)|^2
$$
in \eqref{e:correctiontermmatrix}, which a priori depends on $\varphi_\alpha$, they obtain, in this diagonal case, the centering term $\langle a_N\rangle_{\lambda_\alpha} = \sum_{i\in [r]} \langle a_N(\cdot,i)\rangle\frac{ \Im R^{\lambda_\alpha+i\eta}(\tilde{i},\tilde{i})}{\sum_{j\in [r]}\Im R^{\lambda_\alpha+i\eta}(\tilde{j},\tilde{j})}$, which only depends on $\lambda_\alpha$.

\begin{proof}[Proof of Proposition~\ref{t:randomNlift}]

Regarding assumption \eqref{eq:ACmat}, since $H$ is a finite connected graph with minimal degree $d\geq 2$ which is not a cycle, \cite[Proposition 4.2]{zbMATH07162035} applies (as is explicitly written in \cite[Section 4.3]{zbMATH07162035}, see below). Therefore, \eqref{eq:ACmat} holds for any $I_0$ which is contained in the interior of $\sigma(\lambda(p))\setminus D$ where $D$ is a finite set possibly containing atoms. We move to \eqref{e:RDrenforceeweakmat}.

\textbf{Step 1.} Our first step is to rephrase \eqref{e:RDrenforceeweakmat} in terms of the decay of the resolvent entries of the universal covering tree of $H$, cf. \eqref{e:fourthmomentdecompogoal} below.  For this we partly follow \cite[Appendix]{arXiv:2304.05714}. For $i \in \INT{d}$, we set $g_{i+d} = g_i^{-1}$, $i^* = i+d$ and $(i+d)^* = i$. The resolvent $R^z=(\lambda(p)-z\Id)^{-1}$ evaluated at $(a,b)\in \Gamma\times\Gamma$ is an element of $M_r(\mathbb{C})$. Let $\lambda^{(e)}(p)$ be the restriction of $\lambda(p)$ to $\mathbb{C}^r\otimes \ell^2(\Gamma\setminus\{e\})$. For $z\in \mathbb{C}\setminus \sigma(\lambda^{(e)}(p))$, we denote by 
$$
\gamma^z_{i}=\left(( \lambda^{(e)}(p)-z\Id)^{-1}\right)_{g_{i}g_{i}} \in M_r(\mathbb C)
$$  
the resolvent in the subtree of $\mathrm{Cay}(F_d,(g_i)_{i \in \INT{2d}})$ emanating from $g_{i}$, and we let $Z^z_{i}=-\gamma^z_{i}p_{g_i}\in M_r(\mathbb{C})$.
We first prove that for any $g\in\Gamma$ written in reduced form as $g=g_{i_\ell}\ldots g_{i_1}$, with $i_k \in \INT{2d}$, we have 
\begin{equation}\label{e:recresmat}
 R^z(g,e)= Z^z_{i_\ell} \ldots Z^z_{i_1} R^z(e,e).
\end{equation}

For $i\le 2d$, we denote by $V_i\subset \Gamma$ the subset of group elements $h\in\Gamma$ written
in reduced form $h = g_{j_k}\ldots g_{j_2}g_{j_1}$ with $j_1 = i$ and $k\in\N$.
A sequence $(j_1,\ldots,j_k)$ such that $g=g_{j_k}\ldots g_{j_1}$ determines a walk of length $k$ in $\Gamma$ starting from $e$ and ending at $g$. Since $\Gamma$ is the free group and the $g_i$'s are the free generators, we may decompose this walk at the last passage at $e$, $g_{i_1}$, $g_{i_2}g_{i_1}$, $\ldots$, $g_{i_\ell}\ldots g_{i_1}=g$. This induces a decomposition of the sequence $(j_1,\ldots,j_k)$ into $\ell+1$ sequences $s_0,\ldots,s_\ell$ (possibly empty). For any $0\leq t\leq \ell$, the sequence $s_t$ is a walk from $g_{i_t}$ to $g_{i_t}$ in $V_{i_t}$ and $(j_1,\ldots, j_k) = (s_0,i_1,s_1,\ldots,i_\ell,s_\ell)$. Writing $s_t$ as $s_t=(j_{t,1},\ldots,j_{t,n_t})$, we thus have (setting just here $p_i=p_{g_i}$) 
$$
\sum_k \frac{\lambda(p)^k(g,e)}{z^{k+1}}=\sum_k\sum_{\substack{(j_1,\dots,j_k)\\ g=g_{j_k}\cdots g_{j_1}}} \frac{p_{j_k}\cdots p_{j_1}}{z^k}=\sum_{s_0,\dots,s_\ell} \frac{p_{j_{\ell,n_\ell}}\cdots p_{j_{\ell,1}}}{z^{n_\ell}}\frac{p_{i_\ell}}{z}\cdots \frac{p_{j_{1,n_1}}\cdots p_{j_{1,1}}}{z^{n_1}}\frac{p_{i_1}}{z}\frac{p_{j_{0,n_0}}\cdots p_{j_{0,1}}}{z^{n_0+1}}
$$
where the sum is over all sequences $(s_0,\ldots,s_\ell)$ of non-negative lengths $(n_0,\ldots,n_\ell)$ as above. Observe now that for $z$ with large enough modulus, 
\begin{equation}\label{e:franceirlande}
R^z(e,e)=-\sum_{n_0\geq 0}\frac{\lambda(p)^{n_0}(e,e)}{z^{n_0+1}}=-\sum_{n_0\geq 0}\sum_{s_0}\frac{p_{j_{0,n_0}}\ldots p_{j_{0,1}}}{z^{n_0+1}}
\end{equation}
where the sum runs over all sequences $s_0=(j_{0,1},\ldots,j_{0,n_0})$ of length $n_0$ from $e$ to $e$. Similarly,
$$
Z_t^z=-\gamma_t^zp_{i_t}=\sum_{n_t\geq 0}\sum_{s_t} \frac{p_{j_{t,n_t}}\ldots p_{j_{t,1}}}{z^{n_t}}\frac{p_{i_t}}{z}
$$
where the sum is over all sequences $s_t= (j_{t,1},\ldots,j_{t,n_t})$ of length $n_t$ from $g_{i_t}$ to $g_{i_t}$ in $V_{i_t}$.
This concludes the proof of \eqref{e:recresmat} for $z$ with large enough modulus, and thus for any $z$ in the upper half-plane since both expressions in \eqref{e:recresmat} are analytic. %\textbf{\textcolor{cyan}{Par coherence il faudrait aussi des somme sur $n_i$ dans l'eq de $P^k(g,e)$}} \textcolor{red}{Je suis d'accord mais ça sera moche je pense...}

The following formula, which extends \cite[Proposition 2.1]{zbMATH01123534} to the matricial framework, is proved in the proof of \cite[Lemma 11]{zbMATH07128141}: 
\begin{equation}
\gamma_i^z=-\Bigl(z\Id_r+\sum_{j\neq i^*} p_{g_j}^* \gamma_j^z p_{g_j}\Bigr)^{-1}. \label{e:2res}
\end{equation}
It is a direct consequence of the resolvent identity. It can also be proved using the generating function of walks as we did in the proof of \eqref{e:recresmat}. 

Let us now specify the relations \eqref{e:recresmat} and \eqref{e:2res} to the case of $p$ of the form \eqref{eq:defplift}. In this case, $p_{g_i} =a_i E_{v_iu_i}$ for all $i \in \INT{2d}$, where for $i \in \INT{d}$, we have set $a_{i^*} = \bar a_i$ and $(u_{i^*},v_{i^*}) = (v_i,u_i)$. In this case the relation \eqref{e:2res} takes a simpler form: indeed, the matrix $p_{g_j}^*\gamma_j^zp_{g_j}$ is diagonal, and has only one non-vanishing coefficient, in the entry $(u_j,u_j)$. Let $\hat \gamma_i^z=(\gamma_i^z)_{v_iv_i}$, then \eqref{e:2res} implies the scalar equation
$$
\hat \gamma_i^z =-\Bigl(z+\sum_{j\neq i^*}|a_j|^2 \hat \gamma_j^z e_{v_i}^*E_{u_ju_j}e_{v_i}\Bigr)^{-1}=- \Bigl(z+\sum_{\substack{(u_j,v_j)\\ u_j=v_i,\ j \ne i^*}}|a_j|^2 \hat \gamma_j^z\Bigr)^{-1}
$$
where $e_{v_i}\in\mathbb{C}^{r}$ denotes the $i$-th vector of the canonical basis. We thus retrieve the equations appearing in \cite[Section 4.3]{zbMATH07162035}. In particular, we have that $|\hat \gamma_i^z|$ is uniformly bounded below as $\eta\rightarrow 0$ if $\Re(z) \in I_0 $.

The matrix $R^z(e,e)\in M_r(\mathbb{C})$ is diagonal due to \eqref{e:franceirlande} and the fact that for any sequence $(j_1,\ldots,j_k)$ such that $g_{j_k}\ldots g_{j_1}=e$, we have $u_{j_1}=v_{j_k}$. Let $\zeta^z_i= - a_i \hat \gamma_i^z$. Since $g=g_{i_\ell}\ldots g_{i_1}$ and $p_{g_i}=a_iE_{v_iu_i}$, the entry $(v,u)$ of $R^z(g,e)$ in \eqref{e:recresmat} is equal to
\begin{equation}\label{e:Rpzeeui1}
R^z(g,e)_{vu}= \zeta_{i_\ell}^z \ldots\zeta_{i_1}^z  R^z(e,e)_{u_{i_1}u}
\end{equation}
if $u_{i_1}=u$, $v_{i_\ell}=v$, and $v_{i_j}=u_{i_{j+1}}$ for any $1\leq j\leq \ell-1$. It is equal to $0$ otherwise.
Hence, the Ward identity \eqref{e:ward} becomes in the present context
\begin{equation}\label{e:wardici}
\Im R^z(e,e)_{uu}=\sum_{g,v} \eta |R^z(g,e)_{vu}|^2=\eta\sum_\ell \sum_{\substack{(i_1,\ldots,i_\ell)\in\mathcal{I}_\ell \\ u_{i_1}=u}}  |\zeta_{i_1}^z \ldots  \zeta_{i_\ell}^z|^2|R^z(e,e)_{u_{i_1}u}|^2.
\end{equation}
Here,
$$
\mathcal{I}_\ell=\{(i_1,\ldots,i_\ell)\in [2d]^\ell \mid i_{k} \ne i_{k+1}^*, v_{i_k}=u_{i_{k+1}} \text{\ for any\ } 1\leq k\leq \ell-1\},
$$
that is, each $i\in\mathcal{I}_\ell$ indexes a sequence of $\ell$ adjacent directed edges $(u_1,v_1)$, $(u_2,v_2)$, etc. in $H$ without backtracking.
For all but a finite set $D$ if $\Re(z) \notin D$, for all $u \in \INT{r}$, the matrix entries $R^z(e,e)_{uu}$ have non-vanishing modulus in the limit $\eta\rightarrow 0$ (this is mostly due to \cite[Proposition 4.2]{zbMATH07162035}). Therefore, \eqref{e:wardici} implies that 
\begin{equation}\label{e:petitsgamma}
 \sum_{\ell\in\N} \sum_{(i_1,\ldots,i_\ell)\in\mathcal{I}_\ell}  |\zeta_{i_1}^z \ldots  \zeta_{i_\ell}^z|^2 \leq \frac{C}{\eta}
\end{equation}
for some $C>0$ independent of $\eta$. 

%\textbf{\textcolor{cyan}{je suis toujours embet\'e par la presence de $u$ et $v$. Est-ce qu'on peut indexer ici et par la suite par les $g_i$ plutot ? On dit au debut que $g_i=(u_i,v_i)$ sont les aretes dirigees. Ca me parait mieux d'indexer par $g_i$ surtout pour les multigraphes.}}

\textbf{Step 2.}
In view of \eqref{e:Rpzeeui1} and 
since $R^z(e,e)$ has all its entries uniformly bounded above in modulus as $\eta\rightarrow0$, proving \eqref{e:RDrenforceeweakmat} reduces to showing 
\begin{equation}\label{e:fourthmomentdecompogoal}
\lim_{\eta\rightarrow 0} \eta^2\sum_{\ell\in\N}\ell^{C_1'} \sum_{(i_1,\ldots,i_\ell)\in\mathcal{I}_\ell} |\zeta_{i_1}^z \ldots  \zeta_{i_\ell}^z|^4=0.
\end{equation}
Such a bound essentially follows from \cite[Section~5.2]{zbMATH07172075} in case the minimal degree is $\ge 3$. We provide below an alternative argument which does not need this restriction.

We call a {\em loop} any sequence of the form $(i_1,\ldots,i_\ell)\in\mathcal{I}_\ell$ for some $\ell\in\N$, such that $v_{i_\ell}=u_{i_1}$, $i_1 \ne i_\ell^*$, and all the $i_j$ are distinct  (note that the visited vertices $(u_{i_1},\ldots,u_{i_\ell})$ are not necessarily distinct). Two loops are distinct if they are not equal up to cyclic ordering.  Our assumption on $H$ implies that for any loop $(i_1,\ldots,i_\ell)$, there exists a distinct loop $(j_1,\ldots, j_{k})$ sharing a common vertex. 
We may assume without loss of generality that $u_{i_1}=u_{j_1}$. Moreover, up to replacing $(j_1,\ldots, j_{k})$ by its inverse $(j_{k}^*, \ldots , j_1^*)$, we may assume that the two loops can be composed as non-backtracking paths (note that one loop could be a subloop of the other, for example if $H$ is a dumbbell graph). Then, repeating successively these two loops in any order,  \eqref{e:petitsgamma} implies that
$$
\sum_{n\in\N} \sum_{s=0}^n \binom{n}{s} |\zeta_{i_1}^z \ldots  \zeta_{i_\ell}^z|^{2s}  | \zeta_{j_1}^z \ldots \zeta_{j_{k}}^z|^{2(n-s)} \leq \frac{C}{\eta}
$$
and therefore
$$
\frac{1}{1-|\zeta_{i_1}^z \ldots  \zeta_{i_\ell}^z|^2-| \zeta_{j_1}^z \ldots  \zeta_{j_{k}}^z|^2}\leq \frac{C}{\eta}.
$$
Using that each $|\hat \gamma_i^z|$ is uniformly bounded below as $\eta\rightarrow 0$ if $\Re(z) \in \sigma(\lambda(p))\backslash D$, we obtain that there exists $\delta>0$ such that for any loop $(i_1,\ldots,i_{\ell})$ 
\begin{equation}\label{e:boundcycle1-delta}
|\zeta_{i_1}^z \ldots \zeta_{i_{\ell}}^z|\leq 1-\delta
\end{equation}
uniformly for $\eta$ sufficiently small. Since there are finitely many loops, we can take $\delta >0$ uniform over all loops of $H$.

We use this bound to find a bound on the modulus of arbitrary products of the form \eqref{e:Rpzeeui1}. 
Take a path $(i_1,\ldots,i_\ell)\in\mathcal{I}_\ell$, which is not necessarily a loop. We apply to $(i_1,\ldots, i_\ell)$ a loop erasure procedure: we construct a new path (an element of $\mathcal{I}_{\ell'}$ for some $\ell'\leq \ell$) by successively removing loops  as follows:
\begin{itemize}
\item Traverse $i_1\ldots i_\ell$ in order.
\item As soon as an index appears for the second time, delete the entire segment of the path between the two occurrences (this segment forms a loop, and it has at most $2d$ elements).
\item Continue until no index appears more than once.
\end{itemize}
The key observation is that when we remove a loop of length $k$ in the second step, the remaining sequence is an element of $\mathcal I_{\ell - k}$. The output of the procedure is a path with no repeated indices, and therefore it is of length at most $2d$. We thus have decomposed a path in $\mathcal I_\ell$ into a sequence in $\mathcal I_k$ with $k \leq 2d$ and at least $(\ell-k)/2d$ loops (since each loop has length at most $2d$).
Since the product of the $|\zeta_i^z|$ in each deleted loop is at most $1-\delta$ according to \eqref{e:boundcycle1-delta}, this algorithm yields the bound
$$
| \zeta_{i_1}^z \ldots  \zeta_{i_\ell}^z|\leq c_0^{2d} (1-\delta)^{\frac{\ell-2d}{2d}},
$$
where $c_0 \geq 1$ denotes an upper bound on $| \zeta_i^z|$ uniform in $i\in[2d]$ and in $\eta$ as $\eta\rightarrow0$.
As a consequence, for some new constants $\delta',c'_0>0$, for all $(i_1,\ldots,i_\ell)\in\mathcal{I}_\ell$, uniformly in $\eta > 0$,
\begin{equation}\label{e:supellC1'}
| \zeta_{i_1}^z \ldots \zeta_{i_\ell}^z| \leq c_0' (1-\delta')^{\ell}.
\end{equation}
In particular, for any $\eta > 0$, we have using \eqref{e:petitsgamma}-\eqref{e:supellC1'} that
\begin{align*}
&\eta^2\sum_{\ell\in \N}\ell^{C_1'} \sum_{(i_1,\ldots,i_\ell)\in\mathcal{I}_\ell} | \zeta_{i_1}^z \ldots \zeta_{i_\ell}^z|^4\\
&\qquad \qquad\leq  \eta \Bigl(\eta\sum_{\ell\in\N} \sum_{(i_1,\ldots,i_\ell)\in\mathcal{I}_\ell} | \zeta_{i_1}^z \ldots  \zeta_{i_\ell}^z|^2\Bigr)\sup_{\ell\in \N} \left(  \ell^{C_1'}| \zeta_{i_1}^z \ldots  \zeta_{i_\ell}^z|^2\right)\\
&\qquad\qquad \leq C \eta \sup_{\ell\in \N} \left( \ell^{C_1'} {c_0'}^2 (1-\delta')^{2\ell}\right) \longrightarrow 0
\end{align*}
as $\eta\to 0$. This concludes the proof of \eqref{e:fourthmomentdecompogoal}, and thus of \eqref{e:RDrenforceeweakmat}.
\end{proof}

\section{Necessity of the assumptions}\label{s:necessityassump}

Beyond the applications of our results discussed in Section~\ref{s:applications}, it is also of interest to present examples where quantum ergodicity and/or quantum mixing \emph{fail}. In this section, we provide examples from the literature which demonstrate that quantum ergodicity/mixing can fail as soon as any of the assumptions in our main results (Theorems \ref{t:asympdecorr}, \ref{t:QEstrongCV} and \ref{t:QEmat}) is removed. This highlights their necessity.

\subsection{The asymptotic uncorrelation condition in Theorem~\ref{t:asympdecorr}}\label{s:asympdecorr}

Following \cite[Section 4]{zbMATH07514683}, we have the following result:
\begin{proposition}
Under the same assumptions of Theorem~\ref{t:asympdecorr}, except that the (diagonal) observables $a_N:[N]\rightarrow \mathbb{C}$ are not assumed to be asymptotically uncorrelated, but only to satisfy $\|a_N\|_{\infty}\leq 1$ and $\langle a_N\rangle =0$, the conclusion of Theorem~\ref{t:asympdecorr} does not necessarily hold, and even quantum ergodicity \eqref{eq:QEKN} can fail.
\end{proposition}
\begin{proof}
We take the same example as in \cite[Section 4]{zbMATH07514683}. Consider any family of $d$-regular graphs $(H_N)$ for $d$ even and $d\geq 8$ such that $|H_N|=N$ and $(H_N)$ converges in the sense of Benjamini-Schramm to the $d$-regular tree. We construct a family of graphs $(G_N)$ as follows (see \cite[Figure 3]{zbMATH07514683}). First, we delete an arbitrary edge of $H_N$, and call this new graph $H_N'$. We create $d/2$ copies of $H_N'$, add a vertex $v_N$, and add an edge from $v_N$ to each of the $d$ vertices of degree $d-1$, two for each copy of $H_N'$. We call this graph $G_N$, and enumerate the copies of $H_N'$ in $G_N$ as $H_{N,1}',\ldots, H_{N,d/2}'$. Since $G_N$ is a $d$-regular graph of even degree, it is a Schreier graph, and its adjacency matrix is of the form $\rho_N(\mathbf{1}_S)$ for some representation $\rho_N$ and for the canonical set $S$ of generators of the free group $F_d$ (see Section~\ref{s:setup}). Moreover, $\rho_N$ converges weakly in distribution to the regular representation of $F_d$, since $(H_N)$ (and thus $(G_N)$) converges in the Benjamini-Schramm sense to the $d$-regular tree. 

It is observed in \cite[Theorem 4.2]{zbMATH07514683} that the family of graphs $(G_N)$ has an orthonormal eigenbasis which violates quantum ergodicity. The observable used to prove this is
$$
a_N(x)=\begin{cases} 1 & \text{if } x\in V_{H'_{N,1}}\cup V_{H'_{N,2}} \\  -1 & \text{if } x\in V_{H'_{N,3}}\cup V_{H'_{N,4}} \\
0 & \text{otherwise.} \end{cases}
$$
The conditions $\|a_N\|_{\infty}\leq 1$ and $\langle a_N\rangle=0$ are satisfied, but $(a_N)$ is \emph{not} asymptotically uncorrelated, otherwise Theorem~\ref{t:asympdecorr} would apply, and quantum ergodicity would hold. One can also check by hand that the asymptotic uncorrelation condition \eqref{e:epsgconv} does not hold. 
\end{proof}

\subsection{Necessity of the resolvent condition \eqref{e:RDrenforceeweak}}

\subsubsection{The $\Z^d$ case: no quantum mixing} \label{s:Zd}
In this section, we consider the additive group $\Gamma=\mathbb{Z}^d$ with the canonical set of generators $S=\{e_1,\ldots,e_d\}$ (the only non-vanishing coordinate of $e_i\in\mathbb{Z}^d$ is the $i$-th one, with entry $1$). We consider a sequence of Schreier graphs given by large boxes $(\mathbb{Z}/M\mathbb{Z})^d$: we let $N=M^d$ and consider $\rho_{N}(g)(x)=(g {\ \rm mod\ } M)+x$ where the addition takes place in $(\mathbb{Z}/M\mathbb{Z})^d$.

Quantum ergodicity (for diagonal observables) is known to hold in this set-up, see \cite[Theorem 1.1]{zbMATH07751295}. However, quantum mixing fails:
\begin{proposition}\label{p:Zd}
Quantum mixing (i.e., \eqref{eq:QMKN}) does not hold in the sequence of Schreier graphs $(\mathbb{Z}/M\mathbb{Z})^d$ described above: there exists a sequence of diagonal observables $(a_M)$ satisfying $\|a_M\|_\infty \leq 1$ and violating quantum mixing \eqref{eq:QMKN}.
\end{proposition}
Before proving Proposition~\ref{p:Zd}, let us observe that Theorem~\ref{t:QEstrongCV} does not apply -- otherwise quantum mixing would actually hold! -- because the resolvent condition \eqref{e:RDrenforceeweak} is not met. However, all other assumptions of Theorem~\ref{t:QEstrongCV} are satisfied: 
\begin{itemize}
    \item $\Z^d$ satisfies the rapid decay property (see \cite{zbMATH06859874});
    \item the spectral measure of the adjacency matrix of the Cayley graph $\mathbb{Z}^d$ is obtained as the convolution of $d$ arcsine distributions. Its support is $[-2d,2d]$ and is absolutely continuous. For $d=2$, the density is bounded above and below except at $\{-4,0,4\}$, in particular \eqref{eq:AC} holds on any compact interval $I_0\subset (-4,4)\setminus\{0\}$. For $d\geq 3$ the density is bounded over $(-2d,2d)$ and \eqref{eq:AC} holds on any compact interval $I_0\subset (-2d,2d)$.
    \item $\rho_{N}$ converges in distribution to $\lambda$ since for any $g\in\Z^d\setminus\{0\}$, ${\rm Tr}(\rho_N(g))=0$ for $N$ large enough. As a consequence of Proposition~\ref{p:amenableBSstrong}, $\rho_N$ also converges strongly to $\lambda$.
\end{itemize}
This shows the necessity of the resolvent assumption \eqref{e:RDrenforceeweak} in Theorem~\ref{t:QEstrongCV}.

We turn to the proof of Proposition~\ref{p:Zd}. The main idea is that eigenvectors with nearby wave vectors (denoted by $k,k'$ in the proof below) are highly correlated.
\begin{proof}[Proof of Proposition~\ref{p:Zd}]
We omit $M$ and $N$ most of the time in the notation. We consider the canonical eigenbasis $\varphi_k(j)=\frac{1}{\sqrt{N}}e^{2i\pi\frac{k\cdot j}{M}}$ of $(\mathbb{Z}/M\mathbb{Z})^d$, indexed by $k\in(\Z/M\Z)^d$, and the associated eigenvalue $\lambda_k=\cos(2\pi k_1/M)+\ldots+\cos(2\pi k_d/M)$. For any real-valued $a$ on $(\Z/M\Z)^d$, there holds
$$
\langle \varphi_k, a \varphi_{k'}\rangle=\frac{1}{N} \sum_{j\in (\Z/M\Z)^d} e^{2i\pi\frac{(k-k')\cdot j}{M}}a(j)=\frac{1}{\sqrt{N}} \widehat{a}(k'-k)
$$
where $\widehat{a}$ denotes the discrete Fourier transform of $a$, defined as 
$$
\widehat{a}(k)=\frac{1}{\sqrt{N}}\sum_{j\in (\Z/M\Z)^d} e^{-2i\pi\frac{k\cdot j}{M}}a(j)
$$
Hence, fixing $\eta>0$, if $I^\eta$ is an interval of size $\eta$ included in the interior of $(-2d,2d)$ and not containing $0$ if $d=2$, and $\langle a\rangle=0$ (meaning that $\widehat{a}(0)=0$), we have 
\begin{equation}\label{e:sumcorrel}
\frac{1}{|\Lambda_{I^\eta}|}\sum_{k,k'\in \Lambda_{I^\eta}} |\langle \varphi_k, a \varphi_{k'}\rangle|^2=\frac{1}{|\Lambda_{I^\eta}|}\sum_{k\in \Lambda_{I^\eta}}\sum_{k'\in \Lambda_{I^\eta}} \frac{|\widehat{a}(k'-k)|^2}{N}.
\end{equation}
Let $u=(1,0,\ldots,0)\in(\Z/M\Z)^d$. 

We claim that for $N$ large enough, at least $\frac12 |\Lambda_{I^\eta}|$ eigenvalues $\lambda_k\in I^\eta$ satisfy $\lambda_{k+u}\in I^\eta$. In fact, if $I^\eta=[b,c]$, choose $I_0^\delta := [b+\delta,c-\delta]\subset I^\eta$ such that $\mu_p(I_0^\delta)\ge \frac{3}{4}\mu_p(I^\eta)$, where $\mu_p$ is the spectral density of $\Z^d$ with standard generators. This choice is possible for $\delta>0$ small enough since $\mu_p$ has a bounded strictly positive density on $I^\eta$ by our assumptions.
%
%Our assumptions on $I^\eta$ imply that the spectral density $\mu_p$ of the Cayley graph $\Z^d$ with standard generators is bounded and strictly positive in this interval. Let $I^\eta=[b,c]$, let $0<\delta<\frac12(c-b)$ such that $\mu_p(I^\delta_0)\geq \frac34 \mu(I^\eta)>0$ where $I^\delta_0=[b+\delta,c-\delta]$. 
Since the spectral measure of the box $(\Z/M\Z)^d$ converges weakly as $M\rightarrow +\infty$ to $\mu_p$, we have $|\Lambda_{I^\delta_0}|\geq \frac12 |\Lambda_{I^\eta}|$ for $M$ large enough. Using finally the fact that $|\lambda_{k+u}-\lambda_k|\leq \frac{2\pi}{M}\leq \delta$ for $M$ large enough, we get that all eigenvalues $\lambda_k$ in $I^\delta_0$ satisfy $\lambda_{k+u}\in I^\eta$, which concludes the proof of the claim.
%\textbf{\textcolor{cyan}{I don't think $\frac{1}{2}$ is possible. It can be that we are unlucky and $\frac{3}{4}$ for the density is near the edges.}}

%\textbf{\textcolor{cyan}{Instead: We claim for $N$ large enough, at least $c|\Lambda_{I^\eta}|$ eigenvalues $\lambda_k\in I^\eta$ satisfy $\lambda_{k+u}\in I^\eta$. In fact, $|\lambda_{k+u}-\lambda_k|\le \frac{2\pi}{M}$. Let $I^\eta = J_E^\eta$ and choose $\delta$ such that $\delta+\frac{2\pi}{M}<\eta$ and $|\Lambda_{J_E^\delta}|\ge c |\Lambda_{J_E^\eta}|$. This is possible by Lemma~\ref{cor:CMS2}, e.g. with $\delta=\frac{\eta}{2}$ and $c=\frac{1}{8(C_0')^2}$, assuming $M$ is large. Then the claim is satisfied for any $k\in J_E^\delta$ as required. }}

We now conclude the proof of Proposition~\ref{p:Zd}. We choose $a(n)=e^{\frac{2i\pi n_1}{M}}$, i.e., $\widehat{a}(\ell)=M^{d/2} \delta_{\ell=u}$. By the previous paragraph, at least half the $k\in\Lambda_{I^\eta}$ satisfy $\sum_{k'\in\Lambda_{I^\eta}} \frac{|\widehat{a}(k'-k)|^2}{N}\ge \frac{|\widehat{a}(k+u-k)|^2}{N} = 1$. Hence, the sum in \eqref{e:sumcorrel} is greater than $\frac12$, which shows that quantum mixing \eqref{eq:QMKN} does not hold.

To show the same phenomenon for a real eigenbasis, we consider $\varphi_k^{(1)}=\sqrt{\frac2M}\cos(2\pi k\cdot j/M)$ and $\varphi_k^{(2)}=\sqrt{\frac2M}\sin(2\pi k\cdot j/M)$. Using that $\varphi_k=\frac{1}{\sqrt{2}}(\varphi_k^{(1)}+i\varphi_k^{(2)})$, we get
$$
|\langle \varphi_k, a \varphi_{k'}\rangle|^2\leq |\langle \varphi^{(1)}_k, a \varphi^{(1)}_{k'}\rangle|^2+|\langle \varphi_k^{(1)}, a \varphi_{k'}^{(2)}\rangle|^2+|\langle \varphi_k^{(2)}, a \varphi_{k'}^{(1)}\rangle|^2+|\langle \varphi_k^{(2)}, a \varphi_{k'}^{(2)}\rangle|^2
$$
and the same result for this real eigenbasis follows. 
\end{proof}

\subsubsection{The butterfly: necessity of the resolvent condition \eqref{e:RDrenforceeweakmat} in Theorem~\ref{t:QEmat}(iii)}
\label{s:necessityresolvent}

In this section, we provide a family of graphs in which all assumptions of Theorem~\ref{t:QEmat}(iii) hold, except the $4$-th moment condition  \eqref{e:RDrenforceeweakmat} on the resolvent, and quantum ergodicity in the sense of \eqref{eq:QEKN} fails. This shows the necessity of a resolvent assumption like \eqref{e:RDrenforceeweakmat} for quantum ergodicity to hold. This example was first considered in \cite{zbMATH07751295}. 

Recall that the tensor product $G_1\times G_2$ of 
two graphs $G_1=(V_1,E_1)$ and $G_2=(V_2,E_2)$ is defined as the graph with vertex set $V_1\times V_2$, and edge set given as follows: there is an edge between $(u_1,u_2)\in V_1\times V_2$ and $(v_1,v_2)\in V_1\times V_2$ if and only if $(u_1,v_1)\in E_1$ and $(u_2,v_2)\in E_2$.

Let $G_F$ be the butterfly graph in Figure \ref{f:butterfly}, and consider the tensor product $\mathbb{Z}\times G_F$. The graph $\mathbb{Z}\times G_F$, which is not regular, fits into the matricial framework of Section~\ref{s:matricial}: its adjacency matrix reads
$$
\sum_{g\in \{-1,1\}} a\otimes \lambda(g)
$$
where $a$ denotes the adjacency matrix of the butterfly graph and $\Gamma=\Z$. If the circle $\mathbb{Z}/N\mathbb{Z}$ is endowed with the action $\rho_N$ given for $g\in \mathbb{Z}$ by $\rho_N(g)(x)=(g {\ \rm mod\ } N)+x$, we observe that $\rho_N$ converges in distribution to the regular representation $\lambda$ of $\mathbb{Z}$: this is due to Proposition~\ref{p:conv} and the fact that the cycle graph with $N$ vertices converges in the Benjamini-Schramm sense to $\Z$ as $N\rightarrow +\infty$.

All conditions in Theorem~\ref{t:QEmat}(iii) except \eqref{e:RDrenforceeweakmat} are satisfied. Firstly, the spectrum is purely absolutely continuous, see \cite[Proposition 1.5 and Section 3.4.2]{zbMATH07751295}. Secondly, $\mathbb{Z}$ satisfies the rapid decay property \eqref{e:rddef} since it has polynomial growth. Finally, since $\rho_N$ converges in distribution to $\mathbb{Z}$, it also converges strongly according to Proposition~\ref{p:amenableBSstrong}. 

\begin{figure}[h!]
\begin{center}
\begin{tikzpicture}[scale=2, every node/.style={circle, fill=black, inner sep=1.5pt}, vertex/.style={circle, fill=black, inner sep=1.7pt},
    vlabel/.style={font=\small}]

    \node (v1) at (-3,1.5) [label=left:$v_1$] {};
    \node (v2) at (-3,0.5) [label=left:$v_2$] {};
    \node (v3) at (-2,1) [label=above:$v_3$] {};
    \node (v4) at (-1,1.5) [label=right:$v_4$] {};
    \node (v5) at (-1,0.5) [label=right:$v_5$] {};

    \draw (v1)--(v2)--(v3)--(v1);
    \draw (v3)--(v4)--(v5)--(v3);

\node[vertex] (T1) at (0.0,2)  {};
\node[vertex] (T2) at (0.9,2.3)  {};
\node[vertex] (T3) at (1.8,2)  {};
\node[vertex] (T4) at (2.7,2.3)  {};
\node[vertex] (T5) at (3.6,2)  {};

\node[vertex, fill=red] (M1) at (0.0,1)  {};
\node[vertex, fill=red] (M2) at (0.9,1.3)  {};
\node[vertex, fill=red] (M3) at (1.8,1)  {};
\node[vertex, fill=red] (M4) at (2.7,1.3)  {};
\node[vertex, fill=red] (M5) at (3.6,1)  {};

\node[vertex] (B1) at (0.0,0)  {};
\node[vertex] (B2) at (0.9,0.3)  {};
\node[vertex] (B3) at (1.8,0)  {};
\node[vertex] (B4) at (2.7,0.3)  {};
\node[vertex] (B5) at (3.6,0)  {};

\draw (T1) -- (M2);
\draw (T1) -- (M3);

\draw (T2) -- (M1);
\draw (T2) -- (M3);

\draw (T3) -- (M1);
\draw (T3) -- (M2);
\draw (T3) -- (M4);
\draw (T3) -- (M5);

\draw (T4) -- (M3);
\draw (T4) -- (M5);

\draw (T5) -- (M3);
\draw (T5) -- (M4);

\draw (M1) -- (B2);
\draw (M1) -- (B3);

\draw (M2) -- (B1);
\draw (M2) -- (B3);

\draw (M3) -- (B1);
\draw (M3) -- (B2);
\draw (M3) -- (B4);
\draw (M3) -- (B5);

\draw (M4) -- (B3);
\draw (M4) -- (B5);

\draw (M5) -- (B3);
\draw (M5) -- (B4);
\end{tikzpicture}
\caption{The butterfly graph $G_F$ (left) and part of the tensor product $\Z\times G_F$ (right). A fundamental set is colored in red.} \label{f:butterfly}
\end{center}
\end{figure}
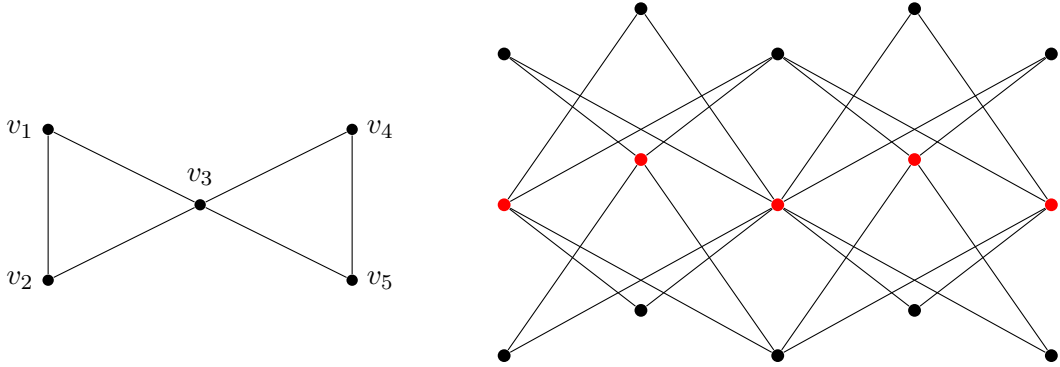

It is shown in \cite[Proposition 1.5]{zbMATH07751295} that quantum ergodicity does not hold for the adjacency matrix of $\Gamma_N=(\mathbb{Z}/N\mathbb{Z})\times G_F$. Indeed, for the subsequence of even $N$, there exists an orthonormal basis $(\psi_j^{(N)})_{j\in \Gamma_N}$ and a sequence of observables $(a_N)$ such that
$$
\frac{1}{|\Gamma_N|}\sum_{j\in\Gamma_N} |\langle \psi_j,a_N\psi_j\rangle|^2\geq \frac{1}{320}.
$$
Their sequence of observables satisfies $\frac1N \sum_{i\in[N]} a_N(i,x)=0$ for any $x\in G_F$ (in particular, $\langle a_N\rangle=0$), which means that $a_N$ belongs to $\mathbb{C}^r \otimes 1^\perp$.

\subsection{Necessity of $a_N\in \mathbb{C}^r \otimes 1^\perp$ in Theorem~\ref{t:QEmat}(iii)}\label{s:cartesiancycle}
In this section, we provide an example showing the necessity of the assumption  $a_N\in \mathbb{C}^r \otimes 1^\perp$ in Theorem~\ref{t:QEmat}(iii); in particular assuming $\langle a_N\rangle=0$ would not be sufficient to get Theorem~\ref{t:QEmat}(iii). The family of graphs considered below has also been used by McKenzie in \cite{zbMATH07514683}, who showed that in this family, quantum ergodicity fails for some observables. Our Theorem~\ref{t:QEmat}(iii) shows that quantum ergodicity holds nevertheless in the class of observables $\mathbb{C}^r\otimes 1^\perp$.

In the sequel, $d$ is even.
Let $g_1,\ldots,g_{d/2},g_1^{-1},g_{d/2}^{-1}$ be the set of generators of the free group on $d$ generators $F_d$, and let 
$$
A_N=\sum_{i=1}^{d/2} \rho_N(g_i)+\rho_N(g_i^{-1}).
$$
where $\rho_N$ is a random representation of $F_d$ in the symmetric group $S_N$. According to \cite[Theorem 3]{zbMATH07128141}, $\rho_N$ converges strongly (in probability) to the regular representation $\lambda$. 

We consider the Cartesian product $C_4\mathop\Box G_N$ where $G_N$ is the $d$-regular graph with adjacency matrix $A_N$, and $C_4$ is the cycle graph with $4$ elements. Its adjacency matrix  $A_{C_4}\otimes I_{G_N}+I_{C_4}\otimes A_N$ is of the form \eqref{e:ANmatcase} with $p_e=A_{C_4}$ the adjacency matrix of the $4$-cycle, and $p_g=I_{C_4}$ for $g\in S$ where $\Gamma$ is the free group generated by $S$, with $d/2$ elements and their inverses.

The strong convergence condition in Theorem~\ref{t:QEmat}(iii) holds, by our choice of $\rho_N$. The rapid decay condition is also satisfied in the $d$-regular tree (see \cite{zbMATH06859874}). Since the spectrum of the $d$-regular tree $\mathbb{T}_d$ is absolutely continuous, the spectral measure of $C_4\mathop\Box \mathbb{T}_d$, which is the convolution of the spectral measures of $C_4$ and $\mathbb{T}_d$, is absolutely continuous too. Finally, the condition \eqref{e:RDrenforceeweakmat} is satisfied, as a consequence of \cite[Theorem 1.3]{zbMATH07514683}.
Theorem~\ref{t:QEmat}(iii) thus applies and shows that quantum ergodicity holds in this family of graphs for the class of observables $\mathbb{C}^r \otimes 1^\perp$. 

However, the conclusion of Theorem~\ref{t:QEmat}(iii) fails in the family of graphs $C_4\mathop\Box G_n$ if one considers observables which are not in $\mathbb{C}^r \otimes 1^\perp$: as shown in \cite{zbMATH07514683}, there exists for any $n$ an orthonormal eigenbasis $(\psi_j)$ of the adjacency matrix of $C_4\mathop\square G_n$ and an observable $a_n$ with $\langle a_n\rangle=0$ and $\|a_n\|_{\infty}=1$ such that
\begin{equation}\label{e:valeur1/2}
\frac{1}{4n}\sum_{j=1}^{4n}|\langle \psi_{j},a_{n}\psi_{j}\rangle|^2=\frac12.
\end{equation}
Let us briefly describe the eigenbasis and the observables. The eigenvectors $\psi_j$ are taken in a ``separation of variables" form, namely $\psi_j(i,x)=\phi_1(i)\phi_2(x)$ for $(i,x)$ a vertex in the graph $C_4\mathop\Box G_n$, with $\phi_1, \phi_2$ normalized eigenvectors of the adjacency matrices of $C_4$ and $G_n$ respectively. The eigenvectors $\phi_1$ of $C_4$ are chosen as follows:
$$
\begin{tabular}{c|c|c|c|c}
eigenvalue & $2$ & $0$ & $0$ & $-2$ \\
  \hline
eigenvector & $(\frac12,\frac12,\frac12,\frac12)$ & $(\frac1{\sqrt2},0,-\frac1{\sqrt{2}},0)$ & $(0,\frac1{\sqrt2},0,-\frac1{\sqrt{2}})$ & $(\frac12,-\frac12,\frac12,-\frac12)$ 
    \end{tabular}
$$
Define the observable $a_n$ by its values at $(i,x)\in C_4\times G_n$: $a_{n}(i,x)=1$ for $i\in\{1,3\}$ and $a_{n}(i,x)=-1$ for $i\in\{2,4\}$. It is not difficult to check that \eqref{e:valeur1/2} holds, see \cite[Section 2]{zbMATH07514683}. 
\begin{remark}\label{r:depeigvect}
Theorem~\ref{t:QEmat}(iii) implies nevertheless that \eqref{e:correctiontermmatrix} holds. The above array of eigenvectors of $C_4$ shows that the quantity $\|\varphi_\alpha(i,\cdot)\|^2$ appearing in \eqref{e:correctiontermmatrix} depends on the eigenvector $\varphi_\alpha$. In fact, taking the $C_4$ eigenvectors $(\frac{1}{2},\frac{e^{ik}}{2},\frac{e^{2ik}}{2},\frac{e^{3ik}}{2})$, $k=0,1,2,3$ would not yield a violation of ergodicity.
\end{remark}

\appendix

\section{Strong convergence in distribution}
\label{s:strongcvindistrib}

In this appendix, we present some facts on the strong convergence in distribution and its relation to convergence in distribution.  First, let us make a remark on the terminology. In some works and in the recent survey \cite{van2025strong}, convergence in distribution is not required in the definition of strong convergence (only \eqref{e:strconvinitial} is required). We adopt in this paper the original terminology stemming from the work \cite{zbMATH06324768}, which is more adapted to our purposes and which has the advantage that strong convergence obviously implies weak convergence. It is important to recall here that if $\Gamma$ has the unique trace property, then \eqref{e:strconvinitial} automatically implies that $\rho_N$ converges in distribution to $\lambda$, see for instance \cite[Lemma 2.13]{van2025strong}.

In the following, $\rho_N : \Gamma \to \mathbb{U}_N$ is a unitary representation of dimension $N$ of $\Gamma$. 
We denote by $C^*_r(\Gamma)$ the reduced $C^*$-algebra of $\Gamma$, this is the closure of $\mathbb C[\Gamma]$ by the $\ell^\infty$-norm:
$$
\| a \| = \lim_{k \to \infty} \tau \left( (a a ^*)^k \right)^{\frac{1}{2k}}.
$$
It coincides with the operator norm of $\lambda(a)$ in $\ell^{2}(\Gamma)$. Here $\tau(a)=a_e$.

The next proposition states that one side of the strong convergence is a consequence of the convergence in distribution. We say that a vector space $H_N$ of $\mathbb{C}^N$ is invariant by $\rho_N$ if for all $g \in \Gamma$, $H_N$ is an invariant subspace of $\rho_N(g)$. For example, if $\rho_N$ is a permutation representation then $\mathrm{span}(1)$ is invariant by $\rho_N$. 
\begin{proposition}\label{p:lowerboundnorm}
Let $\rho_N$ be  a sequence of unitary representation of $\Gamma$ of dimension $N$ (along a subsequence of integers). Let $H_N$ be an invariant sequence of vector spaces of $\mathbb{C}^N$ of dimension $o(N)$. If $(\rho_N)$ converges in distribution to $\lambda$, then for any $p \in \mathbb{C}[\Gamma]$,
\begin{equation}\label{e:liminf}
\liminf_{N\rightarrow +\infty} \|\rho_N(p)_{|H^\perp_N}\|_{\rm op}\geq \|\lambda(p)\|_{\rm op}.
\end{equation}
\end{proposition}
\begin{proof} 
Set $P = \lambda(p)$ and $P_N = \rho_N(p)$. Since $\| A \|^2_{\rm op} = \| A A^* \|_{\rm op}$, we may assume without loss of generality that $p = p ^*$ (that is, for all $g \in \Gamma$, $\bar p_g = p_{g^{-1}}$).  In $C^*_r(\Gamma)$, the reduced norm coincides with the operator norm. Hence, for any $0 < \epsilon< 1$,  there exists an even integer $k \geq 2$ such that 
$$
\tau ( P^k )  = \langle \delta_e , P^k \delta_e \rangle \geq (1-\epsilon)^k \|P\|^k_{\rm op}.
$$    
Now, by assumption, 
\begin{equation}\label{eq:BSinfk}
\lim_{N \to \infty} \frac {1}{N} \Tr P_N^k = \tau ( P^k ) \geq (1-\epsilon)^k \|P\|^k_{\rm op}.
\end{equation}
Let $\Pi_N$ be the orthogonal projection onto $H^\perp_N$, $P_{N,1}  = (P_N)_{|H^\perp_N}= P_N \Pi_N$ and $P_{N,2} = (P_N)_{|H_N}= P_N (\mathrm{Id} - \Pi_N) $.
Using the invariance of $H_N$, the operators $P_{N,i}$, $i=1,2$, are self-adjoint and $P_{N,1} P_{N,2} = P_{N,2} P_{N,1} = 0$. We thus find
\begin{eqnarray*}
\Tr P_N^k  =  \Tr ( P_{N,1} + P_{N,2} ) ^k =  \Tr  P_{N,1}^k   + \Tr P_{N,2}^k.
\end{eqnarray*}
Note also that 
$$
\| P_{N,2} \|_{\rm op} \leq \| P_N \|_{\rm op} \leq \| p \|_{1}  \quad \hbox{ and } \quad \mathrm{rank}(P_{N,2} ) \leq d_N = \mathrm{dim}(H_N).
$$
Since $\Tr M \leq \mathrm{rank}(M) \| M\|_{\rm op}$, we deduce that 
$$
\left| \Tr P_N^k  - \Tr P_{N,1}^k  \right| \leq d_N \| p \|_{1} ^k.
$$
In particular, 
$$
\| P_{N,1} \|^k  \geq  \frac 1 N \Tr P_{N,1} ^k  \geq  \frac 1 N \Tr P_{N} ^k  -  \frac{d_N}{N} \| p \|_{1} ^k.
$$
Finally, from \eqref{eq:BSinfk} and the assumption $d_N = o (N)$, we get 
$$
\liminf_{N \to \infty} \| P_{N,1} \|_{\rm op} \geq (1-\epsilon) \|P\|_{\rm op}.
$$
Since $\epsilon$ is arbitrarily close to $0$, it concludes the proof. \end{proof}

The difficult part in establishing strong convergence is usually establishing the converse inequality to \eqref{e:liminf}, namely  for any $p \in \mathbb{C}[\Gamma]$,
\begin{equation}\label{e:limsup}
\limsup_{N\rightarrow +\infty} \|\rho_N(p)_{|H_N^\perp}\|_{\rm op}\leq \|\lambda(p)\|_{\rm op}.
\end{equation}
It even turns out that the upper bound \eqref{e:limsup} implies the lower bound \eqref{e:liminf} when the reduced $C^*$-algebra $C_r^*(\Gamma)$ is simple, see the first paragraphs of the proof of Theorem 1.1 in \cite[Section~5.2]{zbMATH07974824}. 

The next standard ``power trick" lemma is an apparent relaxation of \eqref{e:limsup}. We say that a function $E : \mathbb{R}_+ \to \mathbb{R}_+$ is subexponential if 
$$
\lim_{k \to \infty} E(k)^{1/k} = 1.
$$
\begin{lemma}\label{le:powertrick}
Let $\rho_N$ be  a sequence of unitary representation of $\Gamma$ of dimension $N$ (along a subsequence of integers).  Let $S$ be a finite generating set of $\Gamma$ and for $p \in \mathbb{C}[\Gamma]$, let ${\rm diam}_S(p)$ be the diameter of the support of $p$ in ${\rm Cay}(\Gamma,S)$. Assume that there exists a sub-exponential function $E$ such that for all $p \in \mathbb{C}[\Gamma]$ 
$$
\limsup_{N \to \infty} \| \rho_N (p) \|_{\rm op} \leq E ( \mathrm{diam}_S(p))  \| \lambda (p) \|_{\rm op}. 
$$
Then for all $p \in \mathbb{C}[\Gamma]$,
$$
\limsup_{N \to \infty} \| \rho_N (p) \|_{\rm op} \leq  \| \lambda (p) \|_{\rm op}.
$$
\end{lemma}

Remark that if $H_N$ is an invariant subspace of $\rho_N$, then the conclusion of Lemma~\ref{le:powertrick} holds if we replace $\rho_N(p)$ by $\rho_N(p)_{H_N^\perp}$ in both inequalities (consider the sub-representation ${\rho_N}_{H_N^\perp}$).

\begin{proof}[Proof of Lemma~\ref{le:powertrick}]
We can assume without loss of generality that $E$ is non-decreasing (replace $E(k)$ by $\max_{l \leq k} E(l)$). Since $\| A \|^2_{\rm op} = \| A A^* \|_{\rm op}$, we may assume without loss of generality that $\lambda(p)$ is self-adjoint and that for all $g \in \Gamma$, $\bar p_g = p_{g^{-1}}$. Set $d =\mathrm{diam}_S(p)$.  For any even integer $k \geq 2$, we thus have
$$
\lambda(p)^k  = \lambda (p^k) \quad \hbox{ and } \quad \rho_N(p)^k  = \rho_N (p^k)
$$
with $p^k \in \mathbb{C}[\Gamma]$ with $\mathrm{diam}_S(p^k) \leq k d$. 
Since $\| A^k \|_{\rm op} = \| A \|_{\rm op}^k$ for $A$ self-adjoint, the assumption implies that 
$$
\limsup_{N \to \infty} \| \rho_N (p) \|_{\rm op}^k  \leq E(kd) \| \lambda (p) \|_{\rm op}^k. 
$$
Equivalently,
$$
\limsup_{N \to \infty} \| \rho_N (p) \|_{\rm op} \leq E (  k d )^{1/k} \| \lambda (p) \|_{\rm op}. 
$$
Since $k$ can be arbitrarily large and for any $c >0$, $E(k)^{c/k}  \to 1$ as $k \to \infty$, the conclusion follows. 
\end{proof}

\section{Rapid decay property} \label{s:apprd}

The next standard lemma \cite[Lemma 1.5]{zbMATH03634902} is useful to estimate the norm of elements in $C^*_r(\Gamma)$. 

\begin{lemma}\label{l:rdopnorm}
Assume that $\Gamma$ has the RD property with constant $C_1>0$ as in Definition \ref{d:RD} and let $S$ be a finite set of generators of $\Gamma$. Then for all $C'_1> 2 C_1+1$, there exist a constant $C > 0$ such that for all $p \in \mathbb{C}[\Gamma]$, 
$$
\| \lambda(p) \|_{\rm op} \leq C \sqrt{\sum_{g} |p_g|^2 (|g|+1)^{C'_1}}.
$$
\end{lemma}
\begin{proof}
For integer $n\geq 0$, let $p_n = \sum_{ g : |g| = n} p_g g$. The RD property implies that, for some $C >0$,
$$
\| \lambda (p_n) \|_{\rm op} \leq C (n+1)^{C_1} \sqrt{ \sum_{g : |g| = n } |p_g|^2 } = C \sqrt{ \sum_{g : |g| = n } |p_g|^2 (|g|+1)^{2C_1}}.
$$
Therefore, from the subadditivity of norms, for any $\alpha >1/2$,
\begin{eqnarray*}
\| \lambda (p) \|_{\rm op} & = & \left\| \sum_n \lambda(p_n) \right\|_{\rm op} \\
& \leq & \sum_n \left\| \lambda(p_n) \right\|_{\rm op} \\
&\leq &  C \sum_n \left(\frac{n+1}{n+1}\right)^{\alpha} \sqrt{ \sum_{g : |g| = n } |p_g|^2 (|g|+1)^{2C_1}} \\
& \leq & C \sqrt{\sum_{n} \frac{1}{(n+1)^{2 \alpha} }} \sqrt { \sum_g |p_g|^2 (|g|+1)^{2(C_1+\alpha)}},
\end{eqnarray*}
where we have used Cauchy-Schwarz inequality at the last line.
\end{proof}

The next lemma extracted from \cite{magee2025strongasymptoticfreenesshaar} (see also \cite[Lemma 2.26]{van2025strong}) shows that to prove strong convergence, it is enough to consider Markovian  $p \in \mathbb{C}[\Gamma]$, that is $p \in\mathbb{R}[\Gamma] $ such that $\sum_g p_g =1$ and $p_g \geq 0$ for all $g \in \Gamma$. We start by introducing a weakening of the rapid decay property, satisfied by all groups with subexponential growth:

\begin{definition}\label{d:RD'} 
We say that $\Gamma$ has the weak rapid decay  (RD') property  if for some subexponential function $E $, for  any  $p \in \mathbb{C}[\Gamma]$,
$$
\| \lambda (p) \|_{\rm op} \leq E({\rm diam}_S(p) ) \| p \|_2
$$
where ${\rm diam}_S(p)$ is the minimum between one and the diameter of the support of $p$  in ${\rm Cay}(\Gamma,S)$.
\end{definition}

\begin{lemma}\label{le:markov}
Let $\rho_N$ be  a sequence of unitary representation of $\Gamma$ of dimension $N$ (along a subsequence of integers). Assume that $\Gamma$ has property RD' and that for all Markovian $p \in\mathbb{R}[\Gamma] $, \eqref{e:limsup} holds. Then  \eqref{e:limsup} also holds for all $p \in\mathbb{C}[\Gamma] $.
\end{lemma}
\begin{proof}
We may write $p \in \mathbb{C}[\Gamma]$ as $p = p_1 - p_2 + i p_3 - i p_4$ with for $ i \in \{1,\ldots,4\}$, $p_i \in \mathbb{R}[\Gamma]$ and for all $g \in \Gamma$ $(p_i)_g \geq 0$, 
\begin{equation}\label{eq:decomppm}
\sum_{i=1}^4 |( p_i)_g |^2  = |p_g|^2.
\end{equation}
From the assumption, we have 
$$
\limsup_{N \to \infty} \| \rho_N (p) \|_{\rm op} \leq \sum_{i=1}^4
\limsup_{N \to \infty} \| \rho_N (p_i) \|_{\rm op} \leq \sum_{i=1}^4 \| \lambda (p_i) \|_{\rm op}.$$
We next use the RD' property: for some non-decreasing subexponential function $E$,
$
\| \lambda (p_i) \|_{\rm op} \leq  E(\mathrm{diam}_S(p)) \|p_i\|_2.$
We get from Cauchy-Schwarz inequality and \eqref{eq:decomppm},
$$
\sum_{i=1}^4 \| \lambda (p_i) \|_{\rm op} \leq 2E(\mathrm{diam}_S(p))\|p \|_2.
$$
Since $\| p \|_2 = \| \lambda(p) \delta_e \| \leq \| \lambda (p) \|_{\rm op}$ the conclusion follows from Lemma~\ref{le:powertrick}.
\end{proof}

The next proposition asserts that for amenable groups of subexponential growth, convergence in distribution implies strong convergence. We refer also to Magee \cite[Section 2.1]{arXiv:2503.21619} for a closely related discussion.

\begin{proposition}\label{p:amenableBSstrong}
Let $\Gamma$ be a  group of subexponential growth. Let $\rho_N$ be  a sequence of unitary representations of $\Gamma$ of dimension $N$ (along a subsequence of integers). If $(\rho_N)$ converges in distribution toward $\lambda$ and $H_N$ is an invariant subspace of dimension $o(N)$, then  $({\rho_N}_{|H_N^\perp} )$ converges strongly in distribution toward $\lambda$.
\end{proposition}
\begin{proof}
Due to Proposition~\ref{p:lowerboundnorm}, we only need to show the inequality \eqref{e:limsup} for any $p \in \mathbb{C}[\Gamma]$.
To prove it, we may further assume by Lemma~\ref{le:markov} (recalling that $\Gamma$ has the RD' property since it has subexponential growth) that $p \in \mathbb{C}[\Gamma]$ is Markovian.
Let us check that $\|\lambda(p)\|_{\rm op} = \sum_g p_g = 1$. We have $\|\lambda(p)\|_{\rm op} \leq \|p \|_ 1= 1$. Conversely, since all groups of subexponential growth are amenable, there exists a F{\o}lner sequence $(F_k)$ for $\Gamma$. If $\chi_k$ denotes the $\ell^2$-normalized characteristic function of $F_k$, then for any $g\in \Gamma$,
$$
\langle \chi_k,\lambda(g) \chi_k\rangle=\frac{\#(F_k\cap gF_k)}{|F_k|} \underset{k\rightarrow +\infty}{\longrightarrow} 1.
$$

By linearity, we deduce that $\langle \chi_k,\lambda(p)\chi_k\rangle \rightarrow \sum_g p_g = 1$ as $k\rightarrow +\infty$. This proves $\|\lambda(p)\|_{\rm op}= 1$ and \eqref{e:limsup} follows since $\| \rho_N(p)_{|H_N^\perp} \|_{\rm op} \leq \| \rho_N(p)\|_{\rm op} \leq \| p \|_1 = 1$. 
\end{proof}

\section{Convergence rate}\label{a:convrate}

In this appendix, we discuss the explicit bounds that can be obtained from the proofs of the main theorems. 

We start by a quantified version of Lemma~\ref{cor:CMS2}, Corollary \ref{cor:CMS} below. The following proposition is a variant of \cite[Proposition 3.2]{zbMATH06534466} which is a consequence of the Chebyshev-Markov-Stieltjes inequality,  see \cite{zbMATH03219899,zbMATH03477793} for background on the moment problems (see also \cite[Proposition 5.6]{zbMATH05379079}, \cite[Section 5]{zbMATH06902684} for closely related statements). 

\begin{proposition}\label{prop:CMS}
  Let $\mu,\nu$ be two probability measures on $[-a,a] \subset \mathbb R$ and $n \in \mathbb N$. Assume that $\mu$ has a density bounded by $b$ on an interval $I \subset [-a,a]$ and that for all integers $k \le 2n$, 
$$
\int_\R \lambda^k d \mu = \int_\R \lambda^k d \nu.
$$
Then, for any $t \in I$ at distance at least $\epsilon > 0$ from $\mathbb R \setminus I$,  we have
$$
| \mu((-\infty,t]) - \nu((-\infty,t]) | \leq  \frac{a }{n} \left( 3b + \frac{a }{n\epsilon^{2}}\right).
$$
  \end{proposition}

\begin{proof}
From Equation (3.2), Equation (3.8) and below in \cite{zbMATH06534466}, we have 
\begin{equation}\label{eq:CMS}
| \mu((-\infty,t]) - \nu((-\infty,t]) | \leq \int_\R p (\lambda) d \mu (\lambda),
\end{equation}
where $p$ is any polynomial of degree at most $2(n-1)$ such that $p(\lambda) \geq 0$ on $[-a,a]$ and $p(t) \geq 1$ (the condition $p(\lambda) \geq 0$ on $\mathbb R$ is required in \cite{zbMATH06534466}, but non-negativity on $[-a,a]$ is sufficient since all roots of orthogonal polynomials are in the support of their associated measure). %\textbf{\textcolor{cyan}{and is $p$ the orthogonal polynomial associated to $\mu$ ?}}

As in \cite{zbMATH06534466}, we consider the following polynomial of degree $2(n-1)$ closely related to the Fejér kernel: 
$$
F_n (x) =  1 + \frac{2}{n}\sum_{k=1}^{n-1} \left( 1- \frac{k}{n}\right) (-1)^k T_{2k}(x),
$$
where $T_k$ is the Chebyshev polynomial of the first kind defined on $[-1,1]$ by $T_k( \cos \theta) = \cos (k \theta)$. For $x = \cos(\theta)$, we find
$$
2 n F_n(x) =  \left( \frac{\sin \left( n(\theta-\frac \pi 2)\right) }{\sin \left(\theta-\frac \pi 2\right)} \right)^2+ \left( \frac{\sin \left( n(\theta+\frac \pi 2)\right) }{\sin \left(\theta+\frac \pi 2\right)}   \right)^2  .
$$
We have $F_n(0) = n$, $F_n(x) \geq 0$ on $[-1,1]$ (and even on $\mathbb R$). Also, since $|\sin(\theta \pm \pi/2)| = |\cos(\theta)| = |x| $, we have for $x \in [-1,1]$, 
$$
F_n(x) \leq \frac{1}{n|x|^2}.
$$
Moreover, since $\int_{-1}^1 T_{2k}(x) dx = - 2 / (4k^2 -1)$, we can easily check that for all $n \geq 1$
$$
\int_{-1}^1 F_n(x) dx  =  2 + \frac{4}{n}\sum_{k=1}^{n-1} \left(1-\frac{k}{n}\right) \frac{(-1)^{k+1}}{4k^2-1} \leq 3.
$$
Coming back to \eqref{eq:CMS}, we take 
$$
p (\lambda)  = \frac{1}{n} F_n \left( \frac{ \lambda - t }{ a} \right).
$$
From what precedes, 
\begin{align*}
 \int p (\lambda) d \mu (\lambda) & \leq b \int_{t-\epsilon}^{t + \epsilon} p(\lambda)  d \lambda + \sup_{ \lambda \in [-a,a] \backslash [t-\epsilon,t+\epsilon]} |p(\lambda) | \\
 & \leq \frac{3ab}{n} + \frac{a^2 }{n^2 \epsilon^2}
\end{align*}
The conclusion follows. 
\end{proof}

The following corollary implies that in situations where we can quantify the convergence in distribution, we obtain a quantitative estimate on the number of eigenvalues in a given interval for self-adjoint $\rho_N(p)$. For $p = p^* \in \mathbb C[\Gamma]$ recall the definition of the spectral measure $\mu_p$ in \eqref{eq:defmup}.

\begin{corollary}\label{cor:CMS}
Let  $p \in \mathbb C[\Gamma]$ with $p = p^*$ and let $S \subset \Gamma$ be its support. For integer $n \geq 1$, let $\Bad_{S}(n) = \FIX_N^*( B_S(2n))$ where $B_S(n)$ is the ball of radius $n$ in $\mathrm{Cay}(\Gamma,S)$. If $I$ is an interval and $\mu_p$ admits a density on $I$ uniformly bounded by ${C_0}$, then for any interval $J \subset I$ at distance at least $\epsilon  > 0$ from $\mathbb R \setminus I$ we have
$$
 (N   - |\Bad_{S}(n)|) \left( \mu_p (J)  - \frac{ C }{n} \right) \leq |\Lambda_J| \leq N\left( \mu_p (J)  + \frac{C }{n} \right) +  |\Bad_{S}(n)|,
$$
where $C  = 2\|p \|_1 \left( 3 {C_0} +  \|p \|_1  / (n \epsilon^2)  \right)$  and $\Lambda_J$ is the set of eigenvalues of $\rho_N(p)$ in $J$ (counting multiplicities). 
\end{corollary}
%\textbf{\textcolor{cyan}{Not very important but I get $N(\mu_p(J)-\frac{N-|Bad|}{N}\frac{C}{n}-\frac{2|Bad|}{N})\le |\Lambda_J|\le N(\mu_p(J)+\frac{C}{n})+2|Bad|$. It's mostly from the bound $|\mu^x-\mu_p|\le 2$ for bad $x$.}}

\begin{proof} 
    Set $P_N = \rho_N(p) \in M_N(\mathbb C)$ and let $\mu_{P_N}$ be the empirical distribution of eigenvalues of $P_N$. We have $|\Lambda_J |= N \mu_{P_N}(J)$. Also, for $\psi \in \mathbb C^N$, recall the definition of $\mu_{P_N}^\psi$ in \eqref{e:muPpsi}. As explained in the proof of Lemma~\ref{cor:CMS2}, if $x \notin \Bad_S(n)$ then  for all integers $k \le 2n$, 
    $$
\int \lambda^k d \mu^{\delta_x}_{P_N}  = \int \lambda^k  d \mu_p .
$$
The support of $\mu_p$ and $\mu^{\delta_x}_{P_N}$ is contained in $[-a,a]$ with $ a = \|p \|_1 = \sum_g |p_g|$. 
Hence 
$$
|\Lambda_J|=N\mu_{P_N}(J)\geq \sum_{x\notin \Bad_S(n)}\mu_{P_N}^{\delta_x}(J)\geq (N-|\Bad_S(n)|)\left(\mu_p(J)-\frac{C}{n}\right)
$$
where we used \eqref{e:averageP} and then Proposition~\ref{prop:CMS}. For the right-hand side we use $\mu_{P_N}^{\delta_x}(J)\leq 1$ and Proposition~\ref{prop:CMS} to get
$$
|\Lambda_J|=N\mu_{P_N}(J)=\sum_{x\in \Bad_S(n)}\mu_{P_N}^{\delta_x}(J)+\sum_{x\notin \Bad_S(n)}\mu_{P_N}^{\delta_x}(J)\leq  N\left( \mu_p (J)  + \frac{C }{n} \right) +  |\Bad_{S}(n)|
$$
which concludes the proof.
\end{proof}

We may now discuss the rates of convergence that we can obtain from the proofs of Theorem~\ref{t:HNperp} and Theorem~\ref{t:QEstrongCV}. Similar arguments can be given for Theorem~\ref{t:asympdecorr}. For ease of discussion, we will restrict ourselves to the following quantified convergence in distribution of $\rho_N$ toward $\lambda$: for some $\alpha >0$, $\beta \in [0,1)$, $C > 0$
\begin{equation}\label{eq:ErN}
|\Bad_S(r_N)| \leq C N^\beta \quad \hbox{ with } \quad r_N = \lceil \alpha \ln N \rceil,   
\end{equation}
where we recall that $\Bad_S(r) = \mathrm{Fix}^*_N( B_S(2r))$. This situation occurs for example for congruence actions of linear groups (for example $\Gamma= \mathrm{SL}_d(\mathbb Z)$ and $\rho_{N_n} (g) = g \mod [n]$ acting on $\mathrm{SL}_d(\mathbb Z/n\mathbb Z)$) then \eqref{eq:ErN} holds for some $\alpha > 0$ and $|\Bad_S(r_N)| = 0$ (for $d=2$ and $n$ prime, see \cite[Eqn (2.6)]{zbMATH04219238}, the argument holds for any $d$ and $n$ non-prime).  Also, this situation is typical in random actions of free groups (see e.g. \cite[Appendix~D]{Cyril}), the argument could also be extended to random actions of free products of finite groups or surface groups, see \cite{zbMATH07931115} and references therein.

We will also assume for some $C > 0$, $C'_1 > 2C_1 +1$, where $C_1$ is as in \eqref{e:rddef}, and $\gamma \in (0,1]$, that if $z = \lambda +i \eta$, $\lambda \in I$, then
\begin{equation}\label{eq:AC+N}
\sum_{g\in\Gamma}\eta^2 |g|^{C'_1} |(\Im R^z)(e,g)|^4 \leq C \eta^{1+ \gamma}.
\end{equation}
For example, in the free group case, we have $\gamma =1$. In the setting of Theorem~\ref{t:QEstrongCV}, we assume for example the following quantitative strong convergence:
\begin{equation}\label{eq:SCN}
    \| \rho_N(q)_{|1^\perp} \| \leq 2 \| \lambda(q) \|, 
\end{equation}
for all $q \in \mathbb C[\Gamma]$ supported on $B_S(r_N)$ with $r_N$ as in \eqref{eq:ErN}. This condition is typical for random actions of the free group, see \cite[Theorem 3.9]{arXiv:2405.16026}.  We then obtain the following quantitative result.

\begin{theorem}\label{t:QEN}
Assume that $\Gamma$ has the RD property with constant $C_1$ and that $\rho_N$ satisfies \eqref{eq:ErN}-\eqref{eq:SCN}. Let $ p \in \mathbb C[\Gamma]$ such that $p = p^*$ and an interval $I_0$ such that \eqref{eq:AC}-\eqref{eq:AC+N} holds for $P = \lambda(p)$. Let $I$ be a closed interval in the interior of $I_0$ and  $(\varphi_\alpha)$ be an orthonormal eigenbasis for $P_N = \rho_N(p)$. Then, for any finite $T \subset \Gamma$, there exists $C >0$ such that for any $T$-local matrix $(K_N)$ in $M_N(\mathbb C)$ with $\|K_N\|_{1,\infty} \leq 1$, we  have
 for all $E_1, E_2 \subset I$,
$$
\frac{1}{|\Lambda_{J_{E_1}^{\eta_N}}|} \sum_{\alpha \in \Lambda_{J_{E_1}^{\eta_N}}} \sum_{\beta  \in \Lambda_{J_{E_2}^{\eta_N}}} \left|\langle \varphi_\beta , K_N \varphi_\alpha \rangle  - \langle \varphi_\beta,  \langle K_N \rangle  \varphi_\alpha \rangle \right|^2  \leq C \eta_N^{\frac{1 +\gamma}{2}},
$$
and 
$$
\frac{1}{|\Lambda_{I}|} \sum_{\alpha \in \Lambda_{I}}  \left|\langle \varphi_\alpha , K_N \varphi_\alpha \rangle  - \langle \varphi_\alpha,  \langle K_N \rangle  \varphi_\alpha \rangle \right|^2  \leq C \eta_N^{\frac{1 +\gamma}{2}},
$$
where $ \eta_N = C \ln \ln N  /  \ln N$.

Moreover, without assumption \eqref{eq:SCN}, for any $b>0$, there exist $C  >0$ and a vector space $H_N$ of dimension at most $N / (\ln N)^b$ such that the two above bounds hold with $\eta_N = C (\ln \ln N )^2 /  \ln N$ and all $K_N$ such that for all $g \in T$, $k_{N,g} \in H_N^\perp$ where $k_{N,g}$ is as in \eqref{eq:defKN}. 
\end{theorem}

\begin{proof}
We start with the first statement.  Let $Q = \lambda(q)$ as in \eqref{eq:QRD}, we find for some $C >0$, 
$$
\frac{\| Q \|}{\eta} \leq C \eta^{\frac{1 +\gamma}{2}}. 
$$ 
Also, by Corollary~\ref{cor:CMS} with $n:=r_N$ and \eqref{eq:AC}, assumption \eqref{eq:ErN} implies that for some $C >0$,  $1/ C \leq |\Lambda_J|/ (N |J|) \leq C$ for all intervals of $J \subset I$ of length $|J| \geq C / \ln N$.  If \eqref{eq:SCN} holds,  in the proof of Theorem~\ref{t:QEstrongCV} we take $\eta$ depending on $N$, namely
$$
\eta_N = \frac{a \ln \ln N  }{ \ln N}
$$
where $a > (2c)/\alpha$ with $c$ as in Lemma \ref{l:approxpolpos} and $\alpha$ as in \eqref{eq:ErN}. Let
$\epsilon_N = 1/( 2 \ln (\max(1/\eta_N,4))) \sim 1/(2 \ln \ln N)$, and observe that $n_N  =  \lceil  c / (\eta_N \epsilon_N) \rceil \sim (2 c/a) \ln N \leq r_N-1$ for all $N$ large enough. Then using $|\Bad_S(n_N+1)| \leq C N^\beta$ from \eqref{eq:ErN}, we readily deduce the requested bound from \eqref{eq:TrC} and Lemmas \ref{le:mixtoergo}-\ref{le:TraceQ}.

For the second statement, in the proof of Theorem~\ref{t:HNperp}, for $a>0$, we may take $$\eta_N = \frac{a (\ln \ln N) ^2 }{ \ln N},$$
$\epsilon_N = 1/( 2 \ln (\max(1/\eta_N,4))) \sim 1/(2 \ln \ln N)$ and $n_N = \lceil  c / (\eta_N \epsilon_N) \rceil \sim (2 c/a) \ln N /  \ln \ln N$ with $c >0$ as in Lemma~\ref{l:approxpolpos} and $\ell_N = \lfloor r_N / n_N \rfloor \sim (a \alpha /(2c)) \ln \ln N$.  Further, in the proof of Theorem~\ref{t:HNperp}, instead of \eqref{e:domHoNm}, given any sequence $(g_N)$ with $g_N\rightarrow +\infty$ as $N\rightarrow +\infty$, we may impose that the dimension of the vector space $H_N$ defined in the proof of Theorem~\ref{t:HNperp} verifies $\dim(H_N)\leq \frac{N}{f_Ng_N}$, where $f_N$ is given by \eqref{e:afNbas}, and thus $f_N\leq N / (CN^\beta+N4^{-\ell_N})$. As a consequence of this upper bound on $f_N$, for any $b>0$, we may find $a >0$ large enough, and a sequence $(g_N)$, so that $\dim(H_N) \leq N / (\ln N)^b$. The conclusion follows similarly. Note that it is also possible to pick a larger $\eta_N \to 0$ to obtain a vector space $H_N$ of smaller dimension.
\end{proof}

\bibliographystyle{abbrv}
\bibliography{biblio}

\end{document}